\documentclass[letterpaper,10pt]{amsart}
\usepackage{amsmath,amssymb,amsxtra,amsthm, amstext, amscd,amsfonts,fancyhdr, hyperref,enumerate}
\usepackage[OT2,T1]{fontenc}
\DeclareSymbolFont{cyrletters}{OT2}{wncyr}{m}{n}
\DeclareMathSymbol{\Sha}{\mathalpha}{cyrletters}{"58}

\newcommand{\bA}{{\mathbb{A}}}

\newcommand{\bC}{{\mathbb{C}}}
\newcommand{\bD}{{\mathbb{D}}}

\newcommand{\bN}{{\mathbb{N}}}

\newcommand{\bQ}{{\mathbb{Q}}}
\newcommand{\bR}{{\mathbb{R}}}

\newcommand{\bZ}{{\mathbb{Z}}}

\newcommand{\Bm}{{\mathbf{m}}}
\newcommand{\Bn}{{\mathbf{n}}}

\newcommand{\Bu}{{\mathbf{u}}}
\newcommand{\Bv}{{\mathbf{v}}}
\newcommand{\Bw}{{\mathbf{w}}}
\newcommand{\Bx}{{\mathbf{x}}}
\newcommand{\By}{{\mathbf{y}}}
\newcommand{\Bz}{{\mathbf{z}}}

\newcommand{\A}{{\mathcal{A}}}
\newcommand{\B}{{\mathcal{B}}}
\newcommand{\C}{{\mathcal{C}}}
\newcommand{\D}{{\mathcal{D}}}
\newcommand{\E}{{\mathcal{E}}}
\newcommand{\F}{{\mathcal{F}}}
\newcommand{\G}{{\mathcal{G}}}
\newcommand{\I}{{\mathcal{I}}}
\newcommand{\J}{{\mathcal{J}}}
\newcommand{\K}{{\mathcal{K}}}  
\renewcommand{\L}{{\mathcal{L}}}

\newcommand{\N}{{\mathcal{N}}}
\renewcommand{\O}{{\mathcal{O}}}
\renewcommand{\P}{{\mathcal{P}}}
\newcommand{\Q}{{\mathcal{Q}}}
\newcommand{\R}{{\mathcal{R}}}
\renewcommand{\S}{{\mathcal{S}}}

\newcommand{\UUU}{{\mathcal{U}}}

\newcommand{\Y}{{\mathcal{Y}}}
\newcommand{\Z}{{\mathcal{Z}}}

\newcommand{\fp}{\mathfrak{p}}

\newcommand{\fa}{\mathfrak{a}}

\newcommand{\fm}{\mathfrak{m}}
\newcommand{\fn}{\mathfrak{n}}
\newcommand{\fd}{\mathfrak{d}}
\newcommand{\fA}{\mathfrak{A}}
\newcommand{\fB}{\mathfrak{B}}

\newcommand{\fh}{\mathfrak{h}}
\newcommand{\fk}{\mathfrak{k}} 
\newcommand{\fJ}{\mathfrak{J}}
\newcommand{\fI}{\mathfrak{I}} 
\newcommand{\fM}{\mathfrak{M}}  
\newcommand{\fQ}{\mathfrak{Q}} 
\newcommand{\fR}{\mathfrak{R}}

\newcommand{\Tr}{\operatorname{Tr}}

\newcommand{\Cl}{\operatorname{Cl}}

\newcommand{\ff}{\mathfrak{f}}

\newcommand{\fS}{\mathfrak{S}}
\newcommand{\fC}{\mathfrak{C}}
\newcommand{\fF}{\mathfrak{F}}

\newcommand{\fZ}{\mathfrak{Z}}
\newcommand{\fK}{\mathfrak{K}}  
\newcommand{\fD}{\mathfrak{D}} 

\newcommand{\ep}{\varepsilon}

\newcommand{\ol}{\overline}

\newcommand{\upchi}{{\raise.35ex\hbox{$\chi$}}}

\newcommand{\SL}{\operatorname{SL}}

\makeatletter
\@namedef{subjclassname@2010}{%
	\textup{2010} Mathematics Subject Classification}
\makeatother

\newtheorem{theorem}{Theorem}[section]
\newtheorem{corollary}[theorem]{Corollary}
\newtheorem{proposition}[theorem]{Proposition}
\newtheorem{lemma}[theorem]{Lemma}

\theoremstyle{definition}

\newtheorem{remark}[theorem]{Remark}

\numberwithin{equation}{section}

\begin{document}
	
	\title{Prime values of $f(a,b^2)$ and $f(a,p^2)$, $f$ quadratic}
	
	\author{Stanley Yao Xiao}
	\dedicatory{Dedicated to the occasion of John Friedlander's 80th birthday.}
	\address{Department of Mathematics and Statistics \\
		University of Northern British Columbia \\
		3333 University Way \\
		Prince George, British Columbia, Canada \\  V2N 4Z9}
	\email{StanleyYao.Xiao@unbc.ca}
	\indent
	
	
	\begin{abstract} We prove an asymptotic formula for primes of the shape $f(a,b^2)$ with $a,b$ integers and of the shape $f(a,p^2)$ with $p$ prime.  Here $f$ is a binary quadratic form with integer coefficients, irreducible over $\bQ$ and has no local obstructions. This refines the seminal work of Friedlander and Iwaniec on primes of the form $x^2 + y^4$ and Heath-Brown and Li on primes of the form $a^2 + p^4$, as well as earlier work of the author with Lam and Schindler on primes of the form $f(a,p)$ with $f$ a positive definite form. 
	\end{abstract}
	
	\maketitle

	\section{Introduction}
	\label{Intro}
	
	Questions concerning prime values taken by polynomials are among the oldest and most interesting  in number theory. For example, the question of whether or not there are infinitely many twin primes can be phrased as the question whether the linear polynomial $x - y - 2$ has infinitely many zeroes $(x,y)$ with $x,y$ both prime numbers. Investigating prime values of polynomials has driven much of the research in additive and analytic number theory in the last two centuries. \\
	
	Two of the most stunning results in this area are the seminal works of Friedlander and Iwaniec \cite{FI1} and Heath-Brown \cite{HB1}, demonstrating that the polynomials $x^2 + y^4$ and $x^3 + 2y^3$ respectively take on infinitely many prime values. In particular, Friedlander and Iwaniec obtained the beautiful asymptotic formula 
	\begin{equation} \label{FIMT} \mathop{\sum \sum}_{a^2 + b^4 \leq X} \Lambda(a^2 + b^4) = \frac{2 \Gamma(1/4)^2}{3\pi \sqrt{2\pi}} X^{\frac{3}{4}} \left(1 + O \left(\frac{\log \log X}{\log X} \right) \right)
	\end{equation}
where $\Lambda(\cdot)$ is the von Mangoldt function and $\Gamma$ is the Gamma function. \\ 

	 Heath-Brown's result on $x^3 + 2y^3$ was quickly generalized by Heath-Brown and Moroz in \cite{HBM}, which demonstrated that any admissible binary cubic form takes on infinitely many prime values. More recently, X.~Li has proved that the cubic form $x^3 + 2y^3$ takes on infinitely many prime values with $y$ restricted to a short interval \cite{Li}. One also notes the stunning work of J.~Maynard on representation of primes by incomplete norm forms, a substantial generalization of Heath-Brown's work \cite{May}. \\ 
	
	In another direction, one might ask whether \emph{reducible polynomials} take on infinitely many \emph{semi-prime values}, with the order of the semi-prime being equal to the number of irreducible factors. A first example of this type of result is due to Fouvry and Iwaniec \cite{FouI}, who showed that the binary cubic form $y(x^2 + y^2)$ takes on infinitely many values with exactly two prime factors. This work paved the way for the later work of Friedlander and Iwaniec \cite{FI1}. Heath-Brown and Li then combined the result of Fouvry and Iwaniec and Friedlander and Iwaniec in \cite{HBL}, showing that the polynomial $y(x^2 + y^4)$ takes on infinitely many values with exactly two prime factors. In particular they obtained the asymptotic formula 
	\begin{equation} \label{HBLMT} \mathop{\sum \sum}_{a^2 + b^4 \leq X} \lambda(b) \lambda(a^2 + b^4) = \frac{2 \Gamma(1/4)^2}{3\pi \sqrt{2\pi}} \frac{X^{\frac{3}{4}}}{(\log X)^2 }\left(1 + O_\ep \left(\frac{1}{(\log X)^{1-\ep}} \right) \right),
	\end{equation}
where $\lambda$ is the prime indicator function. \\ 

	 Lam, Schindler and the author generalized the work of Fouvry and Iwaniec in another direction, proving that for any admissible positive definite binary quadratic form $f$ the cubic form $y f(x,y)$ takes on infinitely many values with exactly two prime factors. Our main result implies:
	 \begin{equation} \label{LSXMT} \mathop{ \sum \sum}_{f(m, \ell) \leq X} \Lambda(\ell) \Lambda(f(m, \ell)) = \nu_f \fS_f^\prime X + O_A(X (\log X))^{-A}),
	 \end{equation}
 where $\nu_f$ is a product of local densities given by 
 \begin{equation} \label{nuf def} \nu_f = \prod_{p \nmid \Delta(f)} \left(1 - \frac{\rho_f(p)}{p} \right) \left(1 - \frac{1}{p} \right)^{-1} \prod_{p | \Delta(f)} \left(1 - \frac{1}{p} \right)^{-1}, 
 \end{equation}
$\fS_f^\prime$ is given by (\ref{Sf prime}), and $\rho_f(m) = \# \{x \pmod{m} : f(x,1) \equiv 0 \pmod{m}\}$. \\
	
	Despite the passage of more than two decades, a generalization akin to that of Heath-Brown and Moroz \cite{HBM} has yet to materialize for the main result of \cite{FI1}, despite the authors of that paper claiming that such a result should be readily obtainable from their arguments \footnote{``We expect, but did not check, that the methods carry over to the prime values
		of $\phi(a, b^2)$ for $\phi$ a quite general binary quadratic form." \cite{FI1}, p. 947.}. That is, there has yet to be a proof that $f(x,y^2)$ takes on infinitely many prime values for any binary quadratic form $f$ other than $f(x,y) = x^2 + y^2$. \\
	
	In this paper, we simultaneously generalize the main results of Friedlander and Iwaniec \cite{FI1} and Heath-Brown and Li \cite{HBL}. If $f$ is definite put 
	\[\fS_f = \text{Area} \{(x, y) \in \bR^2 : f(x,y^2) \leq 1 \} \]
	and for $f$ indefinite we define 
	\[\fS_f = \lim_{X \rightarrow \infty} \frac{\text{Area}\{(x,y) \in \bR^2 : 0 < f(x,y^2) < X, 0 < y \leq X^{1/4}\}}{X^{3/4}}.\]
	Our first main result is: 
	
	\begin{theorem}\label{MT0} Let $f(x,y) = f_2 x^2 + f_1 xy + f_0 y^2 \in \bZ[x,y]$ be an irreducible and primitive binary quadratic form, with the property that $f(x,1) \not \equiv x(x+1) \pmod{2}$. Then for $f$ positive definite we have
		\begin{equation} \label{pos def eq} \sum_{\substack{m, \ell \in \bZ \\ f(m, \ell^2) \leq X }} \lambda\left(f(m, \ell^2) \right) =  \frac{\nu_f \fS_f X^{3/4}}{\log X} \left(1 +  O \left(\frac{ \log \log X}{\log X} \right) \right)
		\end{equation} 
		and for $f$ indefinite we have 
		\begin{equation} \label{indef eq} \sum_{\substack{m, \ell \in \bZ \\ 0 < f(m, \ell^2) \leq X \\ 0 < \ell \leq X^{1/4} }}  \lambda\left(f(m, \ell^2) \right) =  \frac{\nu_f \fS_f X^{3/4}}{\log X} \left(1 + O \left(\frac{\log \log X}{\log X} \right) \right).
		\end{equation} 
	\end{theorem} 

The condition that $f(x,1) \not \equiv x(x+1) \pmod{2}$ is necessary, as otherwise $f(x,k)$ is divisible by $2$ whenever $k$ is odd, precluding the possibility that it could be prime unless $k = 4$. Note that Theorem \ref{MT0} recovers Theorem 1 of \cite{FI1} upon setting $f(x,y) = x^2 + y^2$. It also implies, for example, that the polynomial $x^2 + xy^2 + y^4$ represents infinitely many primes. \\ \\
We note that the choice of cutting off the $y$-variable at $X^{1/4}$ is somewhat arbitrary, and is mostly done for aesthetic reasons. \\ \\
	Our proof, which follows along the lines of \cite{HBL}, yields the following refinement which is analogous to Theorem 1 of \cite{HBL} or (\ref{HBLMT}): 
	
	\begin{theorem} \label{MT} Let $f(x,y) = f_2 x^2 + f_1 xy + f_0 y^2 \in \bZ[x,y]$ be an irreducible and primitive binary quadratic form, with the property that $f(x,1) \not \equiv x(x+1) \pmod{2}$. Then for $f$ positive definite we have
		\begin{equation} \label{pos def eq} \sum_{\substack{m, \ell \in \bZ \\ 0 < f(m, \ell^2) \leq X }} \lambda(\ell) \lambda\left(f(m, \ell^2) \right) =  \frac{\nu_f \fS_f X^{3/4}}{(\log X)^2} \left(1  + O \left(\frac{\log \log X}{\log X} \right) \right)
		\end{equation} 
		and for $f$ indefinite we have 
		\begin{equation} \label{indef eq} \sum_{\substack{m, \ell \in \bZ \\ 0 < f(m, \ell^2) \leq X \\ 0 < \ell \leq X^{1/4} }} \lambda(\ell) \lambda\left(f(m, \ell^2) \right) =  \frac{\nu_f \fS_f X^{3/4}}{(\log X)^2} \left(1 + O \left(\frac{\log \log X}{\log X} \right) \right).
		\end{equation} 
	\end{theorem}

	Theorem \ref{MT} implies that there are infinitely many integers $x$ and primes $p$ for which $f(x,p^2)$ is prime. Note that both Theorem \ref{MT0} and Theorem \ref{MT} apply to \emph{indefinite} as well as definite forms. We further note that the error term in Theorem \ref{MT} is slightly better than in (\ref{HBLMT}), due to choosing a slightly different sieving parameter. \\
	
	In \cite{HBM}, the key new insight is that the arithmetic of \emph{ideal numbers} allows one to connect the multiplicative structure on the set of ideals of a ring of integers, which has unique factorization, to the arithmetic of the elements in a ring of integers which need not have unique factorization. This breaks a key barrier in \cite{HB1} where the fact that $\bZ[\sqrt[3]{2}]$ is a unique factorization domain is used in a crucial manner. Moreover, \cite{HBM} shows that the analytic estimates obtained by Heath-Brown in \cite{HB1} can be applied with relatively few changes in the general setting. \\
	
	In \cite{LSX} we essentially pursued the same approach, although we did not state things in terms of ideal numbers but rather worked out an explicit composition law for binary quadratic forms, in the spirit of Gauss and Dirichlet. In the present work we have decided to adopt the approach of Heath-Brown and Moroz and use ideal numbers, as this is a more elegant and general approach. \\
	
	In order to prove Theorems \ref{MT0} and \ref{MT} we adopt an approach introduced by Heath-Brown in \cite{HB1}, which we call Heath-Brown's \emph{comparison sieve}. This involves applying the same sieve procedure to two comparable sequences $\A = (a_n)$ and $\B = (b_n)$, producing cancellation at appropriate junctures. This was used again by Heath-Brown and Li in \cite{HBL} for the proof of their result. \\ \\
	In order to prove Theorem \ref{MT0} we choose our sequence $\B$ to simply be the set of prime ideals of the ring of integers $\O_K$. The sequence $\B$ used by Heath-Brown and Li is exactly the sequence studied by Fouvry and Iwaniec in \cite{FouI}. For positive definte forms $f$ we may then apply the result in \cite{LSX}, and for indefinite forms we will need to prove an extension of our main result with Lam and Schindler in \cite{LSX}, which gives an asymptotic formula for the number of representation of primes by $f(x,p)$, with $p$ prime.\\
	
	For $f$ positive definite put
	\begin{equation} \label{Sf prime} \fS_f^\prime = \text{Area} \{(x,y) \in \bR^2 : f(x,y) \leq 1\}\end{equation}
	and for $f$ indefinite put
	\[\fS_f^\prime = \lim_{X \rightarrow \infty} \frac{\text{Area} \{(x,y) \in \bR^2 : 0 < f(x,y) < X, 0 < y < X^{1/2}\}}{X}.\]
	Then: 
	
	\begin{theorem} \label{MT2}Let $f(x,y) = f_2 x^2 + f_1 xy + f_0 y^2 \in \bZ[x,y]$ be an irreducible and primitive binary quadratic form, with the property that $f(x,1) \not \equiv x(x+1) \pmod{2}$. Then for $f$ positive definite we have
		\begin{equation} \label{pos def eq} \sum_{\substack{m, \ell \in \bZ \\ 0 < f(m, \ell) \leq X }} \Lambda(\ell) \Lambda\left(f(m, \ell) \right) =  \nu_f \fS_f^\prime X + O_A \left( \frac{X}{(\log X)^A}  \right)
		\end{equation} 
		and for $f$ indefinite we have 
		\begin{equation} \label{indef eq} \sum_{\substack{m, \ell \in \bZ \\ 0 < f(m, \ell) \leq X \\ 0 < \ell \leq X^{1/2} }} \Lambda(\ell) \Lambda\left(f(m, \ell) \right) =  \nu_f \fS_f^\prime X + O_A \left(\frac{X}{(\log X)^A} \right).
		\end{equation} 
		Here $\nu_f$ is as in Theorem \ref{MT} and $\fS_f^\prime$ is as in (\ref{Sf prime}). 
	\end{theorem}

Note that Theorem \ref{MT2} is stated with the von Mangoldt function rather than $\lambda$ to emphasize that a substantially better error term, giving an arbitrary log-power saving, is possible. \\ 
	
	Theorem \ref{MT2} implies the following, which completely settles \emph{Schinzel's hypothesis} for binary cubic forms:
	
	\begin{corollary} \label{red cube} Let $F(x,y)$ be a reducible binary cubic form of the shape $F(x,y) = L(x,y) Q(x,y)$, where $Q$ is an irreducible binary quadratic form. Then there are infinitely many pairs of integers $x,y$ such that $F(x,y)$ is divisible by exactly two primes. 
	\end{corollary}
	
	Corollary \ref{red cube} is the final case of \emph{Schinzel's hypothesis} in the setting of binary cubic forms. The hardest case, that of irreducible binary cubic forms, is settled by the work of Heath-Brown \cite{HB1} and Heath-Brown and Moroz in \cite{HBM}. The case with $F$ reducible with a positive definite quadratic factor is settled by the author's joint work with Lam and Schindler in \cite{LSX}. The totally reducible case was settled by B.~J.~Green's work on 3-term arithmetic progressions in the primes \cite{Green}. \\
	
	The main contribution of this paper is to insert composition laws involving ideal numbers of a fixed quadratic field into the analytic estimates of Friedlander and Iwaniec in \cite{FI1} and Heath-Brown and Li in \cite{HBL}. There are several places where this is quite delicate, which requires us to redo certain parts of \cite{FI1} and \cite{HBL}. Specifically, the so-called Jacobi-Kubota symbol introduced by Friedlander and Iwaniec in \cite{FI1} does not have an obvious analogue in the general setting, and we need to introduce substitutes. We give a rough explanation of this in the following subsection.   
	
	\subsection{Sketch of the main ideas for the proof of Theorems \ref{MT0} and \ref{MT}}
	
	The framework behind the proof of Theorems \ref{MT0} and \ref{MT} is the so-called \emph{asymptotic sieve}, originally developed by Bombieri and laid out in detail by Friedlander and Iwaneic to detect primes in \cite{FI2}. Their work gives us a way to estimate the sum 
	\[S(X) = \sum_{n \leq X} a_n \Lambda(n) \]
	with $\Lambda(\cdot)$ the von Mangoldt function and $a_n$ a non-negative sequence supported on the natural numbers, provided that a suitable level of distribution can be obtained for the sequence $\A = (a_n)$ and that a \emph{bilinear sum estimate} can be obtained for a sum of the shape
	\[\sum_m \alpha(m) \sum_{\substack{N \leq n < 2N \\ mn \leq X }} \beta(n) a_{mn}.  \]
	In practice, the proof of results giving asymptotic formulae for primes involve obtaining an acceptable bilinear sum estimate as expected. \\ \\
	That the appropriate level of distribution can be obtained is a consequence of the main result of \cite{BBDT} and \cite{FI3}; this aspect was exploited by the author, Lam, and Schindler in \cite{LSX}. Therefore, the remaining difficulty in proving Theorems\ref{MT0} and \ref{MT} involves dealing with bilinear sum estimates. \\ \\
	This is where the story takes a somewhat unexpected path: it turns out that the type of bilinear sums that come up in the case with $K$ a general quadratic field depend both on the structure of the class group of $\O_K$ and the existence of units of infinite order. This aspect will be explained in detail in Section \ref{algebra}. In fact the number of bilinear sums is equal to the class number $h(K)$ of $\O_K$. \\ \\
	With each such bilinear sum, we are free to subdivide it into sums over small regions as in \cite{FI1} and \cite{HBL}. One then has to estimate certain `main' terms and bound the error terms. Fortuitously, the methods used to bound the error terms in \cite{FI1} and \cite{HBL} do not depend much on the arithmetic of $\O_K$, and in fact it suffices to treat the problem as one over $\bZ^2$. That is, this part of the argument largely depends only on the structure of $\O_K$ or a set of ideal numbers as a $\bZ$-module. \\ \\
	It is only in the estimation of various main terms where the arithmetic of $\O_K$, or more precisely the arithmetic of ideal numbers of $\O_K$, becomes crucial. In particular, we are required to generalize certain results in \cite{FI1} which are used by \cite{HBL} to estimate the relevant bilinear sums. This turns out to be a delicate task, as there are three properties of $\bZ[i]$ used by Friedlander and Iwaniec that turn out to be very convenient, and no other ring of quadratic integers possess all of them: 
	\begin{itemize}
		\item The class number of $\bZ[i]$ is $1$; 
		\item The norm of $\bZ[i]$ is the same as the Euclidean norm on $\bC$; and 
		\item The odd rational primes that split in $\bZ[i]$ are precisely those that are congruent to $1 \pmod{4}$. 
	\end{itemize}
	If a ring of integers $\O_K$ fails to have any of the above properties (in general, it is only possible for $\O_K$ to have class number one; the other two properties essentially characterize $\bZ[i]$), then certain objects introduced in \cite{FI1}, for example the so-called Jacobi-Kubota symbol, will not have analogues with equally nice properties. In general one needs to introduce a \emph{family} of Jacobi-Kubota symbols; one symbol for each element of the class group \emph{and a choice of basis} for the class of ideal numbers. This makes the symbols non-canonical but the choice to choose the basis freely will be used to our advantage. Then the analytic estimates obtained by \cite{FI1} can be applied to each symbol in the family to produce the required cancellation. This allows us to then apply these estimates to the argument of \cite{HBL}, leading to the proof of Theorem \ref{MT}. \\ \\
	We remark that two key results in \cite{HBL}, namely Corollaries 1 and 2 which are a refinement of the Barban–
	Davenport–Halberstam theorem and a Siegel-Walfisz type estimate respectively, are not explicitly invoked here. This is because these two results are used in \cite{HBL} to prove their Proposition 6 which, surprisingly, can be applied more or less without change in our case. 
	
	\subsection{Organization of the paper} 
	
	In Section \ref{sieve} we review Friedlander and Iwaniec's asymptotic sieve for primes \cite{FI2}, which gives us our main framework to produce asymptotic formulae involving primes. In Section \ref{HB sieve} we discuss our approach to implementing the asymptotic sieve for primes, in the manner introduced by Heath-Brown in \cite{HB1} which we dub \emph{Heath-Brown's comparison sieve}, also used by Heath-Brown and Moroz in \cite{HBM} and Heath-Brown and Li in \cite{HBL}. In Section \ref{algebra} we introduce the necessary algebraic number theory involving the arithmetic ideal numbers, necessary to establish the framework needed to apply the analytic estimates in \cite{FI1} and \cite{HBL}. In Section \ref{Type I} we establish the needed level of distribution or Type I estimates. In Section \ref{piB bisum} we will prove the necessary bilinear sum estimates to obtain the analogue of the main theorem of \cite{LSX} in the indefinite case, which for us is needed to apply Heath-Brown's comparison sieve in the indefinite case. In Section \ref{piApiB} we establish the preliminary steps to proving our two key technical propositions, being Propositions \ref{HBL prop6} and \ref{HBL prop7}, which are analogues of Heath-Brown and Li's Propositions 6 and 7 in \cite{HBL}. In Section \ref{HBL sec7} we prove Proposition \ref{HBL prop6}, the proof being identical to that of \cite{HBL} except we avoid the language of Gaussian integers. In Sections \ref{HBL sec10} and \ref{HBL sec11} we modify Heath-Brown and Li's proof of their Proposition 7 in the setting of a general quadratic field $K$, thereby proving our Proposition \ref{HBL prop7}, which then completes the proof of Theorem \ref{MT}, conditioned on certain character sum estimates that they imported from \cite{FI1}. Finally, in Section \ref{char sums} we introduce the analogues of Friedlander and Iwaniec's notion of \emph{Jacobi-Kubota symbols} in the setting of a general quadratic field, as well as the analogue of their symbol $[\cdot]$ which in some sense measures the ``spin" of an ideal in $\bZ[i]$, which allows us to prove versions of their Proposition 23.1 and Theorem $\psi$ which are needed by Heath-Brown and Li. This may be of independent interest.
	
	\subsection*{Notation} Throughout, we fix our binary quadratic form
	\[f(x,y) = f_2 x^2 + f_1 xy + f_0y^2 \in \bZ[x,y]\]
	which satisfies the hypothesis that for all primes $p$ there exist integers $x_p, y_p$ such that $p \nmid f(x_p, y_p)$, and $f(x,1) \not \equiv x(x+1) \pmod{2}$. We will use both the Landau and Vinogradov notation $\ll$ and $O \left( \cdot \right)$. 
	
	\subsection*{Acknowledgements} This paper and the author owes an incalculable debt of gratitude to John Friedlander, whose encouragement and guidance made this paper possible. The author also thanks D.~R.~Heath-Brown whose work on prime number theory is an inspiration for the present work, D.~Schindler and J.~Maynard for helpful discussions, to C.~L.~Stewart for a careful reading of an earlier version of this paper and for providing instrumental advice, and to S.~Yamagishi whose collaboration and friendship was instrumental in the author's pursuit of prime number theory. 
	
	\section{Asymptotic sieve for primes}
	\label{sieve}
	
	Since our goal is to prove an asymptotic formula involving primes, the most straightforward way to achieve this is to apply Bombieri's asymptotic sieve, refined by Friedlander and Iwaniec in \cite{FI2} to detect primes. We denote by $\A = (a_n)$ a sequence of non-negative real numbers. We are concerned with the sum
	\begin{equation} \label{p-sum} S(X) = \sum_{n \leq X} a_n \Lambda(n) \sim \sum_{p \leq X} a_p \log p.
	\end{equation}
	As usual in sieve theory we introduce the related quantity
	\begin{equation} \label{t-sum} A(X) = \sum_{n \leq X} a_n.
	\end{equation}
	The main result of \cite{FI2} says that, if the sequence $(a_n)$ satisfies certain favourable conditions, then $S(X)$ satisfies an asymptotic formula in terms of $A(X)$. To state these conditions, we first suppose that $A(X)$ satisfies the bounds
	\begin{equation} \label{asp1}  A(X) \gg \max\left\{ A \left(\sqrt{X} \right) (\log X)^2, X^{1/3} \left(\sum_{n \leq X} a_n^2 \right)^{1/2} \right\}.
	\end{equation} 
	For each $d \in \bN$ we suppose 
	\begin{equation} \label{cong-sum} A_d(X) = \sum_{\substack{n \leq X \\ n \equiv 0 \pmod{d}}}
		a_n = g(d) A(X) + r_d(X),
	\end{equation}
	where $g : \bN \rightarrow \bR$ is a multiplicative function satisfying 
	\begin{equation} \label{asp2} 0 \leq g(p^2) \leq g(p) < 1 \text{ and }  g(p^j) \ll p^{-j} \text{ for } j = 1,2
	\end{equation} 
	for all primes $p$. We shall further assume a Mertens' type property for $g$: that is, we assume there exists a real number $c$ such that 
	\begin{equation} \label{asp3} \sum_{p \leq Y} g(p) = \log \log Y + c + O \left((\log Y)^{-10}\right),
	\end{equation}
	the implied constant depending only on $g$. Next we shall assume that uniformly for $d \leq X^{1/3}$ we have
	\begin{equation} \label{AdX crude1} A_d(X) \ll d^{-1} \tau(d)^8 A(X).
	\end{equation}
	We remark that the exponent $8$ in the divisor function in (\ref{AdX crude1}) is immaterial: in our application we can replace it with the exponent $1$. Next we shall assume a \emph{Type I estimate} for the remainder terms $r_d(X)$, namely
	\begin{equation} \label{rdx type1} \sideset{}{^3} \sum_{d \leq DL^{-2}} |r_d(t)| \leq A(X) L^{-2},
	\end{equation} 
	where the superscript $3$ in (\ref{rdx type1}) refers to summation over cube-free integers and we take $L = (\log X)^{2^{24}}$. Again, the exponent is to be interpreted as some sufficiently large absolute constant. \\ 
	
	The true bottleneck in the asymptotic sieve for primes in practice is a \emph{bilinear sum estimate} or a \emph{Type II estimate}. Indeed, we shall require a bound of the form
	\begin{equation} \label{rdx type2} \sum_m \left \lvert \sum_{\substack{N < n \leq 2N \\ mn \leq X \\ \gcd(n, m \Pi) = 1}} \beta(n) a_{mn} \right \rvert \leq A(X) (\log X)^{-2^{26}} 
	\end{equation}
	where 
	\begin{equation} \label{betan} \beta(n) = \beta(n, C) = \mu(n) \sum_{c | n, c \leq C} \mu(c).\end{equation}
	We shall require that (\ref{rdx type2}) holds for every $C$ satisfying 
	\[1 \leq C \leq XD^{-1}\]
	and for some $\Delta \geq \delta \geq 2$, that (\ref{rdx type2}) holds with 
	\[\Delta^{-1} \sqrt{D} < N < \delta^{-1} \sqrt{X}.\]
	Here $\Pi$ is the product of all primes $p < P$ with $P$ a parameter chosen so that
	\begin{equation} 2 \leq P \leq \Delta^{2^{-35} \log \log X}.
	\end{equation}
	The main result of \cite{FI2} then states:
	\begin{proposition}[Theorem 1, \cite{FI2}] \label{FI ASP} Assuming that $\A = (a_n)$ satisfies (\ref{asp1}), (\ref{asp2}), (\ref{asp3}), (\ref{AdX crude1}), (\ref{rdx type1}), and (\ref{rdx type2}). Then
		\begin{equation} \label{SX asp} S(X) = HA(X) \left\{1 + O \left(\frac{\log \delta}{\log \Delta} \right) \right\}
		\end{equation}
		where $H$ is the positive constant given by the convergent product
		\[H =  \prod_p (1 - g(p)) \left(1 - \frac{1}{p} \right)^{-1}.\]
	\end{proposition} 
	
	As is well-known by now, the optimal form of Hypothesis (\ref{rdx type1}) is usually relatively easy to obtain for sequences of interest, and the bottleneck for the asymptotic formula (\ref{SX asp}) is the Type II estimate (\ref{rdx type2}). Indeed, almost all of the cases where (\ref{SX asp}) has been obtained involve \emph{norm forms}; this includes \cite{FouI}, \cite{LSX}, \cite{FI1}, \cite{HB1}, \cite{HBM}, \cite{HBL}, and \cite{May}. The present work is not an exception to this rule. \\ \\
	In view of Proposition \ref{FI ASP} the most pressing matter to resolve in order to obtain Theorem \ref{MT} is confirm (\ref{rdx type2}) in the relevant setting. However, this is quite difficult to do directly; indeed in \cite{FI1} the main obstacle was the calculation of certain `main terms' occurring in the bilinear sum estimates which required substantial harmonic analysis to obtain. Instead, we will adopt an approach pioneered by Heath-Brown in \cite{HB1}, which we call the \emph{comparison sieve}, which allows us to avoid certain difficult main term calculations by comparing a given sequence $\A$ to a well-know sequence $\B$ where such main terms are readily available. Indeed, we will carry out a nearly identical strategy as Heath-Brown and Li in \cite{HBL}. In the next section, we will discuss Heath-Brown's comparison sieve following the set-up in \cite{HBL}.

	\section{Heath-Brown's comparison sieve}
	\label{HB sieve} 
	
	In this section, we describe the ideas given by D.~R.~Heath-Brown in \cite{HB1} and expanded upon and refined in \cite{HBM} and \cite{HBL}. Heath-Brown's great insight is that quite often it is possible to establish the infinitude of primes in a sequence $\A$ by comparing it to a suitable sequence $\B$ known to contain infinitely many primes, suitably weighted. For example in \cite{HB1} Heath-Brown compared the sequence of values of the binary cubic form $x^3 + 2y^3$ (weighted by multiplicity) and the sequence of values taken by the norm form of the cubic field $K = \bQ(\sqrt[3]{2})$. \\ 
	
	We shall consider two non-negative sequences $\A = (a_n), \B = (b_n)$ supported on positive integers $n \leq X$, and put
	\begin{equation} \label{piAB} \pi(\A) = \sum_p a_p \text{ and } \pi(\B) = \sum_p b_p,\end{equation}
	where the summations run over primes. If one establishes an asymptotic relation of the form
	\[\pi(\A) = \varkappa \pi(\B)(1 + o(1))\]
	say, then an asymptotic formula for $\pi(\B)$ implies an asymptotic formula for $\pi(\A)$. In particular, this allows us to avoid working through the difficult harmonic analysis in \cite{FI1}, and allows one to work with estimates that apply to general complex sequences rather than relying on properties of the M\"{o}bius function. \\ 
	
	To simplify matters, we will restrict the variable of interest, namely $\ell$, to a short interval of the shape $I(X) = (X^\ast, (1 + \eta)X^\ast]$ where $\eta \asymp (\log X)^{-1}$ and $X^{1/2} (\log X)^{-4} \leq X^\ast \leq c_f X^{1/2}$ where
	\[c_f = \begin{cases} \sup_{f(x,y) \leq 1} y & \text{if } f \text{ is definite} \\ 1 & \text{if } f \text{ is indefinite}. \end{cases} \]
	 We then define
	\begin{equation} \label{an def} a_n = \sum_{\substack{f(m, \ell) = n \\ \ell \in I(X)}} \fZ(\ell) 
	\end{equation}
	and
	\begin{equation} \label{bn def} b_n = \sum_{\substack{f(m, \ell) = n \\ \ell \in I(X)}} \Lambda(\ell).
	\end{equation}
	Here 
	\begin{equation} \label{fzdef} \fZ(\ell) = \begin{cases} 2 p \log p & \text{if } \ell = p^2 \\ 0 &\text{ otherwise} \end{cases}
	\end{equation}
	and $\Lambda$ is the von Mangoldt function. In the definite case Lam, Schindler, and the author proved that $\pi(\B)$ satisfies an asymptotic formula. We will extend this to the indefinite, irreducible case. \\
	
	One notes that the sequences $(a_n), (b_n)$ introduced in (\ref{an def}) and (\ref{bn def}). The analogous sequences $\A^\spadesuit, \B^\spadesuit$ for the purpose of Theorem \ref{MT0} are
	\begin{equation} \label{MT0 def} a_n^\spadesuit = \sum_{\substack{f(m, \ell) = n \\ \ell \in I(X)}} \fZ^\spadesuit(\ell) \text{ and } b_n^\spadesuit = \sum_{\substack{f(m,\ell) = n \\ \ell \in I(X)}} 1
	\end{equation}
respectively, where
\begin{equation} \label{fzspade} \fZ^\spadesuit(\ell) = \begin{cases} 2k & \text{if } \ell = k^2 \\ 0 & \text{otherwise.} \end{cases}\end{equation}
We emphasize that the integer $k$ appearing in (\ref{fzspade}) is not required to be prime, unlike in (\ref{fzdef}). \\

	Having established the asymptotic formula for $\pi(\B), \pi(\B^\spadesuit)$, we will then prove the an analogue of Proposition 1 in \cite{HBL}. In \cite{HBL} they introduced the quantity 
	\[\mu(I) = \int_I \sqrt{X - t^2} dt = \int_I \int_0^{\sqrt{X - t^2}} ds dt.\]
	In other words, $\mu(I)$ is the area of the subset of the positive half-disk with $y$-coordinate restricted to $I$. We generalize this definition to
	\begin{equation} \label{Iarea} \mu_f(I) = \text{Area} \{(x,y) \in \bR^2 : 0 < f(x,y) < X, y \in I(X) \} = \int_I \int_{0 < f(x,y) < X} ds dt.
	\end{equation} 
	Observe that $\mu_f(I) \ll_f \sqrt{X} \cdot |I|$, where $|I|$ is the length of $I$. This brings us to the following statement: 
	
	\begin{proposition} \label{main prop} Let $\A = (a_n), \B = (b_n)$ be given as in (\ref{an def}) and (\ref{bn def}). Then we have the asymptotic relation
		\[\lvert \pi(\A) - \pi(\B) \rvert \ll_\ep \frac{\mu_f(I) \log \log X}{(\log X)^{2}} \]
		holds. Similarly, for $\A^\spadesuit, \B^\spadesuit$ given by (\ref{MT0 def}) one has
		\[\left \lvert \pi(\A^\spadesuit) - \pi (\B^\spadesuit) \right \rvert \ll_\ep \frac{\mu_f(I) \log \log X}{(\log X)^{2}}\]
	\end{proposition} 
	
	We will see that this is enough to prove Theorems \ref{MT0} and \ref{MT} as in the proof of Theorem 1 from Proposition 1 in \cite{HBL}. First we will prove that 
	\begin{equation} \pi(\B) = \frac{\nu_f \mu_f(I)}{\log X} \left(1 + O \left(\frac{1}{\log X} \right) \right), 
	\end{equation}
	this following from Theorem \ref{MT2} via partial summation. In the case of $\B$ and $f$ is definite we start with the asymptotic formula (\ref{pos def eq}) and write it as 
	\[\sum_{q \leq X} \Lambda(q) \sum_{f(m,\ell) = q} \Lambda(\ell) = \nu_f \fS_f^\prime X + O_A (X (\log X)^{-A}).\]
	Writing $\Psi(q) = \displaystyle \sum_{f(m,\ell) = q} \Lambda(\ell)$ and replacing $\Lambda(q)$ with $\log q$ (supported on primes), we have by partial summation
	\[\log X \sum_{q \leq X} \Psi(q) - \int_1^X \frac{1}{t} \left(\sum_{q \leq t} \Psi(q)\right) dt = \nu_f \fS_f^\prime X + O_A(X(\log X)^{-A}).\]
	An upper bound sieve gives that
	\[\sum_{q \leq X} \Psi(q) = O \left(\frac{X}{\log X} \right),\]
	hence
	\[\log X \sum_{q \leq X} \Psi(q) = \nu_f \fS_f^\prime X + O_A (X (\log X)^{-A}) + O \left(\int_1^X \frac{dt}{\log t} \right) \]
	and thus
	\[\sum_{q \leq X} \Psi(q) = \frac{\nu_f \fS_f^\prime X}{\log X} \left(1 + O \left(\frac{1}{\log X} \right) \right).\]
	By replacing $\Psi(q)$ with 
	\[\Psi^\prime(q) = \sum_{\substack{f(m,\ell) = q \\ \ell \in I(X)}} \Lambda(\ell)\] 
	we see from the same argument that 
	\[\sum_{q \leq X} \Psi^\prime(q) = \frac{\nu_f \mu_f(I)}{\log X} \left(1 + O \left(\frac{1}{\log X} \right) \right),\]
	as desired. The same argument applies to the indefinite case, following (\ref{indef eq}). \\
	
	Thus Proposition \ref{main prop} gives
	\begin{equation} \pi(\A) = \frac{\nu_f \mu_f(I)}{\log X} \left(1 + O \left(\frac{\log \log X}{\log X} \right) \right). 
	\end{equation}
	We then proceed by partial summation as in \cite{HBL}. We consider intervals $I_j = (X_j, X_j(1 + \eta)]$ be a partition of $(X^{1/2} (\log X)^{-4}, c_f X^{1/2}]$. Here $\eta \asymp (\log X)^{-1}$ is chosen so we have an exact partition. We let $\A_j$ be defined as in (\ref{an def}) with $I(X) = I_j$. Note that the number of pairs $(a,p)$ with $0 < f(a,p^2) \leq X$ and $p | a, p^2 \in I$ is bounded by 
	\[\sum_{p^2 \in I} \frac{\sqrt{X}}{p} \ll_\ep X^{1/2 + \ep}.\]
	It follows that
	\begin{align*} & \# \{(a,p) : 0 < f(a, p^2) \leq X \text{ is prime }, p \text{ is prime}, p \leq X^{1/4}\} \\
		& = \sum_j \frac{1}{\sqrt{X_j} \log X_j} \pi(\A_j) \left(1 + O \left(\frac{1}{\log X} \right) \right) + O \left(\frac{X^{3/4}}{(\log X)^3} \right) \\
		& = \frac{\nu_f + O \left((\log X)^{ - 1} \log \log X \right)}{(\log X)^2} \sum_j \frac{\mu_f(I_j)}{\sqrt{X_j}} + O \left(\frac{X^{3/4}}{(\log X)^3} \right) \\
		& = \frac{\nu_f + O \left((\log X)^{ - 1} \log \log X \right)}{(\log X)^2} \int_{\sqrt{X}/(\log X)^4}^{\sqrt{X}} \frac{1}{\sqrt{t}} \int_{0 < f(s,t) < X} ds dt + O \left(\frac{X^{3/4}}{(\log X)^3} \right) \\
		& = \frac{\nu_f \fS_f X^{3/4} }{(\log X)^2} \left(1 + O_\ep \left(\frac{\log \log X}{\log X} \right)  \right). 
	\end{align*} 
	Thus Theorem \ref{MT} follows from Proposition \ref{main prop}. Next we do something similar to deduce Theorem \ref{MT0}. In this case it is trivial that
	\[\pi(\B^\spadesuit) = \frac{\nu_f \mu_f(I)}{\log X} \left(1 + O \left( \frac{1}{\log X} \right)  \right), \]
	since this is a direct consequence of Landau's prime ideal theorem. Therefore Proposition \ref{main prop} gives
	\[\pi(\A^\spadesuit) =\frac{\nu_f \mu_f(I)}{\log X} \left(1 + O \left( \frac{1}{\log X} \right)  \right) \]
	We the proceed by partial summation as above, but noting that the weight is $2k$ rather than $2p \log p$. The same calculation then gives 
	\[\#\{(a,b) : 0 < f(a,b^2) \leq X \text{ is prime}, b \leq X^{1/4}\} = \frac{\nu_f \fS_f X^{3/4}}{\log X} \left(1 + O \left(\frac{1}{\log X} \right) \right)\]
	which suffices to prove Theorem \ref{MT0}. \\
	
	In order to establish Proposition \ref{main prop} we apply the same sieve procedure to the pairs $(\A, \B)$ and $(\A^\spadesuit, \B^\spadesuit)$, producing cancellation at key junctures and upper bounding the rest. For any complex sequence $\C = (c_n)$ supported on the positive integers put
	\[S(\C, Z) = \sum_{\substack{n \in \bN \\ p | n \Rightarrow p > Z }} c_n\]
	and for each $d \in \bN$ put
	\[\C_d = \{c_{dn} : n \in \bN\}.\]
	We fix 
	\begin{equation} \label{deldef}\delta_1 = \delta_1(X) = (\log X)^{\varpi - 1} \text{ and } \delta_2 = \delta_2(X) = \dfrac{A_1 \log \log X}{\log X} \end{equation}
	for some some large positive number $A_1$ and small number $0 < \varpi < 1$ which we specify later. We remark that in \cite{HBL} they just chose a single choice of $\delta$. The reason why we are having two separate parameters is to obtain the superior error term in Theorem \ref{MT} and the error term in Theorem \ref{MT0}. \\  \\
	We also fix $Y > X^{1/3}$, where the specific choice of $Y$ will be made when it is relevant. Now put
	\begin{equation} \label{SiC def} S_1(\C) = S(\C, X^{\delta_1}), S_2(\C) = \sum_{X^{\delta_1} \leq p < Y} S(\C_p, p), S_3(\C) = \sum_{Y \leq p < X^{1/2 - \delta_2}} S(\C_p, p).\end{equation} 
	The astute reader will note that $S_1(\C)$ is readily handled by the Fundamental Lemma of Sieve Theory, giving an asymptotic formula; see for example Corollary 6.10 in \cite{FI3}. By Buchstab's identity, we have
	\[\pi(\C) = S\left(\C, X^{1/2} \right) = S_1(\C) - S_2(\C) - S_3(\C) - \sum_{X^{1/2 - \delta_2} \leq p \leq X^{1/2}} S(\C_p,p). \]
	The last sum can be handled by Selberg's upper bound sieve, and we conclude:
	\begin{lemma} For $Y = X^{17/48}$ and $\C = \A, \B, \A^\spadesuit, \B^\spadesuit$ we have
		\[\pi(\C) = S_1(\C) - S_2(\C) - S_3(\C) + O \left(\frac{\delta_2 \mu_f(I)}{\log X} \right).\]
	\end{lemma}
	We will see that $S_3(\C)$ can be written in terms of appropriate bilinear forms, but $S_2(\C)$ will require further treatment. Let us put
	\[T^{(n)}(\C) = \sum_{\substack{X^{\delta_1} \leq p_n < \cdots < p_1 < Y \\ p_1 \cdots p_n < Y}} S(\C_{p_1 \cdots p_n}, X^{\delta_1})\]
	and
	\[U^{(n)}(\C) = \sum_{\substack{X^{\delta_1} \leq p_{n+1} < \cdots < p_1 < Y \\ p_1 \cdots p_n < Y \leq p_1 \cdots p_{n+1}}} S(\C_{p_1 \cdots p_{n+1}}, p_{n+1}).\]
	We then have:
	\begin{lemma} For $n_0 = \left \lfloor \frac{\log Y}{\delta_1 \log X} \right \rfloor$ we have
		\[S_2(\C) = \sum_{1 \leq n \leq n_0} (-1)^{n-1} \left(T^{(n)}(\C) - U^{(n)}(\C) \right) \]
	\end{lemma}
	The sums 
	\begin{equation} \label{S1dif} |S_1(\A) - S_1(\B)|, |S_1(\A^\spadesuit) - S_1(\B^\spadesuit)|  \end{equation}  
	and
	\begin{equation} \label{Tndif} \sum_{1 \leq n \leq n_0} \left \lvert T^{(n)}(\A) - T^{(n)}(\B) \right \rvert, \sum_{1 \leq n \leq n_0} \left \lvert T^{(n)}(\A^\spadesuit) - T^{(n)}(\B^\spadesuit) \right \rvert \end{equation}
	can be handled by our Type I estimate Proposition \ref{Type1 prop} and the Fundamental Lemma. To control these sums it suffices to prove: 
	\begin{proposition} \label{fun lem app} Let $\fQ$ be a set of square-free numbers not exceeding $Y$. Then for any $A > 0$ we have
		\[\left \lvert \sum_{q \in \fQ} S\left(\A_q, X^{\delta_1} \right) - \sum_{q \in \fQ} S\left(\B_q, X^{\delta_1} \right) \right \rvert \ll_A \frac{X}{(\log X)^A}\]
		and
		\[\left \lvert \sum_{q \in \fQ} S\left(\A_q^\spadesuit, X^{\delta_1} \right) - \sum_{q \in \fQ} S\left(\B_q^\spadesuit, X^{\delta_1} \right) \right \rvert \ll_A \frac{X}{(\log X)^A}\]
	\end{proposition}
	
	By the definition of $S_1(\C)$ and $T^{(n)}(\C)$, it is clear that Proposition \ref{fun lem app} gives the bound of $O_A(X (\log X)^{-A})$ for both (\ref{S1dif}) and (\ref{Tndif}). \\ 
	
	We now give a proof for Proposition \ref{fun lem app}. 
	
	\begin{proof}[Proof of Proposition \ref{fun lem app}] The Fundamental Lemma allows us to give an asymptotic formula for the sum
		\[\sum_{q \in \fQ} S\left(\C_q, X^{\delta_1} \right)\]
		for $\C = \A, \B, \A^\spadesuit, \B^\spadesuit$. Recall that $\delta_1 = (\log X)^{\varpi - 1}$. Proposition \ref{Type1 prop} gives us a level of distribution of $X^{3/4} (\log X)^{-B}$ for some large $B$. We then apply an upper and lower bound sieve of level of distribution $X^{1/4}$, so that the sifting variable 
		\[s = \frac{\log D}{\log z} = \frac{\log X^{1/4}}{\log X^{\delta_1}} = \frac{1}{4 \delta_1}.\] 
		We use the usual notation 
		\[V(z) = \prod_{p < z} (1 - g(p)) = \prod_{p < z} \left(1 - \frac{\rho_f(p)}{p} \right),\]
		and
		\[R_{d}(\C) = |A_d(\C) - M_d(\C)|\]
		with $M_d(\C)$ as in Proposition \ref{Type1 prop}. By Corollary 6.10 in \cite{FI3} and applying Proposition \ref{Type1 prop} we obtain
		\begin{align*} \sum_{q \in \fQ} S\left(\C_q, X^{\delta_1} \right) & = V \left(X^{\delta_1} \right) \sum_{q \in \fQ} \frac{\rho_f(q)}{q} \mu_f(I) \left(1 + O \left(\exp( - (4 \delta)^{-1} ) \right) \right) + O \left(\sum_{q \in \fQ} \sum_{d < X^{1/4}} R_{dq}(\C) \right) \\
			& =  V \left(X^{\delta_1} \right) \sum_{q \in \fQ} \frac{\rho_f(q)}{q} \mu_f(I) \left(1 + O \left(\frac{1}{(\log X)^A} \right) \right) + O \left( \sum_{d < X^{3/4 - 1/8}} \tau(d) R_{q}(\C) \right) \\
			& =  V \left(X^{\delta_1} \right) \sum_{q \in \fQ} \frac{\rho_f(q)}{q} \mu_f(I) \left(1 + O \left( \frac{1}{(\log X)^A} \right) \right) + O_A \left(X (\log X)^{-A} \right)
		\end{align*}
		for any $A > 0$. The last line is independent of whether $\C = \A, \B, \A^\spadesuit$ or $\C = \B^\spadesuit$. Since $V(X^{\delta_1}) \leq 1$ it follows that 
		\begin{align*} \sum_{q \in \fQ} \left(S(\A_q, X^{\delta_1}) - S(\B_q, X^{\delta_1}) \right) & \ll_A \frac{1}{(\log X)^{A}} \mu_f(I) \sum_{q \in \fQ} \frac{\rho_f(q)}{q} + X (\log X)^{-A} \\
			& \ll_A X (\log X)^{-A + 2},
		\end{align*}
		since $\rho_f(q) \ll \tau(q)$. Likewise,
\[ \sum_{q \in \fQ} \left(S(\A_q^\spadesuit, X^{\delta_1}) - S(\B_q^\spadesuit, X^{\delta_1}) \right) \ll_A X (\log X)^{-A+2}.\]	
 \end{proof}

	Thus it remains to show that
	\begin{equation} \label{S3 bd} |S_3(\A) - S_3(\B)| \ll_A \frac{X}{(\log X)^A} \text{ and } \left \lvert U^{(n)}(\A) - U^{(n)}(\B) \right \rvert \ll_A \frac{X}{(\log X)^A} \text{ for } n \geq 3
	\end{equation}
	and
	\begin{equation} \label{Un bd} \left \lvert U^{(n)}(\A) - U^{(n)}(\B) \right \rvert \ll \frac{\delta_2 \mu_f(I)}{\log X} 
	\end{equation}
	for $n = 1,2$, with analogous statements for $\A^\spadesuit, \B^\spadesuit$. \\ \\
	We proceed to reduce the verification of (\ref{S3 bd}) and (\ref{Un bd}) to a bilinear sum estimate. 
	
	\subsection{Reduction to a bilinear sum bound}
	
	Let us write $U^{(1)}$ and $U^{(2)}$ into a more convenient form, as in \cite{HBL}. To do so let us put
	\begin{align*} & U_1^{(1)}(\C) = \sum_{\substack{X^{\delta_1} \leq p_2 < p_1 < Y \\ Y \leq p_1 p_2 < X^{1/2 - \delta_2}}} S(\C_{p_1 p_2}, p_2) \\
		& U_2^{(1)}(\C) = \sum_{\substack{X^{\delta_1} \leq p_2 < p_1 < Y \\ p_1 p_2 \geq X^{1/2 + \delta_2}}} S(\C_{p_1 p_2}, p_2) \\
		& U_1^{(2)}(\C) = \sum_{\substack{X^{\delta_1} \leq p_3 < \cdots < p_1 < Y \\ p_1 p_2 < Y \leq p_1 p_2 p_3 < X^{1/2 - \delta_2}}} S(\C_{p_1 p_2 p_3}, p_3), \text{ and} \\
		& U_2^{(2)}(\C) = \sum_{\substack{X^{\delta_1} \leq p_3 < \cdots < p_1 < Y \\ p_1 p_2 < Y \leq p_1 p_2 p_3 \\ p_1 p_2 p_3 \geq X^{1/2 + \delta_2}}} S(\C_{p_1 p_2 p_3}, p_3).
	\end{align*}
	We now state Lemmas 6 and 7 from \cite{HBL}. Their proofs apply equally well, but since for us $\delta_1, \delta_2$ are different we write out the proofs. 
	\begin{lemma}[Lemma 6, \cite{HBL}] For $\C = \A, \B$ we have $U^{(j)}(\C)$ satisfies
		\begin{equation} \label{Lem6bd1} U^{(1)}(\C) = U_1^{(1)}(\C) + U_2^{(1)}(\C) + O \left(\frac{\delta_2 \mu_f(I)}{\log X} \right)\end{equation}
		and
		\begin{equation}\label{Lem6bd2} U^{(2)}(\C) = U_1^{(2)}(\C) + U_2^{(2)}(\C) + O \left(\frac{\delta_2 \mu_f(I)}{\log X} \right)\end{equation}
	\end{lemma}

\begin{proof} To prove (\ref{Lem6bd1}) it suffices to show
	\[\sum_{\substack{X^{\delta_1} \leq p_2 < p_1 < Y \\ X^{1/2 - \delta_2} < p_1 p_2 \leq X^{1/2 + \delta_2}}} S(\C_{p_1 p_2},  p_2) \ll \frac{\delta_2 \mu_f(I)}{\log X}. \]
	In the sum above we have 
	\[p_2 \geq \frac{X^{1/2 - \delta_2}}{p_1} > \frac{X^{1/2 - \delta_2}}{Y} > X^{1/10}\]
	so we may apply Selberg's upper bound sieve and our level of distribution to obtain 
	\begin{align*} \sum_{\substack{X^{\delta_1} \leq p_2 < p_1 < Y \\ X^{1/2 - \delta_2} < p_1 p_2 \leq X^{1/2 + \delta_2}}} S(\C_{p_1 p_2},  p_2) & \ll \sum_{\substack{X^{\delta_1} \leq p_2 < p_1 < Y \\ X^{1/2 - \delta_2} < p_1 p_2 \leq X^{1/2 + \delta_2}}} S(\C_{p_1 p_2},  X^{1/10}) \\
		& \ll \frac{\mu_f(I)}{\log X} \sum_{\substack{X^{1/10} < p_2 < p_1 < Y \\ X^{1/2 - \delta_2} < p_1 p_2 < X^{1/2 + \delta_2}}} \frac{1}{p_1 p_2} \\
		& \ll \frac{\delta_2 \mu_f(I)}{\log X}.
	\end{align*}
The proof for (\ref{Lem6bd2}) follows similarly. 
\end{proof}
	
	\begin{lemma}[Lemma 7, \cite{HBL}] \label{HBL7} Let $\kappa$ be a positive number satisfying $X^{-\delta_1} \leq \kappa \leq 1$. Let $N_1, N_2$ be positive numbers in the interval $[X^\delta, X^{1/3}]$. We then have for any $A > 0$:
		\[\sum_{N_1 \leq p_1 \leq (1 + \kappa) N_1} \sum_{N_2 \leq p_2 \leq (1 + \kappa)N_2} \sum_{n \equiv 0 \pmod{p_1 p_2}} c_n \tau(n) \ll_A \kappa^2 X (\log X)^{2^{17}} + \frac{X}{(\log X)^A}.\]
	\end{lemma}
	
	For $k \geq 3$, the condition of summation in $U^{(k)}(\C)$ is 
	\[Y \leq p_1 \cdots p_{k+1} < (p_1 \cdots p_k)^{\frac{k+1}{k}} \leq Y^{4/3} < X^{1/2 - \delta_2}.\]
	Therefore, upon defining
	\[U_\ast^{(k)}(\C) = \sum_{\substack{X^{\delta_1} \leq p_{k+1} < \cdots < p_1 \cdots p_k < Y \leq p_1 \cdots p_{k+1} < X^{1/2 - \delta}}} S(\C_{p_1 \cdots p_{k+1}}, p_{k+1})\]
	we have
	\[S_3(\C) = U_\ast^{(0)}(\C), U_1^{(1)}(\C) = U_\ast^{(1)}(\C), U_1^{(2)}(\C) = U_\ast^{(2)}(\C)\]
	and
	\[U^{(k)}(\C) = U_\ast^{(k)}(\C) \text{ for } k \geq 3.\] 
	If $p \in \J = [V, (1 + \kappa) V)$ and an integer is counted by $S(\C_{pq}, V)$ but not by $S(\C_{pq}, p)$, then $n$ has at least two prime factors in $\J$. In our application we will have $V \leq X^{1/2 - \delta_2}$ and $n \geq X(\log X)^{-8}$ and therefore $n$ will have at least one more prime factor. It follows that
	\[V^3 \leq n \leq X.\]
	A given integer $n$ may be counted multiple times by $U_\ast^{(k)}(\C)$ but the multiplicity is bounded by the number of choices for $p_{k+1} < \cdots < p_1$ all dividing $n$, and therefore the multiplicity is at most $\tau(n)$. Applying Lemma \ref{HBL7} and setting 
	\[\J(r) = [V_r, V_{r+1}) = \left[X^{\delta_1} (1 + \kappa)^r, X^{\delta_1}(1 + \kappa)^{r+1} \right), r \geq 0\]
	and $R \ll \kappa^{-1} \log X$ satisfying $X^{\delta_1} (1 + \kappa)^R > X$, we obtain
	\begin{equation} U_\ast^{(k)}(\C) = \sum_{0 \leq r \leq R} \sum_{p \in \J(r)} \sum_{p < p_k < \cdots < p_1 \cdots p_k < Y \leq p_1 \cdots p_k p < X^{1/2 - \delta_2}} S(\C_{p_1 \cdots p_k p}, V_r) 
	\end{equation}
	\[+ O_A \left(\kappa X (\log X)^{1 + 2^{17}} + \kappa^{-1} \frac{X}{(\log X)^{A - 1}} \right).\]
	We note that we need to make sure that both 
	\[\kappa X (\log X)^{1 + 2^{17}}, \kappa^{-1} \frac{X}{(\log X)^{A-1}}\] 
	are $O(X (\log X)^{-A^\prime})$ for some $A^\prime > 1$. This compels us to choose 
	\[\kappa = (\log X)^{-A/2}.\]
	This gives
	\begin{equation} \label{kapchoice} \kappa X (\log X)^{1 + 2^{17}} = X (\log X)^{1 + 2^{17} - A/2} \text{ and } \kappa^{-1} \frac{X}{(\log X)^{A-1}} = \frac{X}{(\log X)^{A/2 -1}}.
	\end{equation}

	This procedure allows us to reduce our proof to estimations of certain bilinear sums since
	\begin{equation} \label{bisum SC}  \sum_{p \in \J(r)} \sum_{p < p_k < \cdots < p_1 \cdots p_k < Y \leq p_1 \cdots p_k p < X^{1/2 - {\delta_2}}} S(\C_{p_1 \cdots p_k p}, V_r) = \sum_{m,n} \alpha_m^{(r)} \beta_n^{(r)} c_{mn}\end{equation}
	where $\alpha_m^{(r)}$ is the characteristic function for the integers $m$ all of whose prime factors are at least $V_r$ and $\beta_n^{(r)}$ is the characteristic function for integers $n = p_1 \cdots p_k p$ satisfying
	\[p \in \J(r), p < p_k < \cdots < p_1 < Y \text{ and } p_1 \cdots p_k < Y \leq p_1 \cdots p_k p < X^{1/2 - \delta_2}.\]
	Observe that $\beta_n^{(r)}$ is supported on integers $n \in [Y, X^{1/2 - \delta_2})$. \\ \\
	The procedure for $U_2^{(1)}(\C)$ and $U_2^{(2)}(\C)$ will be somewhat different. We may use Lemma \ref{HBL7} to replace $S(\C_{p_1 p_2}, p_2)$ in $U_2^{(1)}(\C)$ by $S(\C_{p_1 p_2}, V_r)$ when $p_2 \in J(r)$. This yields
	\begin{align*} U_2^{(1)}(\C) & = \sum_{0 \leq r \leq R} \sum_{p_2 \in J(r)} \sum_{\substack{p_1 \geq X^{1/2 - \delta_2}/p_2 \\ p_2 < p_1 < Y}} S(\C_{p_1 p_2}, V_r) + O \left(\kappa X (\log X)^{1 + 2^{17}} \right) + O \left(\kappa^{-1} \frac{X}{(\log X)^{A-1}} \right). 
	\end{align*}
	The sum on the right can be expressed as 
	\[\sum_{0 \leq r \leq R} \sum_{m,n} \alpha_m^{(r)} \beta_n^{(r)} c(mn),\]
	where we now take $\alpha_m^{(r)}$ to be the characteristic function for numbers $m = p_1 p_2$ with $p_2 \in J(r), p_2 < p_1 < Y$ and $p_1 p_2 \geq X^{1/2 + \delta_2}$, and $\beta_n^{(r)}$ to be the characteristic function for those numbers $n$ all of whose prime factors are at least $V_r$. Since $c(n)$ is supported in 
	\[\left((X^\ast)^2 , c_f X \right] \subseteq \left(X (\log X)^{-8}, c_f X \right]\]
	we may assume that $\beta_n^{(r)}$ is supported in 
	\[\left(X (\log X)^{-8} Y^{-2}, X^{1/2 - \delta_2} \right] \subseteq \left(X^{1/4 + 1/48}, X^{1/2 - \delta_2} \right].\]
	This is sufficient for our purposes. We may handle $U_2^{(2)}(\C)$ in an analogous fashion. \\ \\
	On setting $\kappa = (\log X)^{-A/2}$ we find that each of 
	\[S_3(\C), U_1^{(1)}(\C), U_2^{(1)}(\C), U_1^{(2)}(\C), U_2^{(2)}(\C), \text{ and } U^{(k)}(\C)\]
	for $k \geq 3$ can be expressed as a sum of $O(R)$ bilinear sums as in (\ref{bisum SC}), together with an error term of $O_A \left(X (\log X)^{1 + 2^{17} - A/2}\right)$. Thus it will be sufficient to prove: 
	
	\begin{proposition}[Main Bilinear Sum Estimate] \label{main bisum} Let $\xi > 0$ and suppose $X^{1/4 + \xi} \leq N < X^{1/2 - \delta_2}$. Suppose $(\alpha_m), (\beta_n)$ are two complex sequences having sup-norm at most $1$ supported on natural numbers with no prime factors less than $X^\delta$. Then for any $A > 0$ we have
		\begin{equation} \label{bisum prop eq} \sum_{N < n \leq 2N} \sum_{m < X/N} \alpha_m \beta_n (a_{mn} - b_{mn}) \ll_{A, \xi} \frac{X}{(\log X)^A} \end{equation}
		and 
		\begin{equation} \label{bisum prop eq2} 
			\sum_{N < n \leq 2N} \sum_{m < X/N} \alpha_m \beta_n (a_{mn}^\spadesuit - b_{mn}^\spadesuit) \ll_{A, \xi} \frac{X}{(\log X)^A}
		\end{equation}
	\end{proposition}
	
	It will be important that the sequences $\{\alpha_m\}, \{\beta_n\}$ are supported on those numbers whose prime factors all exceed $X^\delta$, and in particular, they are supported on odd numbers. \\ \\
	The remainder of the paper is devoted to proving Proposition \ref{main bisum}. In particular, Propositions \ref{HBL prop6} and \ref{HBL prop7} will imply Proposition \ref{main bisum}. In order to get there, we need to decompose the terms $c_{mn}$ for any positive integers $m,n$ into components that resemble $c_m, c_n$. This turns out to be somewhat delicate and we will require the composition laws of the ideals of $\O_K$, expressed in terms of ideal numbers. This will be the primary focus of the next section.

	\section{Algebraic characterization of the multiplicative structure in terms of ideal numbers}
	\label{algebra}
	
	The main purpose of this section is obtain an analogue of Proposition 2.3 in \cite{LSX}. However, instead of using an explicit Dirichlet composition law as in \cite{LSX} we will instead adopt the language of Hecke's \emph{ideal numbers} as in \cite{HBM}. \\ 
	
	Choose ideals $\fa_1, \cdots, \fa_t$ whose classes generate the ideal class group of $\O_K$. Having fixed these representatives, every fractional ideal $\fa \subseteq \O_K$ has a unique decomposition
	\[\fa = (\alpha) \fa_1^{\ell_1} \cdots \fa_t^{\ell_t}\]
	where $\alpha \in K^\ast$ and $\ell_{j} \in \bZ$ with $0 \leq \ell_j < h_j$, with $h_j$ the smallest positive integer such that $\fa_j^{h_j} = (\alpha_j)$ is principal. Then the class number $h(K)$ of $\O_K$ is equal to
	\begin{equation} \label{classnumb} h(K) = \prod_{j=1}^t h_j.\end{equation} 
	Let us choose complex numbers $b_1, \cdots, b_t$ so that 
	\[b_j^{h_j} = \alpha_j \text{ for } j = 1, \cdots, t,\]
	and $b_j^{(i)}$ are complex numbers such that
	\[\left(b_j^{(i)}\right)^{h_j} = \alpha_j^{(i)} \text{ for } i = 1,2.\]
	Now put $L = K (b_1, \cdots, b_t)$ and $\fJ(K)^\ast$ the subgroup of $L^\ast$ generated by $K^\ast$ and $\{b_j : 1 \leq j \leq t\}$. Then $\fJ(K) = \{0\} \cup \fJ(K)^\ast$ is the domain of \emph{ideal numbers} of $K$. The quotient group $\fI(K)^\ast/\O_K^\ast$ is then isomorphic to the group of fractional ideals of $K$. Each $\gamma \in \fI(K)$ corresponds a unique ideal $J(\gamma)$ of $\O_K$; the norm of the ideal $J(\beta)$ is given by the product
	\[N(\gamma) = N(J(\gamma)) = \gamma^{(1)} \gamma^{(2)}.\]
	Further, we have $J(\gamma)$ is an integral ideal of $\O_K$ if and only if $\gamma \in \O_L$. \\ \\
	We thus have a correspondence between the ideal classes of $\O_K$ and a subset of integers in $\O_L$. Indeed, we can say that $\gamma, \gamma^\prime \in \fI(K)$ belong to the same class if and only if the corresponding ideals $J(\gamma), J(\gamma^\prime) \subseteq \O_K$ are in the same ideal class. It follows that we may partition $\fI(K)$ into $h(K)$ classes, corresponding to the ideal classes of $\O_K$. Such a class of ideal numbers, say $A$, has an integral basis $\{w_1, w_2\}$ such that
	\[A = \{a_1 w_1 + a_2 w_2 : (a_1, a_2) \in \bQ^2\}\]
	and
	\[A \cap \O_L = \{a_1 w_1 + a_2 w_2 : (a_1, a_2) \in \bZ^2\}.\]
	Further, the discriminant of $A$, viewed as a $\bZ$-lattice, is equal to $\Delta(K)$. Moreover for any basis $\{w_1, w_2\}$ of $A$ and $\alpha \in A \setminus \{0\}$ we have that $\{\alpha^{-1} w_1, \alpha^{-1} w_2\}$ is a basis of $K/\bQ$. This implies that there is a unique \emph{dual basis} $\{\widetilde{w_1}, \widetilde{w_2}\}$ of $A^{-1}$ defined by the condition 
	\begin{equation} \label{dualbase} \Tr(w_i \widetilde{w_j}) = \begin{cases} 1 & \text{if } i = j \\ 0 & \text{otherwise}.\end{cases}\end{equation}
	We use the notation $\Cl \fa, \Cl \alpha$ for the ideal class of the integral ideal $\fa \subset \O_K$ and the class of ideal numbers of the ideal number $\alpha$. \\ 
	
	Next we show that there is a correspondence between rank-two submodules of $\O_K$ and $\SL_2(\bZ)$-equivalence classes of irreducible integral binary quadratic forms having splitting field $K$. To establish this correspondence, first start with a rank-two submodule
	\[\Lambda = \{a_1 \omega_1 + a_2 \omega_2 : a_1, a_2 \in \bZ\}\] 
	with $\omega_1, \omega_2 \in \O_K$. Then the form
	\begin{equation} \label{formidcor}g(x,y) = N_{K/\bQ}(\omega_1 x + \omega_2 y) N (\fd^{-1}) \end{equation}
	is an irreducible integral binary quadratic form with splitting field $K$. \\ \\ 
	Conversely, take an arbitrary irreducible integral binary quadratic form $g$ which splits over $K$. Then there exists an integral non-singular matrix $M$ such that 
	\[g(x,y) = g^\ast ((x,y)M)\]
	where $g^\ast$ is a primitive integral binary quadratic form with discriminant equal to $\Delta(K)$. Gauss's composition law then implies that $g^\ast$ corresponds to an ideal class $\alpha$, and in particular, can be expressed in the form 
	\[g^\ast(x,y) = N_{K/\bQ}(\alpha_1 x + \alpha_2y) N(\alpha^{-1})\]
	with $\alpha = (\alpha_1, \alpha_2)$. Viewing $\alpha$ as a $\bZ$-module and applying the transformation induced by $M$ then gives the form $g$. \\ 
	
	Now let $\ff$ be the $\bZ$-module associated to $f$ with basis $\{\nu_1, \nu_2\}$ so that 
	\begin{equation} f(x,y) = N_{K/\bQ}(\nu_1 x + \nu_2 y) N(\fd(f)^{-1}),
	\end{equation}
where $\fd(f) = (\nu_1, \nu_2)$ is the ideal generated by $\nu_1, \nu_2$. Let $\psi_f$ be the ideal number of the ideal $\fd(f)$. Having identified $\ff$ we define the set of ideals:
\[\mathfrak{A}(f) = \{(\nu_1 a_1 + \nu_2 a_2) \fd(f)^{-1}: a_1, a_2 \in \bZ, \gcd(a_1, a_2)  = 1\}. \]

	We now put $\L$ for the set of ideals in $\O_K$ which are not divisible by a rational prime. An integral ideal number $\gamma \in \fI(K)$ is said to be primitive if $J(\gamma) \in \L$. Next put $\L_0$ the set of primitive ideal numbers $\gamma$ satisfying the condition
	\[\gamma = (N_{L/\bQ}(\gamma))^{1/2} \ep_0^z, \frac{-1}{2} < z \leq \frac{1}{2}, \gamma > 0,\]
	where $\ep_0 > 1$ is a fundamental unit of $\O_K$. \\ 
	
	We now want to use the above discussion to obtain a meaningful decomposition for 
	\begin{equation} c_n = \sum_{\substack{f(m, \ell) = n \\ \ell \in I(X)}} \Upsilon (\ell).
	\end{equation}
	We follow the set-up in \cite{HBM} and introduce, for a given primitive vector $\Bu = (u_1, u_2)$ let $\fF(\Bu)$ be the ideal in $\fA(f)$ given by $(\nu_1 u_1 + \nu_2 u_2) N(\fd(f)^{-1})$. We now put 
	\[\fR(X; n) = \{(u_1, u_2) \in \bZ^2 : u_2 \in I(X), f(u_1, u_2) = n\}. \]
	Note that $\fR(X;n)$ is finite for all $X > 0$ and $n \in \bZ$. We then have
	\[c_n = \sum_{\Bu \in \fR(X;n)} \Upsilon(u_2).\]
	Via the correspondence 
	\[(u_1, u_2) \mapsto (\nu_1 u_1 + \nu_2 u_2)N(\fd(f)^{-1}) = \fF(u_1,u_2)\]
	 $\fR(X;n)$ corresponds to a set of ideals. For a given integer $mn$ we then see that each element $(u_1, u_2)$ of $\fR(X; mn)$ corresponds to a set of ideal factorizations of the form
	\begin{equation} \label{ideal fact} \fm \fn = \fF(u_1, u_2)\end{equation}
	with $N(\fm) = m, N(\fn) = n$. Now associate to $\fm, \fn$ ideal numbers $m^\ast, n^\ast \in \L_0$. Then (\ref{ideal fact}) can be interpreted as multiplication in the set of ideal numbers. To make this concrete, first choose $\{w_1, w_2\}$ to be a basis for the ideal class $\Cl \fd(f)^{-1}$ such that $w_1 \psi_f^{-1} = z \nu_1$ and $w_2 \psi_f^{-1} =  \nu_2$ for some integer $z$. For each pair of ideal classes $A, B = A^{-1}\Cl f$ and any bases $\{a_1, a_2\}, \{b_1, b_2\}$ of $A,B$ respectively we have a composition law
	\[(a_1 x_1 + a_2 x_2)(b_1 y_1 + b_2 y_2) = \psi_f^{-1} (w_1 R_{A,B}(\Bx; \By) + w_2 Q_{A,B}(\Bx; \By)) \]
	holds. By our choice of $\{w_1, w_2\}$ this is equivalent to 
	\[(a_1 x_1 + a_2 x_2)(b_1 y_1 + b_2 y_2) = z \nu_1 R_{A,B}(\Bx; \By) + \nu_2 Q_{A,B}(\Bx; \By).\]
	This gives a bilinear mapping 
	\[\Phi_{A,B} : (\L_0 \cap A) \times (\L_0 \cap B) \rightarrow \{(x,y) \in \bR^2 : y \in I(X)\}, \]
	\[\Phi_{A,B}(m_1, m_2; n_1, n_2) = (R_{A,B}(\Bm; \Bn), Q_{A,B}(\Bm; \Bn))\]
	say. Let us write $A_0 = A \cap \L_0$ and $B_0 = B \cap \L_0$ for convenience. We then have
	\begin{equation} \label{key decomp} c_{mn} = \sum_{A \cdot B = \Cl f} \mathop{\sum \sum}_{\substack{\Bm \in A_0, \Bn \in B_0 \\ N(\Bm) = m, N(\Bn) = n \\ Q_{A,B}(\Bm; \Bn) \in I(X) }} \Upsilon(Q_{A,B}(\Bm; \Bn)).  
	\end{equation}
This is the desired analogue to equation (5.2) in \cite{FI1}. We summarize this below:

	\begin{proposition} \label{alg decomp} For $\C= \A, \B, \A^\spadesuit, \B^\spadesuit$ equation (\ref{key decomp}) holds. 
	\end{proposition}

	\section{Type I estimates} 
	\label{Type I} 
	
	In this section we will establish the necessary Type I estimate we need, following the work of Friedlander and Iwaniec in \cite{FI3}. For this section, we shall put $\lambda(\ell)$ to be any function bounded by one supported on $r$-th powers of integers, and put
	\begin{equation} \label{type1 an} a_n = \sum_{\substack{f(\ell, m) = n \\ \ell \in I(X)}} \lambda(\ell).
	\end{equation}
	We recall that
	\[A_d(X) = \sum_{\substack{n \leq X \\ n \equiv 0 \pmod{d}}} a_n.\]
	For a given positive integer $\ell$ put
	\[\I(\ell; X) = \{x \in \bR^2 : 0 < f(x, \ell) < X\}\]
	and $\iota(\ell; X)$ to be the length of $\I(\ell; X)$. We then expect $A_d(X)$ to be well-approximated by
	\[M_d(X) = \frac{\rho_f(d)}{d} \sum_{\substack{\ell \in I(X) \\ \gcd(\ell, d) = 1}} \lambda(\ell) \frac{\varphi(\ell)}{\ell} \iota(\ell; X) \]
	where $\varphi$ is the Euler totient function and $\rho_f(d)$ is the number of solutions to the congruence
	\begin{equation} \label{quadcong} f(x,1) \equiv 0 \pmod{d}.\end{equation}
	
	Our goal is to establish:
	
	\begin{proposition} \label{Type1 prop} Suppose that $\lambda$ is supported on $r$-th powers. Then uniformly for $X^{\frac{1}{2}} \leq D \leq X^{\frac{r+1}{2r}}$ we have
		\[\sum_{d \leq D} \left \lvert A_d(X) - M_d(X) \right \rvert \ll D^{\frac{1}{4}} X^{\frac{3(r+1)}{8r}} (\log X)^{130}.\]
	\end{proposition}
	
	As usual, our starting point is the following result from \cite{BBDT}, which states that the roots of quadratic congruences are separated as much as possible: 
	
	\begin{proposition}[Proposition 3 \cite{BBDT}] \label{BBDT prop}  Let $F(x,y) = \alpha x^2 + \beta xy + \gamma y^2 \in \bZ[x,y]$ be an arbitrary binary quadratic form whose discriminant is not a perfect square. For any sequence $(\alpha_n)$ of complex numbers and positive real numbers $D,N$ we have
		\[\sum_{D \leq d \leq 2D} \sum_{F(1, \nu) \equiv 0 \pmod{d}} \left \lvert \sum_{n \leq N} \alpha_n e \left(\frac{\nu n}{d} \right) \right \rvert^2 \ll_F (D + N) \sum_n \lvert \alpha_n \rvert^2. \]
	\end{proposition}
	
	It is the fact that such a strong large sieve inequality is possible for roots of quadratic congruences exists that enables such powerful results to be proved about thin variables as in \cite{FouI}, \cite{FI1}, and \cite{HBL}. We show how to derive the Type I estimates we need by following the same steps carried out in \cite{FI3} and \cite{LSX}. 
	We first replace $A_d(X), M_d(X)$ with their smooth counterparts. Consider an auxiliary smooth function $\phi : \bR \rightarrow \bR$ satisfying: 
	\begin{enumerate}
		\item $\phi(u) = 1$ if $0 < u \leq X-Y$; 
		\item $\phi^{(j)}(u) \ll Y^{-j}$ for $j = 0, 1,2$; and
		\item $\phi(u) = 0$ if $u \geq X$.
	\end{enumerate}
	Here $X^{7/8} \leq Y \leq X$ will be chosen later. We then introduce (by abuse of notation)
	\begin{equation} \label{Adp} A_d(\phi) = \sum_{n \equiv 0 \pmod{d}} a_n \phi(n)
	\end{equation}
	and
	\begin{equation} \label{Mdp} M_d(\phi) = \frac{\rho_f(d)}{d} \sum_{\gcd(\ell, d) = 1} \lambda(\ell) \frac{\varphi(d)}{d} \int_0^\infty \phi(f(\ell, t)) dt.
	\end{equation}
	We estimate the differences by elementary means as follows. Note that
	\[\sum_{d \leq D} \left \lvert A_d(X) - A_d(\phi) \right \rvert \leq \sideset{}{^\prime} \sum_{\substack{X - Y < f(m, \ell) \leq X \\ \gcd(\ell, m) = 1}} \lambda(\ell) \tau \left(f(m, \ell) \right) + O \left(\sqrt{X} \log X \right),\]
	where $\Sigma^\prime$ means that the terms with a value of $\ell$ closest to $\sqrt{X}$ are omitted. We then have the following consequence of Landreau's inequality, resulting in the bound
	\[\sideset{}{^\prime} \sum_{\ell \ll \sqrt{X}} \sum_{\substack{d \leq X^{1/4} \\ \gcd(d, \ell) = 1}} \tau(d)^8 \sum_{\substack{X - Y < f(m, \ell) \leq X \\ f(m, \ell) \equiv 0 \pmod{d}}} 1. \]
	Note that the conditions 
	\[X - Y < f(m, \ell) \leq X \text{ and } \ell \in I(X)\]
	imply that $m$ is restricted to an interval of length $O_f(Y/\sqrt{X + \ell^2})$. Splitting into residue classes $m \equiv \alpha \ell \pmod{d}$ with $\alpha$ running over the roots of (\ref{quadcong}) we see that the above sum is bounded by
	\[O \left(Y  \left(\sum_{d \leq X^{1/4}}   \tau(d)^8 \frac{ \rho_f(d)}{d} \right) \left(\sideset{}{^\prime} \sum_{\ell \ll \sqrt{X}} |\lambda(\ell)| (X + \ell^2)^{-1/2} \right)  + X^{\frac{1}{4} + \frac{1}{2r}} (\log X)^{256} \right).\]
	We have the bounds 
	\[\sum_{d \leq X^{1/4}}   \tau(d)^8 \frac{ \rho_f(d)}{d} \ll (\log X)^{256}\]
	and
	\begin{align*}\sideset{}{^\prime} \sum_{\ell \ll \sqrt{X}} |\lambda(\ell)| (X + \ell^2)^{-1/2} & \leq \sideset{}{^\prime} \sum_{k \ll X^{1/2r}} (X + k^{2r})^{-1/2} \\
		& \ll X^{\frac{1 - 2r}{4r}} \sideset{}{^\prime} \sum_{k \ll X^{1/2r}} \left(X^{\frac{1}{2r}} + k \right)^{-1/2} \\
		& \ll X^{\frac{1 - 2r + 1}{4r}} = X^{\frac{1-r}{2r}}.
	\end{align*} 
	It follows that
	\begin{equation} \sum_{d \leq D} \left \lvert A_d(X) - A_d(\phi) \right \rvert \ll YX^{\frac{1-r}{2r}} (\log X)^{256}.
	\end{equation}
	Similarly, we obtain
	\begin{equation} \sum_{d \leq D} \left \lvert M_d(X) - M_d(\phi) \right \rvert \ll YX^{\frac{1-r}{2r}} (\log X)^{256}.
	\end{equation}
	We then proceed to decompose $A_d(\phi)$ as follows:
	\begin{align*} A_d(\phi) & = \mathop{\sum \sum}_{\substack{f(m, \ell) \equiv 0 \pmod{d} \\ \gcd(\ell, m) = 1}} \lambda(\ell) \phi(f(m, \ell)) \\
		& = \sum_{f(\alpha, 1) \equiv 0 \pmod{d}} \sum_\ell \lambda(\ell) \sum_{\substack{m \equiv \alpha \ell \pmod{d} \\ \gcd(\ell, m) = 1}} \phi(f(\ell, m)) \\
		& = \sum_{f(\alpha, 1) \equiv 0 \pmod{d}} \sum_a \mu(a) \sum_\ell \gamma_{a \ell} \sum_{m \equiv a \ell \pmod{d/\gcd(a,d)}} \phi(a^2 f( m, \ell))
	\end{align*}
	where we applied M\"{o}bius inversion to the inner sum to remove the awkward co-primality condition. We then apply Poisson's formula to the inner sum to obtain
	\[\sum_{m \equiv a \ell \pmod{d/\gcd(a,d)}} = \frac{\gcd(a,d)}{d} \sum_{h \in \bZ} e \left(\alpha h \ell \frac{\gcd(a,d)}{d} \right) \Phi_{a \ell} \left(\frac{\gcd(a,d)}{d} \right), \]
	where $\Phi_{a \ell}(v)$ is the Fourier integral
	\begin{equation} \label{FourPhi} \Phi_{a \ell}(v) = \int_{-\infty}^\infty \phi( a^2 f (\ell, t)) e(-vt) dt.  \end{equation} 
	Integrating by parts we obtain
	\begin{align*} \int_{-\infty}^\infty \phi( a^2 f (\ell, t)) e(-vt) dt & = \left[\frac{-1}{2 \pi i v} e^{-2 \pi i vt} \phi(a^2 f(\ell, t)) \right]_{-\infty}^\infty + \frac{a^2}{2 \pi i v} \int_{-\infty}^\infty \phi^\prime(a^2 f(\ell, t)) (2 c t) e(-vt) dt \\
		& = \left[\frac{-a^2}{4 \pi^2 v^2} e^{-2 \pi i vt} \phi^\prime(a^2 f(\ell, t))   \right]_{-\infty}^\infty + \int_{-\infty}^\infty \frac{a^2}{4 \pi^2 v^2} e^{-2 \pi i vt} \left(\left(2c \phi^\prime + 4 a^2 c t^2 \phi^{\prime \prime} \right) \left(a^2 f(\ell, t) \right) \right)
	\end{align*}
	The zero-frequency $h = 0$ gives exactly $M_d(\phi)$, and therefore we obtain
	\[\left \lvert A_d(\phi) - M_d(\phi) \right \rvert \leq \frac{1}{d} \sideset{}{^\flat} \sum_a \sum_{\substack{bc = d \\ b | a}} \rho_f(b) b \sum_{\substack{\alpha \pmod{c} \\ f(\alpha, 1) \equiv 0 \pmod{c}}} |W_a(c, \alpha)|,\]
	where
	\begin{equation} \label{Wac} W_a(c,\alpha) = \sum_{h > 0} \sum_\ell \lambda(a \ell) e \left(\frac{\alpha h \ell}{c} \right) \Phi_{a \ell} \left(\frac{h}{c} \right)\end{equation}
	and $\Sigma^\flat$ denotes a sum over square-free integers. Summing over the moduli $d$ in a dyadic segment we obtain
	\begin{equation} \label{Adf dya}  \sum_{D < d \leq 2D} \left \lvert A_d(f) - M_d(f) \right \rvert \leq \frac{1}{D} \sideset{}{^\flat} \sum \sideset{}{^\flat} \sum_a \sum_{b | a} \rho_f(b) b V_a(D/b),
	\end{equation}
	where
	\begin{equation} \label{Vac} V_a(C) = \sum_{C < c \leq 2C} \sum_{f(\alpha, 1) \equiv 0 \pmod{c}} \left \lvert W_a(c, \alpha) \right \rvert.\end{equation}
	Next we split the outer summation of (\ref{Wac}) into dyadic ranges $H < h \leq 2H$ and we will treat these partial sums separately. By (\ref{Vac}) we obtain
	\begin{equation} \label{Vac2} V_a(C) \leq \sum_H V_a(C,H)
	\end{equation}
	where 
	\[V_a(C,H) = \sum_{C < c \leq 2C} \sum_{f(\alpha, 1) \equiv 0 \pmod{c}} \left \lvert W_a(H; c, \alpha) \right \rvert\]
	and
	\[W_a(H; c, \alpha) = \sum_{H \leq h < 2H} \sum_\ell \lambda(a \ell) e \left(\frac{\alpha h \ell}{c} \right) \Phi_{a \ell} \left(\frac{h}{c} \right).\]
	We make the substitution $t \mapsto t H/h$ in (\ref{FourPhi}). Then trivially bounding the integrand in (\ref{FourPhi}) we obtain the bound
	\[\Phi_{a \ell}(v) \ll \frac{\sqrt{X}}{a}\]
	and by integrating by parts twice, we obtain the bound
	\[\Phi_{a \ell}(v) \ll \frac{\sqrt{X}}{a} \left(\frac{a \sqrt{X}}{vY} \right)^2.\]
	This follows from our hypothesis that $\phi \ll 1, \phi^\prime \ll Y^{-1}$, and $\phi^{\prime \prime} \ll Y^{-2}$. We thus obtain the bound
	\begin{equation} \label{Gach} F_{a \ell} \left(\frac{h}{c} \right) \ll G_a(C,H) = \frac{\sqrt{X}}{a} \min \left\{1, \left(\frac{aC \sqrt{X}}{HY} \right)^2 \right\}.
	\end{equation}
	We thus obtain the bound
	\begin{equation} \label{Vach} V_a(C,H) \ll G_a(C,H) U_a(C,H),
	\end{equation} 
	where 
	\begin{equation} U_a(C,H) = \sum_{C < c \leq 2C} \sum_{f(\alpha, 1) \equiv 0 \pmod{c}} \left \lvert \sum_{H < h \leq 2H} \sum_\ell \lambda(a \ell) \xi_{h \ell} e \left(\frac{\alpha h \ell}{c} \right) \right \rvert
	\end{equation}
	with some coefficients $\xi_{h \ell}$ which do not depend on $c, \alpha$ and which are bounded by $1$ in absolute value. The terms $U_a(C,H)$ are almost of the shape which can be dealt with by Proposition \ref{BBDT prop}; all that is needed is an application of Cauchy-Schwarz. Indeed we obtain
	\begin{equation} \label{Uach} U_a(C,H) \leq \left(\sum_{C < c \leq 2C} \sum_{f( \alpha, 1) \equiv 0 \pmod{c}} 1 \right)^{1/2} \left(\sum_{C < c \leq 2C} \sum_{f( \alpha, 1) \equiv 0 \pmod{c}} \left \lvert \sum_{H \leq h < 2H} \sum_\ell \lambda(a \ell) \xi_{h \ell} e \left(\frac{\alpha h \ell}{c} \right) \right \rvert^2 \right)^{1/2}. \end{equation} 
	We then write
	\[\sum_{H \leq h < 2H} \sum_\ell \lambda(a \ell) \xi_{h \ell} e \left(\frac{\alpha h \ell}{c} \right) = \sum_n \left(\sum_{\substack{h \ell = n \\ H \leq h < 2 H}} \lambda(a \ell) \right) \xi_n e \left(\frac{\alpha n}{c} \right)\]
	and apply Proposition \ref{BBDT prop}. We then obtain:
	\begin{equation} U_a(C,H) \ll C^{1/2} (C + H \sqrt{X/a}) \left(\sum_n \left(\sum_{\substack{h \ell = n \\ H \leq h < 2 H}} \lambda(a \ell) \right)^2 \right)^{1/2}.
	\end{equation} 
	Since $a$ is square-free and $a \ell$ is an $r$-th power, it follows that $\ell = a^{r-1} m^r$ with $m \leq a^{-1} X^{\frac{1}{2r}} = M$, say. Therefore we see that the sum above is bounded by the number of solutions to
	\[h_1 m_1^r = h_2 m_2^r\]
	with $H \leq h_1, h_2 < 2H$ and $m_1, m_2 \leq M$. The solutions are parametrized by $m_1 = s t_1, m_2 = st_2$ with $\gcd(t_1, t_2) = 1$, $s t_1, s t_2 \leq M$ and $h_1 = k t_2^r, h_2 = k t_1^r$ with $k \leq 4H(t_1^r + t_2^r)^{-1}$. It follows that 
	\[\sum_n \left(\sum_{\substack{h \ell = n \\ H \leq h < 2 H}} \lambda(a \ell) \right)^2 \leq 8 H M \mathop{\sum \sum}_{t_1, t_2 \leq M} (t_1^r + t_2^r)^{-1} \leq 16 H M \sum_{t \leq M} t^{-r}. \]
	The sum on the right is maximized when $r = 1$, giving the upper bound of $O (H a^{-1} X^{1/(2r)} (\log X))$. Inserting this into (\ref{Uach}) gives
	\begin{equation} \label{Ucah2} U_a(C,H) \ll C^{\frac{1}{2}} (C + H \sqrt{X} /a )^{\frac{1}{2}} H^{\frac{1}{2}} a^{-\frac{1}{2}} X^{\frac{1}{4r}} (\log X)^{\frac{1}{2}}.  \end{equation}
	Inserting (\ref{Ucah2}) into (\ref{Vach}) we see by (\ref{Gach}) that the series (\ref{Vac2}) converges, with the largest contribution occurring when 
	\[H \asymp a C Y^{-1} \sqrt{X}.\]
	This gives the bound
	\begin{equation} \label{Vac3} V_a(C) \ll a^{-1} Y^{-1} C^{\frac{3}{2}} X^{\frac{5r + 1}{4r}} (\log X)^{\frac{3}{2}} 
	\end{equation}
	Inserting this bound into (\ref{Adf dya}) then gives 
	\begin{equation} \sum_{D < d \leq 2D} \left \lvert A_d(\phi) - M_d(\phi) \right \rvert \ll Y^{-1} D^{1/2} X^{\frac{5r +1 }{4r}} (\log X)^{\frac{5}{2}}
	\end{equation}
	This bound holds uniformly for $d \leq D$. We may thus choose 
	\[Y = D^{\frac{1}{4}} X^{\frac{7r-1}{8r}} (\log X)^{\frac{5}{4} - 128}. \]
	This in turn gives the estimate
	\[\sum_{d \leq D} \left \lvert A_d(\phi) - M_d(\phi) \right \rvert \ll D^{\frac{1}{4}} X^{\frac{3(r+1)}{8r}} (\log X)^{130},\]
	which is enough to prove Proposition \ref{Type1 prop}. 
	
	\section{Estimating $\pi(\B)$: bilinear sum bounds}
	\label{piB bisum}
	
	We will deal with the sum (\ref{p-sum}) in the case of $\pi(\B)$ via Vaughan's identity, which is an elegant combinatorial identity which decomposes the von Mangoldt function. The ideas recorded here are from \cite{FouI}. Suppose $Y,Z \geq 1$ and suppose $n > Z$. Then:
	\begin{equation} \label{Vaughan} \Lambda(n) = \sum_{\substack{m | n \\ m \leq Y}} \mu(m) \log \frac{n}{m} - \sum_{\substack{mc | n \\ m \leq Y, c \leq Z}} \mu(m) \Lambda(c) + \sum_{\substack{mc | n \\ b > Y, c > Z}} \mu(m) \Lambda(c)
	\end{equation} 
	and if $n \leq Z$, the right hand side is zero. For $X > YZ$ then Vaughan's identity implies that
	\begin{align*} P(X) & = P(Z) +  \sum_{n \leq X} b_n \left(\sum_{\substack{m | n \\ m \leq Y}} \mu(m) \log \frac{n}{m} - \sum_{\substack{mc | n \\ b \leq Y, c \leq Z}} \mu(m) \Lambda(c) + \sum_{\substack{mc | n \\ m > Y, c > Z}} \mu(m) \Lambda(c) \right) \\
		& = P(Z) + \sum_{m \leq Y} \mu(m) \left(\sum_{\substack{n \leq X \\ m | n}} b_n \log n - \sum_{\substack{n \leq X \\ m | n}} b_n \log m - \sum_{c \leq Z} \Lambda(c) \sum_{\substack{n \leq X \\ mc | n}} b_n  \right) + \sum_{m > Y} \mu(m) \sum_{c > Z} \Lambda(c) \sum_{\substack{n \leq X \\ mc | n}} b_n \\
		& = P(Z) + \sum_{m \leq Y} \mu(m) \left(A_m^\prime(X) - A_m(X) \log m - \sum_{c \leq Z} \Lambda(c) A_{mc} (X) \right) + \sum_{\substack{md \leq X \\ m > Y}} \mu(m) \left( \sum_{\substack{c | d \\ c > Z}} \Lambda(c) \right) b_{md}  \\
		& = P(Z) + A(X;Y,Z) + B(X;Y,Z),
	\end{align*}
	say. We can treat $P(Z)$ by applying trivial bounds provided that $Z$ is sufficiently small with respect to $X$. The term $A(X;Y,Z)$ can be dealt with using the appropriate Type I estimates; see Proposition \ref{Type1 prop}. The term $B(X;Y,Z)$, as expected, will require some Type II estimates. Given our treatment of the algebraic aspects of bilinear sums in Section \ref{algebra}, the treatment below is very similar to that given in \cite{FouI} and \cite{LSX} so we will be fairly terse on the details. \\
	
	Our target is the estimate
	\[B(X;Y,Z) \ll \Delta X (\log X)^5,\]
	with $\Delta = (\log X)^{-A}$ for any large, fixed $A > 5$. Recall that
	\[B(X;Y,Z) = \sum_{Z < d < X/Y} \left(\sum_{c | d, c > Z} \Lambda(c) \right) \sum_{Y < m \leq X/d} \mu(m) b_{md}.\]
	Using the trivial estimate 
	\[\sum_{c | d, c > Z} \Lambda(c) \leq \log X\]
	we then find that
	\[|B(X;Y,Z)| \leq (\log X) \sum_{d > Z} \left \lvert \sum_{Y < m \leq X/d} \mu(m) b_{md} \right \rvert.\]
	We wish to break the sum into short sums of the shape
	\begin{equation} \label{BMN} \B(M,N) = \sum_{M < m \leq 2M} \left \lvert \sum_{N < n \leq N^\prime} \mu(n) b_{mn} \right \rvert
	\end{equation}
	with $N^\prime = e^\Delta N$. Considering $M = 2^j Z$ and $N = e^{\Delta k} y$ for various $j, k$, we then see that
	\begin{equation} \label{BXYZ up} |\B(X;Y,Z)| \leq (\log X) \mathop{\sum \sum}_{\substack{\Delta X < MN < X \\ M \geq Z, N \geq Y}} \B(M,N) + O \left(\Delta X (\log X)^2 \right)
	\end{equation}
	where the error term $O(\Delta X (\log X)^2)$ represents a trivial bound for the contribution of $\mu(m) b_{md}$ with $md \leq 2 \Delta X$ or $e^{-2 \Delta} X < md \leq X$, where the terms are not covered exactly. There are at most $2\Delta^{-1} (\log X)^2$ short sums $\B(M,N)$ in (\ref{BXYZ up}) so it suffices to show that
	\begin{equation} \B(M,N) \ll \Delta^2 X (\log X)^2
	\end{equation}
	for all $M,N$ in the relevant range. We have a trivial bound
	\[\B(M,N) \leq \sum_{M < m \leq 2M} \varrho(m) \sum_{N < n \leq N^\prime} \varrho(n) \ll \Delta MN,\]
	and we can use this bound to obtain
	\[\B(M,N) \leq \sum_{d \leq \Delta^{-1}} \B_d(M,N) + O(\Delta^2 X),\]
	where $\B_d(M,N)$ consists of the sub-sum of $\B(M,N)$ where $\gcd(m,n) = d$. The error term $O(\Delta^2 X)$ comes from the trivial bound and the condition $d > \Delta^{-1}$. Next observe that
	\[\B_d(M,N) \leq \B_1(dM, N/d),\]
	and so it suffices to show
	\begin{equation} \B_1(M,N) \ll \Delta^3 X (\log X)^2
	\end{equation}
	for $M,N$ satisfying $M \geq Z, N \geq \Delta Y$ and $\Delta X < MN < X$. \\ 
	
	Applying (\ref{key decomp}) to (\ref{BMN}) we then obtain 
	\[\B_1(M,N) \leq \sum_{A \cdot B = \Cl f}  \sum_{\substack{\Bm \in A \\ M < N(J(\Bm)) \leq 2M  }} \left \lvert  \sum_{\substack{\Bn \in B \\ N < N(J(\Bn)) \leq N^\prime \\ (\widehat{\alpha_1}; \widehat{\alpha_2}) \in \K_{\beta_j}^\dagger}} \mu(N(J(\Bn)) \Lambda \left(Q_{A,B}(\Bm; \Bn) \right)  \right \rvert. \]
	Removing the co-primality condition via M\"{o}bius inversion as in \cite{FouI} and \cite{LSX}, as well as partitioning the sum $\B_1(M,N)$ based on the classes $A,B$, it suffices to show that the sums 
	\begin{equation} \label{CrMN} \C_r(M,N) = \sum_{\substack{M < g_1(x_1, x_2) \leq 2M \\ (x_1, x_2) \in \K_1}} \left \lvert \sum_{\substack{N < g_2(y_1, y_2) \leq N^\prime \\ (y_1, y_2) \in \K_2}} \mu(r g_2(y_1, y_2)) \Lambda (Q(x_1, x_2; y_1, y_2)) \right \rvert \end{equation}
	are bounded by $O\left(\Delta^5 X (\log X)^2 \right)$ for every $r, M,N$ satisfying
	\[r < \Delta^{-2}, M \geq Z, N \geq \Delta^3 Y \text{ and } \Delta X < MN < X\]
	and $\K_1, \K_2$ domains which are contained in $[-CX, CX]^2$ for some absolute constant $C$ depending only on our choices of fundamental domains. \\ \\
	If we write
	\[Q(x_1, x_2; y_1, y_2) = x_1 \ell_1(y_1, y_2) + x_2 \ell_2(y_1, y_2)\]
	for linear forms $\ell_1, \ell_2 \in \bZ[x,y]$ then the condition that $Q(\Bx; \By) = 0$ implies that $(\ell_1(y_1, y_2), \ell_2(y_1, y_2))$ is proportional to $(-x_2, x_1)$. We then make a change of variables in the inner sum, obtaining 
	\[\C_r(M,N) =  \sum_{\substack{M < g_1(x_1, x_2) \leq 2M \\ (x_1, x_2) \in \K_1}} \left \lvert \sum_{\substack{N < g_2^\ast(z_1, z_2) \leq N^\prime \\ (y_1, y_2) \in \K_2}} \mu(r g_2^\ast(z_1, z_2)) \Lambda (x_1 z_1 + x_2 z_2) \right \rvert \]
	where $z_i = \ell_i(y_1, y_2)$ and $g_2^\ast$ is such that $g_2^\ast(z_1, z_2) = g_2(y_1, y_2)$. 
	We are then left with the bilinear sum (\ref{gen bisum}) 
	where $\alpha$ is supported in a disk of radius $R_1$ and $\beta$ supported on an annulus $\bA(R_2,2R_2)$ having inner radius $R_2$ and outer radius $2R_2$, say. Further, we assume that $\lambda$ is supported on $|\ell| \leq CAB$ for some absolute constant $C$ depending only on $f$, so in particular the $\ell^2$-norm of $\lambda$ is finite. Applying the Cauchy-Schwarz inequality we obtain
	\begin{equation} \label{CAB bd1} \left \lvert \C(\alpha, \beta; \lambda) \right \rvert \leq \sum_{\ell} |\lambda(\ell)| \sideset{}{^\ast} \sum_{\By} |\beta(\By)| \left \lvert \sum_{\Q(\Bx; \By) = \ell} \alpha(\Bx) \right \rvert \leq \lVert \lambda \rVert_2 \cdot \lVert \beta \rVert_2 \D(\alpha)^{1/2},
	\end{equation}
	where $\lVert \cdot \rVert_2$ denotes the $\ell^2$-norm and 
	\[\D(\alpha) = \sideset{}{^\ast} \sum_{\By} \G(\By) \sum_\ell \left \lvert \sum_{\Q(\Bx; \By) = \ell} \alpha(\Bx) \right \rvert^2\]
	with $\G$ is any non-negative function with $\G(\By) \geq 1$ on the annulus $\bA(R_2, 2R_2)$. As in \cite{FouI} and \cite{LSX} it will be convenient to suppose that $\G$ is a radial, compactly supported, and smooth function. Squaring out we obtain
	\begin{equation} \label{dalpha} \D(\alpha) = \sideset{}{^\ast} \sum_{\By} \G(\By) \sum_{\Q(\Bx; \By) = 0} (\alpha \ast \alpha)(\Bx),
	\end{equation}
	with
	\[(\alpha \ast \alpha)(\Bx) = \sum_{\Bu - \Bv = \Bx} \alpha(\Bu) \ol{\alpha}(\Bv).\]
	Note that
	\[(\alpha \ast \alpha)(0) = \lVert \alpha \rVert_2^2.\]
	The orthogonality relation $\Bx \cdot \By = 0$ for a primitive $\Bx$ in (\ref{dalpha}) is equivalent to the statement that $\By$ is a rational integer multiple of $\Bx^\prime = (-x_2, x_1)$. It follows that
	\begin{equation} \D(\alpha) = \sum_{c \in \bZ} \sideset{}{^\ast} \sum_{\By} \G(\By) (\alpha \ast \alpha)(c \By) = \D_0(\alpha) + 2 \D^\ast(\alpha),
	\end{equation}
	where $\D_0(\alpha)$ denotes the contribution with $c = 0$ and $\D^\ast(\alpha)$ that of all $c > 0$. Thus 
	\[\D_0(\alpha) = \lVert \alpha \rVert_2^2 \sideset{}{^\ast} \sum_{\By} \G(\By) \ll \lVert \alpha \rVert_2^2 B^2\]
	and
	\[\D^\ast(\alpha) = \sum_{\Bx \ne \mathbf{0}} \G(\Bx^\ast) (\alpha \ast \alpha)(z),\]
	where $\Bx^\ast$ is a primitive vector proportional to $\Bx$. Again, we may apply M\"{o}bius inversion to remove the primitivity conditions, and obtain
	\[\D^\ast(\alpha) = \mathop{\sum \sum}_{b, c > 0} \mu(b) \D(\alpha; bc) \]
	where
	\[\D(\alpha; bc) = \sum_{\Bx \equiv 0 \pmod{bc}} \G(c^{-1} \Bx) (\alpha \ast \alpha)(\Bx).\]
	From here, the treatment is identical to the one given in \cite{FouI} and \cite{LSX} as no structure of the Gaussian integers or even an imaginary quadratic field is necessary. This completes our treatment for $\pi(\B)$. 
	
	\section{Type II estimates for $\pi(\A) - \pi(\B)$: preliminary steps}
	\label{piApiB} 
	
	The goal of this section is to discuss the proof of Proposition \ref{main bisum}. We note that Proposition \ref{main bisum} is exactly analogous to Proposition 5 in \cite{HBL}, though our sequences $\A, \B$ are different. We note that we have largely divorced the arithmetic of our field $K$ with the analysis of bilinear sums in Section \ref{algebra}, and so we are in good shape to import results from \cite{HBL} directly. We will make clear which components of \cite{HBL} can be used without change, and where we need to make suitable modifications. \\
	
	We substitute (\ref{key decomp}) into (\ref{bisum prop eq}) to obtain
	\begin{equation} \label{bisum trans1} \sum_{N < n \leq 2N} \sum_{m < X/N} \alpha_m \beta_n (a_{mn} - b_{mn}) \end{equation} 
	\[= \sum_{A \cdot B = \Cl f}  \sum_{\substack{w \in A_0 \\ N < N(J(w)) \leq 2N }} \beta_w \sum_{\substack{v \in B_0 \\ N(J(v)) < X/N}} \alpha_v (\fZ(Q_{A,B}(v, w)) - \Lambda(Q_{A,B}(v,w)),  \]
	where $\alpha_v = \alpha_{N(J(v))}, \beta_w = \beta_{N(J(w))}$. Writing each bilinear form $Q$ above as $w_1 \ell_1(v_1, v_2) + w_2 \ell_2(v_1, w_2)$ say and applying a linear change of variables to the inner sum, we transform the inner sum 
	\[\sum_{\substack{v \in B_0 \\ N(J(v)) < X/N}}  \alpha_v (\fZ(Q(v, w)) - \lambda(Q(v,w)) = \sum_\Bz \alpha_{\Bz} (\fZ(w_1 z_1 + w_2 z_2) - \Lambda(w_1 z_1 + w_2 z_2))\]
	say, with the support of $\Bz$ being the image of the support of the sum on the left under the linear transformation. Note that the linear transformation depends only on $Q$ and not $X$. \\ \\
	After applying these linear transformations, we have now changed all of our bilinear forms $Q$ to 
	\[Q_0(x_1, x_2; y_1, y_2) = x_1 y_1 + x_2 y_2.\]
	 Let us write $\S_1(X) \times \S_2(X)$ for the union of the images of the supports of $w, v$ in (\ref{bisum trans1}), so that (\ref{bisum trans1}) becomes
	\begin{equation} \label{bisum trans2} h(K) \sum_{\Bw \in \S_1(X)} \sum_{\Bv \in \S_2(X)} \alpha_{\Bw} \beta_{\Bv} (\fZ(w_1 v_1 + w_2 v_2) - \Lambda(w_1 v_1 + w_2 v_2)).
	\end{equation}

	\begin{remark}
		Since the linear transformations depend only on the class $1 \leq j \leq h(K)$ and the corresponding choice of fundamental domain, the image of the set $\F_j(X)$ with $N < N(J(w)) \leq 2N$ is contained in the annulus $\bA(c_1 N, c_2 N)$ for some positive numbers $c_1, c_2$ independent of $N$. Similarly, the image of $\F_j^\prime(X)$ with $N(J(v)) \leq X/N$ is contained in the disk $\bD(c_3 X/N)$ for some $c_3 > 0$ depending at most on $f$. This observation is crucial because we will use the Euclidean norm and the corresponding geometry to treat our sums when we wish to import estimates from \cite{FI1} and \cite{HBL}, and switch to using the norm on $\O_K$ and the corresponding induced norm on ideal numbers when the arithmetic of $K$ is relevant.  \end{remark} 
	Since we are looking to save an arbitrary power of $\log$, it suffices to further subdivide the support of (\ref{bisum trans2}), and consider sums of the shape
	\[\sum_{\substack{\Bw \\ N < \lVert \Bw \rVert_2 \leq 2N}} \alpha_{\Bw} \sum_{\substack{\Bz \\ \lVert \Bz \rVert_2 \leq X/N}} \beta_{\Bz} \left(\fZ(w_1 z_1 + w_2 z_2) - \Lambda(w_1 z_1 + w_2 z_2) \right). \] 
	\begin{remark} We abuse notation and refer to the terms $\beta_n$ for some positive integer $n$ as well as $\beta_\Bz$ for some vector $\Bz \in \bZ^2$. In the former case we interpret the support of $\beta_n$ to be a set of ideal numbers of $\O_K$ in a fixed class having norm equal to $n$, and in the latter we simply interpret the set of ideal numbers as a $\bZ$-module. 
	\end{remark} 
	Put
	\begin{equation} S_1(\Bz, \Bw) = \sum_{\substack{p^2 \in I \\ w_1 z_1 + w_2 z_2 = p^2}} 2 p \log p \text{ and } S_2(\Bz, \Bw) = \sum_{\substack{p \in I \\ w_1 z_1 + w_2 z_2 = p }} \log p  
	\end{equation}
and
\[S_1^\spadesuit(\Bz, \Bw) = \sum_{\substack{k^2 \in I(X) \\ w_1 z_1 + w_2 z_2 = k}} 2k \text{ and } S_2^\spadesuit(\Bz, \Bw) = \sum_{\substack{k \in I \\ w_1 z_1 + w_2 z_2 = k}} 1.\]
	
	Our aim is to obtain the estimates
	\begin{equation} \label{bisum eq3} \sum_{\substack{\Bw \\ N < \lVert \Bw \rVert_2 \leq 2N}} \sum_{\substack{\Bz \\ \lVert \Bz \rVert_2 \leq X/N}} \alpha_{\Bw} \beta_{\Bz} (S_1(\Bz, \Bw) - S_2(\Bz, \Bw)) \ll_A \frac{X}{(\log X)^A}
	\end{equation} 
and 
\[\sum_{\substack{\Bw \\ N < \lVert \Bw \rVert_2 \leq 2N}} \sum_{\substack{\Bz \\ \lVert \Bz \rVert_2 \leq X/N}} \alpha_{\Bw} \beta_{\Bz} \left(S_1^\spadesuit(\Bz, \Bw) - S_2^\spadesuit(\Bz, \Bw) \right) \ll_A \frac{X}{(\log X)^A}.\]
	We are almost ready to import the remaining argument from \cite{HBL}. Let us put
	\[\R(N; X) = \left \{\Bz \in \bZ^2 : N \leq \lVert \Bz \rVert_2 < 2N, |\arg(\Bz) - k\pi/2| \leq (\log X)^{-A} \forall k \in \bZ \right\} \]
	We note that, as we will use repeatedly later (and we will remind the reader again of this when this becomes relevant), that once we subdivide the regions into small dyadic ranges that the conditions $\lVert \Bz \rVert_2 \sim N$ and $N(z) \sim N$ are nearly identical. Here $\Bz = \widehat{z}$ is the vector associated to $z$, viewed as an ideal number of $K$. \\ \\
	The following results from \cite{HBL} can now be imported without change:
	\begin{lemma}[Lemma 9, \cite{HBL}] \label{HBL-L9} Suppose that both $\Bz$ and $q$ are fixed. Then the number of possible $\Bw$ with $q = w_1 z_1 + w_2 z_2$ is $O((M/N)^{1/2})$. 
	\end{lemma} 
	\begin{lemma}[Lemma 10, \cite{HBL}] \label{HBL-L10} We have
		\[\sum_{\Bz \in \R(N;X)} \sum_{\Bw} \beta_{\Bz} \alpha_{\Bw} S_j(\Bz, \Bw) \ll_A X(\log X)^{-A}\]
		for $j = 1,2$. 
	\end{lemma}

We remark that Lemma \ref{HBL-L10} apply equally well with $S_j(\Bz, \Bw)$ replaced with $S_j^\spadesuit(\Bz, \Bw)$. \\

	As is standard at this juncture (see \cite{FouI}, \cite{FI1}, and \cite{HBL}), we apply Cauchy-Schwarz to obtain
	\[\left(\sum_{\Bw} \alpha_\Bw \sum_{\Bz} \beta_{\Bz} (S_1(\Bz, \Bw) - S_2(\Bz, \Bw)) \right)^2 \leq \sum_\Bw \alpha_\Bw^2 \sum_\Bw \left(\sum_{\Bz} \beta_{\Bz} (S_1(\Bz, \Bw) - S_2(\Bz, \Bw)) \right)^2.\]
	It is then sufficient to show that
	\begin{equation} \label{bisum cancel0} \sum_{\By, \Bz} \beta_{\By} \beta_{\Bz} \sum_\Bw (S_1^\spadesuit(\By, \Bw) - S_2^\spadesuit(\By, \Bw))(S_1^\spadesuit(\Bz, \Bw) - S_2^\spadesuit(\Bz, \Bw)) \ll_A \frac{XN}{(\log X)^A}\end{equation} 
	and
	\begin{equation} \label{bisum cancel1} \sum_{\By, \Bz} \beta_{\By} \beta_{\Bz} \sum_\Bw (S_1(\By, \Bw) - S_2(\By, \Bw))(S_1(\Bz, \Bw) - S_2(\Bz, \Bw)) \ll_A \frac{XN}{(\log X)^A}\end{equation} 
	for any $A > 0$. \\ \\
	Next we consider the diagonal contribution coming from $\By = \Bz$. This gives the sums
	\[\sum_{\Bz} \beta_{\Bz} \sum_{\Bw} \alpha_\Bw \left(S_1^\spadesuit(\Bz, \Bw) - S_2^\spadesuit(\Bz, \Bw)\right)^2 = \sum_{\Bz} \beta_{\Bz} \sum_\Bw \alpha_\Bw \left(S_1^\spadesuit(\Bz, \Bw)^2 - 2 S_1^\spadesuit(\Bz, \Bw) S_2^\spadesuit(\Bz, \Bw) + S_2^\spadesuit(\Bz, \Bw)^2\right)\]
	and
	\[\sum_{\Bz} \beta_{\Bz} \sum_{\Bw} \alpha_\Bw (S_1(\Bz, \Bw) - S_2(\Bz, \Bw))^2 = \sum_{\Bz} \beta_{\Bz} \sum_\Bw \alpha_\Bw (S_1(\Bz, \Bw)^2 - 2 S_1(\Bz, \Bw) S_2(\Bz, \Bw) + S_2(\Bz, \Bw)^2).\]
	Clearly, 
	\[S_1(\Bz, \Bw) S_2(\Bz, \Bw) = S_1^\spadesuit(\Bz, \Bw) S_2^\spadesuit(\Bz,\Bw) = 0 \] since their supports are incompatible. Next we have the trivial estimate
	\begin{align*} \sum_{\Bz} \sum_{\Bw} S_1(\Bz, \Bw) & \ll \sum_{\substack{N < \lVert \Bz \rVert_2 \leq 2N}} \sum_{p^2 \in I} p \log p \sum_{\substack{\Bw \\ w_1 z_1 + w_2 z_2 = p^2 }} 1 \\
		& \ll \sqrt{\frac{M}{N}} \sum_{p^2 \in I} p \log p \sum_{N \leq \lVert \Bz \rVert_2 < 2N} 1 \\
		& \ll \sqrt{MN} \sum_{p^2 \in I} p \log p \\
		& \ll_\ep \sqrt{MN} X^{1/2 + \ep} \ll_\ep X^{1 + \ep}.
	\end{align*} 
	Similarly, we conclude
	\[\sum_{\Bz} \sum_{\Bw} S_2(\Bz, \Bw) \ll_\ep X^{1 + \ep}, \]
	\[\sum_{\Bz} \sum_{\Bw} S_1^\spadesuit(\Bz, \Bw) \ll_\ep X^{1 + \ep}, \]
	\[\sum_{\Bz} \sum_{\Bw} S_2^\spadesuit(\Bz, \Bw) \ll_\ep X^{1 + \ep}. \]
	From here we obtain
	\begin{align*} \sum_{\Bz} \sum_{\Bw} S_1(\Bz, \Bw)^2 + S_2(\Bz, \Bw)^2 & \ll X^{1/4} \log X \sum_{\Bz} \sum_{\Bw} S_1(\Bz, \Bw) + \log X \sum_{\Bz} \sum_{\Bw} S_2(\Bz, \Bw) \\
		& \ll_\ep X^{5/4 + \ep}.
	\end{align*}
and 
\[\sum_{\Bz} \sum_{\Bw} S_1^\spadesuit(\Bz, \Bw)^2 + S_2^\spadesuit(\Bz, \Bw)^2 \ll_\ep X^{5/4 + \ep}.\]
	At this stage, we expunge the references to the Gaussian domain $\bZ[i]$ in \cite{HBL} to make it clear that much of their treatment of bilinear sums apply equally well in our situation, despite the fact that our number field is different from $\bQ(i)$. For $\By, \Bz \in \bZ^2$ put $\Delta(\By, \Bz) = y_1 z_2 - y_2 z_1$. Given $\Bw, \By, \Bz \in \bZ^2$ such that
	\[w_1 y_1 + w_2 y_2 = q_1 \text{ and } w_1 z_1 + w_2 z_2 = q_2,\]
	we have
	\[\begin{bmatrix} y_1 &  y_2 \\ z_1 &  z_2 \end{bmatrix} \begin{bmatrix} w_1 \\ w_2 \end{bmatrix} = \begin{bmatrix} q_1 \\ q_2 \end{bmatrix}. \]
	Inverting the matrix on the left we see that
	\[\begin{bmatrix} w_1 \\ w_2 \end{bmatrix} = \frac{1}{\Delta(\Bz, \By)} \begin{bmatrix} z_2 & -y_2 \\ -z_1 & y_1 \end{bmatrix} \begin{bmatrix} q_1 \\ q_2 \end{bmatrix}. \]
	Since $\Bw = (w_1, w_2) \in \bZ^2$, it follows that
	\begin{equation} \label{qcong1} q_1 z_2 - q_2 y_2 \equiv q_1 z_1 - q_2 y_1 \equiv 0 \pmod{\Delta(\Bz, \By)}.\end{equation}
	Let $C(q_1, q_2, \Bz, \By)$ be the statement that $q_1, q_2, \Bz, \By$ satisfy (\ref{qcong1}). Next we have
	\begin{align} \label{HBL cond1}  \lVert q_1 (z_1, z_2) - q_2(y_1, y_2) \rVert_2 & = \sqrt{(q_1 z_1 - q_2 y_1)^2 + (q_1 z_2 - q_2 y_1)^2} \\
		& = \sqrt{\left(w_1 \Delta(\Bz, \By) \right)^2 + \left(w_2 \Delta(\Bz, \By) \right)^2} \notag \\
		& = \Delta(\Bz, \By) \sqrt{w_1^2 + w_2^2} \leq \Delta(\Bz, \By) M. \notag 
	\end{align}
	We also wish to impose the condition that $\Delta(\Bz, \By)$ is small. In particular, we wish to only consider those $\Bz, \By$ with 
	\begin{equation} \label{fD0} \Delta(\Bz, \By) > \fD_0 = N(\log X)^{-A - 6}.
	\end{equation} 
	For brevity, let us write
	\[h^\dagger(q) = \begin{cases} 2 p \log p & \text{if } q = p^2 \in I(X)  \\ \\ 0 & \text{otherwise}, \end{cases} \]
	\[h^\ddagger(q) = \begin{cases}  \log p & \text{if } q = p \in I(X) \\ \\ 0 & \text{otherwise}, \end{cases}\]
	and
	\[h(q) = h^\dagger(q) - h^\ddagger(q).\]
	Similarly, let us write
	\[h^{\spadesuit, \dagger}(q) = \begin{cases} 2 p \log p & \text{if } q = p^2 \in I(X)  \\ \\ 0 & \text{otherwise}, \end{cases} \]
	\[h^{\spadesuit, \ddagger}(q) = \begin{cases}  \log p & \text{if } q = p \in I(X) \\ \\ 0 & \text{otherwise}, \end{cases}\]
	and
	\[h^\spadesuit(q) = h^{\spadesuit, \dagger}(q) - h^{\spadesuit, \ddagger}(q).\]
	
	Note that for any subinterval $J \subset I(X)$ we have
	\[\sum_{q \in J} h(q) = O_C \left(\frac{X^{1/4}}{(\log X)^C} \right)\]
	for any $C > 0$. This is a consequence of our choice of weights. \\ \\
	As in \cite{HBL}, we want to carve up the support of $\Bz, \By$ into regions of the form
	\begin{equation} \label{U shape} \UUU = \UUU(c, \theta_0) = \{\Bz : c \sqrt{N} < \lVert \Bz \rVert_2 \leq c(1 + \omega_1) \sqrt{N}, \theta_0 < \arg(\Bz) \leq \theta_0 + \omega_2\},\end{equation} 
	for fixed $1 \leq c \leq \sqrt{2}$ and $\theta_0$. Note that we may choose $\omega_1$ and $\omega_2$ so that the regions $\UUU$ form a partition of the region 
	\[\{\Bz : N \leq \lVert \Bz \rVert_2 < 2N, z_1 > 0\} \setminus \R.\]
	The number of regions needed for the sum over $\Bz, \By$ is $O((\log X)^{4L})$. Here, as in \cite{HBL}, we allow the parameters $\omega_1$ and $\omega_2$ to be different in order to perfectly cover our region. They have the same order of magnitude. \\ 
	
	As in \cite{HBL} let us write $\fC_1(\UUU_1, \UUU_2, J_1, J_2)$ as the condition that all $(\Bz, \By, q_1, q_2) \in \UUU_1 \times \UUU_2 \times J_1 \times J_2$ satisfy (\ref{HBL cond1}) and (\ref{fD0}). We remark that such tuples are the most intricate to estimate; in fact it is only in the treatment of these tuples where we must diverge from the argument given in \cite{HBL}. \\ \\
	Similarly, let $\fC_2(\UUU_1, \UUU_2, J_1, J_2)$ denote the condition that there exists some tuple $(\Bz, \By, q_1, q_2) \in \UUU_1 \times \UUU_2 \times J_1 \times J_2$ which satisfies (\ref{HBL cond1}) and there exists some tuple $(\Bz^\prime, \By^\prime, q_1^\prime, q_2^\prime) \in \UUU_1 \times \UUU_2 \times J_1 \times J_2$ which does not satisfy (\ref{HBL cond1}). Finally, let $\fC_3(\UUU_1, \UUU_2, J_1, J_2)$ be the condition that all tuples $(\Bz, \By, q_1, q_2) \in \UUU_1 \times \UUU_2 \times J_1 \times J_2$ satisfy (\ref{HBL cond1}) but there exists some tuple $(\Bz, \By, q_1, q_2) \in \UUU_1 \times \UUU_2 \times J_1 \times J_2$ which does not satisfy (\ref{fD0}). \\ \\
	Recall that $C(q_1, q_2, \Bz, \By)$ is the condition that $\Bz, \By, q_1, q_2$ satisfy (\ref{qcong1}). For $\UUU_1, \UUU_2, J_1, J_2$ satisfying $\fC_1(\UUU_1, \UUU_2, J_1, J_2)$ put
	\begin{equation} T(\UUU_1, \UUU_2, J_1, J_2) = \sideset{}{^\flat} \sum_{\substack{\Bz \in \UUU_1 \\ \By \in \UUU_2}} \beta_{\Bz} \beta_{\By} \sum_{\substack{q_1 \in J_1 \\ q_2 \in J_2 \\ C(q_1, q_2, \Bz, \By)}} h(q_1) h(q_2),
	\end{equation}
	and otherwise set $T(\UUU_1, \UUU_2, J_1, J_2) = 0$. Further, let 
	\begin{equation} T^\prime(\UUU_1, \UUU_2, J_1, J_2) = \sideset{}{^\flat} \sum_{\substack{\Bz \in \UUU_1 \\ \By \in \UUU_2}}  \sum_{\substack{q_1 \in J_1 \\ q_2 \in J_2 \\ C(q_1, q_2, \Bz, \By)}} |h(q_1) h(q_2)|.
	\end{equation}
Similarly, define
\[T_\spadesuit(\UUU_1, \UUU_2, J_1, J_2) \text{ and } T_\spadesuit^\prime(\UUU_1, \UUU_2, J_1, J_2)\]
analogously with $h$ replaced with $h^\spadesuit$. Then to obtain (\ref{bisum cancel0}) and (\ref{bisum cancel1}) it suffices to show that 
\begin{equation}\sum_{\substack{\UUU_1, \UUU_2, J_1, J_2 \\ \fC_1(\UUU_1, \UUU_2, J_1, J_2}}  T_\spadesuit(\UUU_1, \UUU_2, J_1, J_2) + \sum_{\substack{\UUU_1, \UUU_2, J_1, J_2 \\ \fC_2(\UUU_1, \UUU_2, J_1, J_2) \text{ or } \fC_3(\UUU_1, \UUU_2, J_1, J_2}} T_\spadesuit^\prime(\UUU_1, \UUU_2, J_1, J_2) \ll_A \frac{XN}{(\log X)^A}
\end{equation}
and
	\begin{equation} \label{HBL sum} \sum_{\substack{\UUU_1, \UUU_2, J_1, J_2 \\ \fC_1(\UUU_1, \UUU_2, J_1, J_2}}  T(\UUU_1, \UUU_2, J_1, J_2) + \sum_{\substack{\UUU_1, \UUU_2, J_1, J_2 \\ \fC_2(\UUU_1, \UUU_2, J_1, J_2) \text{ or } \fC_3(\UUU_1, \UUU_2, J_1, J_2}} T^\prime(\UUU_1, \UUU_2, J_1, J_2) \ll_A \frac{XN}{(\log X)^A}.
	\end{equation}
	As in \cite{HBL}, we will show that the contribution from $\fC_i(\UUU_1, \UUU_2, J_1, J_2)$ is negligible for $i = 2,3$. Indeed, we shall obtain:
	\begin{proposition} \label{HBL prop6} We have 
		\[ \sum_{\substack{\UUU_1, \UUU_2, J_1, J_2 \\ \fC_2(\UUU_1, \UUU_2, J_1, J_2) \text{ or } \fC_3(\UUU_1, \UUU_2, J_1, J_2)}} T_\spadesuit^\prime(\UUU_1, \UUU_2, J_1, J_2) \ll_A \frac{XN}{(\log X)^A} \]
		and
		\[ \sum_{\substack{\UUU_1, \UUU_2, J_1, J_2 \\ \fC_2(\UUU_1, \UUU_2, J_1, J_2) \text{ or } \fC_3(\UUU_1, \UUU_2, J_1, J_2)}} T^\prime(\UUU_1, \UUU_2, J_1, J_2) \ll_A \frac{XN}{(\log X)^A}. \]
	\end{proposition}
	In fact, Proposition \ref{HBL prop6} is exactly analogous to Proposition 6 in \cite{HBL}. More strikingly, the proof does not need to be modified and we can simply apply Proposition 6 of \cite{HBL}. However, given that our set-ups are not identical we will explain why our situations are indeed interchangeable. \\ \\
	We will also need the following analogue of Proposition 7 in \cite{HBL}: 
	\begin{proposition} \label{HBL prop7} For fixed $J_1, J_2$ and $L = 6A + 52$ we have
		\[\sum_{\substack{\UUU_1, \UUU_2 \\ \fC_1(\UUU_1, \UUU_2, J_1, J_2)}}  T_\spadesuit(\UUU_1, \UUU_2, J_1, J_2) \ll_A \frac{XN}{(\log X)^{A + 2L}}.\]
		and
		\[\sum_{\substack{\UUU_1, \UUU_2 \\ \fC_1(\UUU_1, \UUU_2, J_1, J_2)}}  T(\UUU_1, \UUU_2, J_1, J_2) \ll_A \frac{XN}{(\log X)^{A + 2L}}.\]
	\end{proposition}
	Unlike Proposition \ref{HBL prop6} we cannot simply import Proposition 7 from \cite{HBL}. This is because Proposition \ref{HBL prop6}, by the definition of $T^\prime(\UUU_1, \UUU_2, J_1, J_2)$, is insensitive to the nature of the coefficients $\beta_\Bz$ and so the treatment in \cite{HBL} is directly applicable to our situation. However in order to prove Proposition 7 in \cite{HBL} they needed to use the specific shape of $\beta_z$ in their paper. That said, the modifications needed to adapt their proof to our case are minor, and we will still be able to follow their argument for the most part. \\ \\
	In the next few sections we will give proofs for Propositions \ref{HBL prop6} and \ref{HBL prop7}. We will largely follow the structure of the argument given in \cite{HBL}. 
	
	\section{Proof of Propositions \ref{HBL prop6}}
	\label{HBL sec7}  
	
	First we have the following lemma, which is Lemma 12 from \cite{HBL}: 
	
	\begin{lemma} \label{gcd lem} The bound
		\[
		\sideset{}{^\flat} \sum_{\Bz, \By} \sum_{\substack{q_1 \in J_1, q_2 \in J_2 \\ C(q_1, q_2, \Bz, \By) \\ \gcd(q_1 q_2, \Delta(\Bz, \By)) > 1}} \left \lvert h^\spadesuit(q_1) h^\spadesuit(q_2) \right \rvert \ll N^2 \sqrt{X} (\log X)^3\]
		and
		\[
		\sideset{}{^\flat} \sum_{\Bz, \By} \sum_{\substack{q_1 \in J_1, q_2 \in J_2 \\ C(q_1, q_2, \Bz, \By) \\ \gcd(q_1 q_2, \Delta(\Bz, \By)) > 1}} | h(q_1) h(q_2)| \ll N^2 \sqrt{X} (\log X)^3\]
		holds.
	\end{lemma}
	
	\begin{proof} See Section 7, \cite{HBL}. 
	\end{proof}
	
	Lemma \ref{gcd lem} allows us, as in \cite{HBL}, to write
	\begin{equation} T(\UUU_1, \UUU_2, J_1, J_2) = \sum_{D \leq 2N} \sideset{}{^\ast} \sum_{a \pmod{D}} \Y(a, D; h, h) \Z(a, D) + O \left( N^2 \sqrt{X} (\log X)^3\right)
	\end{equation}
	where
	\[\Z(a,D) = \sideset{}{^\flat} \sum_{\substack{(\Bz, \By) \in \UUU_1 \times \UUU_2 \\ \Delta(\Bz, \By) = D \\ a \By \equiv \Bz \pmod{D}}} \beta_{\Bz} \beta_{\By} \]
	and
	\[\Y(a, D; h_1, h_2) = \sum_{\substack{q_1 \in J_1, q_2 \in J_2 \\ q_1 \equiv a q_2 \pmod{D} \\ \gcd(q_1 q_2, D) = 1}} h_1(q_1) h_2(q_2).\]
	Similarly, we have
	\begin{equation} T_\spadesuit(\UUU_1, \UUU_2, J_1, J_2) = \sum_{D \leq 2N} \sideset{}{^\ast} \sum_{a \pmod{D}} \Y^\spadesuit(a, D; h^\spadesuit, h^\spadesuit) \Z(a, D) + O \left( N^2 \sqrt{X} (\log X)^3\right)
	\end{equation}
where 
\[\Y^\spadesuit(a, D; h_1^\spadesuit, h_2^\spadesuit) = \sum_{\substack{q_1 \in J_1, q_2 \in J_2 \\ q_1 \equiv a q_2 \pmod{D} \\ \gcd(q_1 q_2, D) = 1}} h_1^\spadesuit(q_1) h_2^\spadesuit(q_2) \]
	This crucial decomposition allows us to separate $T(\UUU_1, \UUU_2, J_1, J_2)$ and $T_\spadesuit(\UUU_1, \UUU_2, J_1, J_2)$ into components $\Z(a,D)$ containing the coefficients $\beta_{\Bz}, \beta_{\By}$ and a congruence sum which no longer has anything to do with the coefficients $\beta$. To treat (\ref{HBL sum}) requires a treatment of $\Y(a, D)$ involving primes. For this purpose they needed a refinement of the Barban-Davenport-Heilbronn theorem, which we will not go into more detail here as we can use their Proposition 6 directly.\\
	
	 The following lemma is critical to the proof of Proposition \ref{HBL prop6}:
	
	\begin{lemma} \label{HBL lem13} Let 
		\[\widetilde{\Z}(a, D) = \sideset{}{^\flat} \sum_{\substack{(\Bz, \By) \in \UUU_1 \times \UUU_2 \\ \Delta(\Bz, \By) = D \\ a \By \equiv \Bz \pmod{D}}} 1.\]
		We then have the bounds
		\begin{equation} \sum_D \tau(D) \sideset{}{^\ast} \sum_{a \pmod{D}} \widetilde{\Z}(a,D) \ll \omega^4 N^2 (\log X)^{16}, 
		\end{equation} 
		\begin{equation} \sum_{\UUU_1, \UUU_2} \sideset{}{^\ast} \sum_{a \pmod{D}} \widetilde{\Z}(a,D) \ll N,
		\end{equation}
		and
		\begin{equation} \sideset{}{^\ast} \sum_{a \pmod{D}} \widetilde{\Z}(a,D) \ll (\log X)^3 \frac{N^2}{D} \tau(D)^6.
		\end{equation}
	\end{lemma}
	
	\begin{proof} See Lemma 13 in \cite{HBL}. 
	\end{proof}
	
	For an interval $J$ and a function $\fh$, put
	\[\Y(J, \fh; D) = \sum_{\substack{q \in J \\ \gcd(q,D) = 1}} \fh(q)\]
	and
	\[\Y_{\fh_1, \fh_2}(D) = Y(D) = \frac{1}{\varphi(D)} \Y(J_1, \fh_1; D) \Y(J_2, \fh_2; D).\]
	Recall that $q_1, q_2$ appearing in $\Y(a, D; \fh_1, \fh_2)$ satisfy $\gcd(q_1 q_2, D) = 1$. If $h_1$ or $h_2$ is equal to $h^\ddagger$, then $\Y(D)$ is the expected value of $\Y(a, D; h_1, h_2)$. If $h_1 = h_2 = h^\dagger$, note that $p_1^2 \equiv a p_2^2 \pmod{D}$ implies that $p_1 \equiv b p_2 \pmod{D}$ for some $b$ such that $a \equiv b^2 \pmod{D}$. Here, $\Y(a, D; h_1, h_2) = 0$ if $a$ is not a square modulo $D$. Therefore 
	\[\sideset{}{^\ast} \sum_{a \pmod{D}} \Y(a, D; h_1, h_2) \Z(a,D) = \sideset{}{^\ast} \sum_{b \pmod{D}} \Y_{h^\dagger}(b, D) \Z(b^2, D) \]
	where
	\[\Y_{h^{\dagger}}(b,D) = \sum_{\substack{p_1^2 \in J_1, p_2^2 \in J_2 \\ p_1 \equiv b p_2 \pmod{D} \\ \gcd(p_1 p_2, D) = 1}} h^\dagger(p_1^2) h^\dagger(p_2^2).\]
	When $h_1 = h_2 = h^\dagger$, then $Y(D)$ is the expected value of $Y_{h^\dagger}(b,D)$. Now put
	\[\E(N) = \sum_{D \leq 2N} \sideset{}{^\ast} \sum_{a \pmod{D}} \left \lvert \Y(a, D; h_1, h_2) - Y_{h_1, h_2}(D) \right \rvert \widetilde{\Z}(a,D)\]
	if either $h_1 = h^\dagger$ or $h_2 = h^\dagger$, and
	\[\E_{h^\dagger}(N) = \sum_{D \leq 2N} \sideset{}{^\ast} \sum_{b \pmod{D}} \left \lvert \Y_{h^\dagger}(b,D) - \Y_{h^\dagger, h^\dagger}(D) \right \rvert \widetilde{\Z}(b^2, D)\]
	if $h_1 = h_2 = h^\dagger$. We then have the following proposition, which is Proposition 8 from \cite{HBL}: 
	\begin{proposition} \label{HBL prop8} 
		For any $C > 0$ we have
		\[\E(N) \ll_C \frac{XN}{(\log X)^C} \text{ and } \E_{h^\dagger}(N) \ll_C \frac{XN}{(\log X)^C}.\]
	\end{proposition}
	With this proposition in hand, we may proceed to prove Proposition \ref{HBL prop6} in the exact same way as Proposition 6 in \cite{HBL}. We will not repeat the details. \\ \\
	We now move to the proof of Proposition \ref{HBL prop7}. Most of the arguments can be adapted from the proof of Proposition 7 in \cite{HBL}, but since we rely on some properties of the coefficients $\beta_\Bz$ in this argument we cannot follow all of the arguments in \cite{HBL} verbatim. We will especially emphasize those points where modifications are required.
	
	\section{Proof of Proposition \ref{HBL prop7}: some manoeuvres}
	\label{HBL sec10}  
	
	Supposing that one of the functions $h_1, h_2$ is $h^\ddagger$, we have according to Proposition \ref{HBL prop8} that 
	\[\sum_{D \leq 2N} \sideset{}{^\ast} \sum_{a \pmod{D}} \Y(a, D; h_1, h_2) \Z(a, D) = \sum_{D \leq 2N} \sideset{}{^\ast} \sum_{a \pmod{D}} \Y_{h_1, h_2}(D) \Z(a, D) + O_C \left(\frac{XN}{(\log X)^C} \right)\]
	for any $C > 0$. In the remaining case with $h_1 = h_2 = h^\dagger$, we have
	\[\sum_{D \leq 2N} \sideset{}{^\ast} \sum_{a \pmod{D}} \Y(a, D; h^\dagger, h^\dagger) \Z(a, D) = \sum_{D \leq 2N} \sideset{}{^\ast} \sum_{b \pmod{D}} \Y_{h^\dagger, h^\dagger}(D) \Z(b^2, D) + O_C \left(\frac{XN}{(\log X)^C} \right).\]
	As in \cite{HBL} we may replace $\Y_{h^\dagger, h^\dagger}(D)$ by $|J_1| |J_2|/\varphi(D)$ in each case, with a total error of 
	\[O \left(X \exp \left(\sqrt{-\log X} \right) N (\log X)^2\right).\] 
	Our remaining task is the inequality
	\[|J_1||J_2| \sum_{\substack{\UUU_1, \UUU_2 \\ \fC_1(\UUU_1, \UUU_2, J_1, J_2)}} \sum_{D \leq 2N} \frac{1}{\phi(D)} \left(\sideset{}{^\ast} \sum_{b \pmod{D}} \Z(b^2, D) - \sideset{}{^\ast} \sum_{a \pmod{D}} \Z(a, D) \right) \ll \frac{XN}{(\log X)^{A+ 2L}},\]
	or 
	\[\E^\prime = \sum_{\substack{\UUU_1, \UUU_2 \\ \fC_1(\UUU_1, \UUU_2, J_1, J_2)}} \sum_D \frac{1}{\varphi(D)} \left(\sideset{}{^\ast} \sum_{b \pmod{D}} \Z(b^2, D) - \sideset{}{^\ast} \sum_{a \pmod{D}} \Z(a, D) \right) \ll \frac{N}{(\log X)^A}. \]
	Here we dropped the condition $D \leq 2N$, which follows automatically since $\beta_{\Bz}$ is supported on $\lVert \Bz \rVert_2 \leq 2N$. \\
	
	Since no further hypotheses regarding the coefficients $\beta_{\Bz}$ is necessary, we may follow Heath-Brown and Li's arguments in \cite{HBL} to conclude that it suffices to obtain the estimate
	\begin{equation} \E_1(\UUU_1, \UUU_2) = \sum_D \frac{D}{\varphi(D)} \left(\sideset{}{^\ast} \sum_{b \pmod{D}} \Z(b^2, D) - \sideset{}{^\ast} \sum_{a \pmod{D}} \Z(a, D) \right) \ll \frac{N^2}{(\log X)^{C_1}}
	\end{equation}
	for any $C_1 > 0$ and for fixed $\UUU_1, \UUU_2$. By M\"{o}bius inversion we deduce that
	\[\E_1 (\UUU_1, \UUU_2) = \sum_{D = 1}^\infty \sum_{k=1}^\infty \frac{D \mu(k)}{\varphi(D)} \left(\sideset{}{^\ast} \sum_{b \pmod{D}} W(b^2, k, D) - \sideset{}{^\ast} \sum_{a \pmod{D}} W(a, k, D) \right),\]
	where 
	\[W(a,k, D) = \sideset{}{^\flat} \sum_{\substack{(\Bz, \By) \in \UUU_1 \times \UUU_2 \\ kD | \Delta(\Bz, \By) \\ a \By \equiv \Bz \pmod{D}}} \beta_{\Bz} \beta_{\By}. \]
	When $kD$ divides $\Delta(\Bz, \By)$ there is a unique integer $c = c(\Bz, \By; kD)$ modulo $kD$ such that $c \By \equiv \Bz \pmod{kD}$, and conversely this congruence implies that $kD$ divides $\Delta(\Bz, \By)$. For this integer $c$ we have $\gcd(c,kD) = 1$ and 
	\begin{align*} \#\{b \pmod{D} : b^2 \By \equiv \Bz \pmod{D}\} & = \# \{b \pmod{D} : b^2 \equiv c \pmod{D}\} \\
		& = \sum_{\substack{\chi \pmod{D} \\ \chi^2 = \chi_0}} \chi(c).
	\end{align*}
	It now follows that 
	\[\sideset{}{^\ast} \sum_{b \pmod{D}} W(b^2, k, D) - \sideset{}{^\ast} \sum_{a \pmod{D}} W(a, k,D) = \sum_{\substack{\chi \pmod{D} \\ \chi^2 = \chi_0 \\ \chi \ne \chi_0}} \sideset{}{^\ast} \sum_{c \pmod{kD}} \sideset{}{^\flat} \sum_{\substack{(\Bz, \By) \in \UUU_1 \times \UUU_2 \\ c \By \equiv \Bz \pmod{kD}}} \beta_{\Bz} \beta_{\By} \chi(c),\]
	and hence
	\[\E_1(\UUU_1, \UUU_2) = \sum_{D=1}^\infty \sum_{k=1}^\infty \frac{D \mu(k)}{\varphi(D)} \sum_{\substack{\chi \pmod{D} \\ \chi^2 = \chi_0 \\ \chi \ne \chi_0}} \sideset{}{^\ast} \sum_{c \pmod{kD}} \sideset{}{^\flat} \sum_{\substack{(\Bz, \By) \in \UUU_1 \times \UUU_2 \\ c \By \equiv \Bz \pmod{kD}}} \beta_{\Bz} \beta_{\By} \chi(c).\]
	Let $d = d(\chi)$ be the conductor of $\chi$ and write $D = d e$ and $e k = \fk$, giving
	\[\E(\UUU_1, \UUU_2) = \sum_{d > 1} \sum_{\fk} C(d, \fk) \sideset{}{^\ast} \sum_{\substack{\chi \pmod{d} \\ \chi^2 = \chi_0}} \sideset{}{^\ast} \sum_{c \pmod{d \fk}} \sideset{}{^\flat} \sum_{\substack{(\Bz, \By) \in \UUU_1 \times \UUU_2 \\ c \By \equiv \Bz \pmod{d \fk}}} \beta_{\Bz} \beta_{\By} \chi(c), \]
	where
	\[C(d,\fk) = \sum_{d_1 k = d_2} \frac{d e \mu(k)}{\varphi(d e)} = \frac{d}{\varphi(d)} \sum_{e k = \fk} \frac{\varphi(d) e \mu(k)}{\phi(d e)}. \]
	Note that the sum for $\chi \pmod{d}$ is empty unless $d = d_1, 4 d_1, 8d_1$ with $d_1$ odd and square-free, in which cases there are at most two possible characters $\chi$. For fixed $d$ the function 
	\[\varphi_d(e) = \frac{\varphi(d)e}{\varphi(de)} \]
	is multiplicative in $e$. Further, for $v \geq 1$ we have
	\[(\varphi_e \ast \mu) \left(p^v \right) = \begin{cases} (p-1)^{-1} & \text{if } v = 1 \text { and } p \nmid d \\ 0 & \text{otherwise}.\end{cases}\]
	We then see that
	\[C(d, \fk) = \frac{d \mu^2(\fk)}{\varphi(d \fk)} \]
	if $\gcd(d, \fk) = 1$ and $C(d, \fk) = 0$ otherwise. This gives the expression
	\begin{equation} \label{E1U1U2} \E_1(\UUU_1, \UUU_2) = \sum_{\substack{\fk, d \\ \gcd(d, \fk) = 1}} \frac{d \mu^2(\fk)}{\varphi(d \fk)} \sideset{}{^\ast} \sum_{\substack{\chi \pmod{d} \\ \chi^2 = \chi_0 \\ \chi \ne \chi_0}} \left(\sideset{}{^\ast} \sum_{c \pmod{d \fk}} \sideset{}{^\flat} \sum_{\substack{(\Bz, \By) \in \UUU_1 \times \UUU_2 \\ c \By \equiv \Bz \pmod{d \fk}}} \beta_{\Bz} \beta_{\By} \chi(c) \right).
	\end{equation}
	We proceed to show that large values of $\fk$ make a negligible contribution. Since $d \fk | \Delta(\Bz, \By)$ we have $d \fk \leq 2N$. Since $0 \leq \beta_{\Bz} \leq 1$ we find that 
	\begin{align*} & \sum_{\fk > F} \sum_{\substack{\fk \\ \gcd(d, \fk) = 1}} \frac{d \mu^2(\fk)}{\varphi(d \fk)} \sideset{}{^\ast} \sum_{\substack{\chi \pmod{d} \\ \chi^2 = \chi_0}} \left \lvert \sideset{}{^\ast} \sum_{c \pmod{d \fk}} \sideset{}{^\flat} \sum_{\substack{(\Bz, \By) \in \UUU_1 \times \UUU_2 \\ c \By \equiv \Bz \pmod{d \fk}}} \beta_{\Bz} \beta_{\By} \chi(c) \right \rvert \\
		& \ll (\log X) \sum_{\fk > \fK} \fk^{-1} \sum_{d \leq 2N/\fk} \sum_{\substack{d \fk | D \\ D \leq 2N}} \sideset{}{^\ast} \sum_{a \pmod{D}} \widetilde{\Z}(a, D) \\
		& \ll (\log X) \sum_{\fk > \fK} \fk^{-1} \sum_{d \leq 2N/\fk} \sum_{\substack{d \fk | D \\ D \leq 2N}} N \\
		& \ll \frac{N^2 (\log X)^2}{\fK}.
	\end{align*} 
	Choosing 
	\[\fK = (\log X)^{C_1 + 2}\]
	and applying Lemma \ref{HBL lem13} then gives a satisfactory bound. \\
	
	Observe that the argument above only depends on the property that $0 \leq \beta_{\Bz} \leq 1$, and so no modification is necessary from the argument given by Heath-Brown and Li in \cite{HBL}. As in \cite{HBL} we divide into three ranges for $d$, namely 
	\[d \leq D_1, D_1 < d \leq D_2, \text{ and } d > D_2 \]
	where
	\[D_1 = \fK^{10} (\log X)^{2C_1 + 14} \text{ and } D_2 = \frac{N}{\fK^{15} (\log X)^{3C_1 + 21}}.\]
	Next we handle the middle range of $d$. The treatment given here is identical to that in \cite{HBL}, since again the specific shape of $\beta_{\Bz}$ is of no consequence in this part. Set
	\[\E_1(D) = \sum_{\fk \leq \fK} \fk^{-1} \mu^2(\fk) \sum_{\substack{D < d \leq 2D \\ \gcd(d, \fk) = 1}} \sideset{}{^\ast} \sum_{\substack{\chi \pmod{d} \\ \chi^2 = \chi_0}} \left \lvert \sideset{}{^\ast} \sum_{c \pmod{d \fk}} \sideset{}{^\flat} \sum_{\substack{(\Bz, \By) \in \UUU_1 \times \UUU_2 \\ c \By \equiv \Bz \pmod{d \fk}}} \beta_{\Bz} \beta_{\By} \chi(c) \right \rvert. \]
	Heath-Brown and Li obtains the following bound, which we summarize in the following lemma:
	\begin{lemma} For any $\ep > 0$ we have
		\[\E_1(D) \ll_\ep \fK^5 (\log X)^6 \left \{D + D^{-1/2} N + D^{1/3} N^{2/3} + N^{23/24 + \ep} \right\} N.\]
	\end{lemma} 
	Summing over dyadic ranges of $D$, we see that the values of $d$ in the range $D_1 \leq d \leq D_2$ make a satisfactory contribution given our choices of $D_1, D_2$. \\ \\
	It then remains to give estimates for the small and large ranges of $d$, where we must depart somewhat from Heath-Brown and Li's treatment due to the dependence on the specific shapes of the coefficients $\beta_\Bz$. 
	
	\section{Proof of Proposition \ref{HBL prop7}: remaining ranges}
	\label{HBL sec11}  
	
	\subsection{Large $d$}
	Our goal in this subsection is to obtain the bound
	\[\sum_{\substack{d > D_2 \\ \gcd(d, \fk) = 1}} \sideset{}{^\ast} \sum_{\substack{\chi \pmod{d} \\ \chi^2 = \chi_0}} \left(\sideset{}{^\ast} \sum_{c \pmod{d \fk}} \sideset{}{^\flat} \sum_{\substack{(\Bz, \By) \in \UUU_1 \times \UUU_2 \\ c \By \equiv \Bz \pmod{d \fk}}} \beta_{\Bz} \beta_{\By} \chi(c) \right) \ll_C \frac{N^2}{(\log X)^C}\]
	for any $C > 0$ and $\fk \leq \fK$. There is still more mileage we can get from the argument given in \cite{HBL}. In particular, we follow their argument in Section 11 \cite{HBL} and decompose $d$ as $d_1 d_2$, as well as $\chi = \chi_1 \chi_2$. We have $d \fk | \Delta(\Bz, \By)$ and thus we may set $\Delta(\Bz, \By) = d_1 e t$ where $e$ is odd and $t$ is a power of $2$. Our conditions on $\UUU_1, \UUU_2$ guarantee that $0 < \Delta(\Bz, \By) \leq 2N$, hence $1 \leq et \leq 16N/D_2 \ll (\log X)^{18C_1 + 51}$. We split the sums over $\Bz, \By$ into congruence classes $\Bz \equiv \Bu \pmod{8et}, \By \equiv \Bv \pmod{8et}$ and fix the parameters
	\begin{equation} \label{paras} \fk, d_2, \chi_2, e, \Bu, \Bv, \text{ and } t.\end{equation}
	Each admissible pair $\Bu, \Bv$ corresponds to a unique integer $k \pmod{\Delta(\Bz, \By)}$ with the property that $k \By \equiv \Bz \pmod{\Delta(\Bz, \By)}$, and then
	\[\chi(c) = \chi(k) = \chi_2(k) \left(\frac{k}{d_1} \right)\]
	where $\chi_2(k)$ is determined by the parameters (\ref{paras}). The number of choices for the parameters (\ref{paras}) is bounded by a fixed power of $\log X$ and so it suffices to show that 
	\[\sum_{\substack{d_1 > D_2/d_2 \\ \gcd(d_2, 2 \fk) = 1}} \frac{d_1 \mu^2(d_1)}{\varphi(d_1)} \left(\sideset{}{^\ast} \sum_{k \pmod{d_1 et}} \sideset{}{^\flat} \sum_{\Bz, \By} \beta_{\Bz} \beta_{\By} \left(\frac{k}{d_1} \right) \right) \ll_C \frac{N^2}{(\log X)^C}\]
	for every $C>0$, where the sum over $\Bz, \By$ satisfies the conditions
	\[(\Bz, \By) \in \UUU_1 \times \UUU_2, k \By \equiv \Bz \pmod{\Delta(\Bz, \By)}, \Bz \equiv \Bu \pmod{8et}, \]
	\[\By \equiv \Bv \pmod{8 et}, \text{ and } \Delta(\Bz, \By) = d_1 et.\]
	Following the same analysis in Section 11.1 of \cite{HBL}, we conclude that it is sufficient to obtain the bound 
	\[\sum_{\substack{(\Bz, \By) \in \UUU_1 \times \UUU_2 \\ \Bz \equiv \Bu, \By \equiv \Bv \pmod{8etn} \\ \Delta(\Bz, \By) > etD_2 /d_2}} \beta_{\Bz}^\prime \beta_{\By}^\prime \ll_C \frac{N^2}{(\log X)^C} \]
	where
	\[\beta_{\Bz}^\prime = \beta_\Bz (-1)^{(z_1-1)/2} \left(\frac{z_2}{z_1} \right). \]
	for every fixed $C > 0$, for each choice of parameters $e,t, n \leq (\log X)^C$, and for each $\Bu, \Bv$. Further subdividing into congruence classes it suffices to handle 
	\begin{equation} \label{twistbisum} \sum_{\substack{(\Bz, \By) \in \UUU_1 \times \UUU_2 \\ \Bz \equiv \Bu, \By \equiv \Bv \pmod{8etn}}} \beta_{\Bz}^\prime \beta_{\By}^\prime = \left(\sum_{\substack{\Bz \in \UUU_1 \\ \Bz \equiv \Bu \pmod{8etn}}} \beta_{\Bz}^\prime \right) \left(\sum_{\substack{\Bz \in \UUU_2 \\ \Bz \equiv \Bv \pmod{8etn} }} \beta_{\Bz}^\prime \right).
	\end{equation}
	At this stage that we must diverge from Heath-Brown and Li's treatment. We briefly discuss why this is necessary. In order to proceed, Heath-Brown and Li relies on the crucial property that their $\beta_z$ are supported on Gaussian integers $z$ such that $N(z)$ has no small prime factors. The analogous condition for us is that the ideal number $\gamma(\Bz)$ has norm (equal to the norm of the ideal $J(\gamma(\Bz))$ in $\O_K$) without small prime factors. Thus, now going to the perspective that $\Bz$ represents an ideal number $\gamma$, we see that $N(\gamma) = N(J(\gamma))$ is automatically co-prime to $8etn$ and therefore we may assume that $\upsilon, \nu$ (the ideal numbers corresponding to $\Bu, \Bv$ respectively) are co-prime to $8etn$. This allows us to pick out the congruence condition $\gamma \equiv \upsilon, \nu \pmod{8etn}$ using multiplicative characters. In order to make this precise, we borrow from the algebraic treatment given in \cite{HBM}, and put
	\[\fJ(q) = \{\alpha : \alpha \in \fJ : \gcd(\alpha, q) = 1\}\]
	and $\fJ_1(q) = \fJ(q) \cap K$. Further, put
	\[\fJ_0(q) = \{\alpha : \alpha \in K, \alpha \equiv 1 \pmod{q}\}.\]
	Then our congruence conditions can be picked out using characters of the quotient group $\fJ_1(q)/\fJ_0(q)$, and we conclude that
	\[\sum_{\substack{\widehat{\alpha} \in \UUU_j \\ \alpha \equiv \nu \pmod{8etn} }} = \frac{1}{\varphi_K(8etn)} \sum_{\chi \pmod{8etn}} \ol{\chi}(\nu) \S(\chi, \UUU_j),    \]
	where $\varphi_K$ is the Euler-$\varphi$ function for $\O_K$ and 
	\[\S(\chi, \UUU) = \sum_{\widehat{\alpha} \in \UUU} \beta_{\alpha}^\prime \chi(\alpha). \]
	In order to obtain acceptable estimates for $\S(\chi, \UUU)$, we will need to generalize certain results from \cite{FI1} to apply to general quadratic fields. This work may be of independent interest and is recorded in the next section; see Propositions \ref{prop psi} and \ref{LMN bd} in particular. \\ \\
	We now proceed to pick out the condition that we are constrained in a narrow sector using a twice-differentiable periodic function $\upsilon(\theta)$, where 
	\[\upsilon(\theta) = \begin{cases} 1 & \text{if } \theta \in (\theta_0, \theta_0 + \varpi_2) \pmod{2 \pi} \\ 0 & \text{if } \theta \not \in [\theta_0 - (\log X)^{-C}, \theta_0 + \varpi_2 + (\log X)^{-C}] \pmod{2\pi} \end{cases} \]
	and where $|\upsilon^{\prime \prime}(\theta)| \ll (\log X)^{-2C}$. Then
	\[\S(\chi, \UUU) = \sum_{N^\prime < N(z) \leq N^\prime(1 + \varpi)} \beta_z^\prime \chi(z) \upsilon (\arg z) + O \left( \frac{N}{(\log X)^C} \right).\]
	The Fourier coefficients of $\upsilon$ satisfy $c_k \ll k^{-2} (\log X)^{2C}$ for $k \ne 0$, and so
	\[\upsilon(\arg z) = \sum_k c_k \left(\frac{z}{|z|} \right)^k = \sum_{|k| \leq (\log X)^{3C}} c_k \left(\frac{z}{|z|} \right)^k + O\left((\log X)^{-C} \right). \]
	It then suffices to show that 
	\[\S(\chi, N^\prime, k) = \sum_{N^\prime < N(z) \leq N^\prime (1 + \varpi)} \beta_z^\prime \chi(z) \left(\frac{z}{|z|}\right)^k \ll_C N (\log X)^{-4C}\]
	for any $C > 0$, and for $|k| \leq (\log X)^{3C}$. As in \cite{HBL} we can obtain in fact a small power-saving in $N$. We recall that $\beta_z = \beta_{N(z)}$ is the indicator function of a set of one of the shapes
	\[Q_j = \{p_1 \cdots p_{j+1} \in (N^\prime, N^\prime(1 + \varpi)] : p_{j+1} \in J, p_{j+1} < \cdots < p_1, \]
	\[p_1 \cdots p_j < Y \leq p_1 \cdots p_{j+1} < X^{1/2 0 \delta}\}\]
	or
	\[R = \{n \in (N^\prime, N^\prime(1 + \varpi)] : \gcd(n, P(V)) = 1\}.\]
	Here we will have $0 \leq j \leq n_0 = \lfloor \log Y/(\delta \log X) \rfloor$, and $J = [V, V(1+ \kappa)) \subseteq [X^\delta, X^{1/2 - \delta})$. In particular we interpret $Q_0$ to be $\{p : p \in J \cap (N^\prime, N^\prime(1 + \varpi)]\}$. \\ \\
	We now write 
	\[\lambda(n) = \sideset{}{^\wedge} \sum_{N(z) = n} \chi(z) \left(\frac{z}{|z|} \right)^k u^{(x-1)/2} \left(\frac{z_2}{z_1} \right)\]
	where $\sideset{}{^\wedge} \sum$ denotes a sum over primitive ideal numbers $z$ in a fixed class of ideal numbers, with $\widehat{z} = (z_1, z_2)$. We then have
	\[\S(\chi, N^\prime, k) = \sum_n \lambda(n)\]
	where $n$ runs over $R$ or $Q_j$ for some $j$. As in \cite{HBL}, the treatment for $R$ and $Q_j$ are similar. To begin, we first handle the contribution from those $n$ whose largest prime factor, say $\P(n)$, exceeds $N^{99/100}$. The contribution from such integers is 
	\[\sum_{m \leq 2N^{1/100}} \sum_{\substack{p > \max\{\P(m), N^{99/100} \\ mp \in Q_j}} \lambda(mp). \]
	Since $p$ is the largest prime factor of $mp$ one sees from the definition of the set $Q_j$ that one may rewrite the conditions $p > \P(m)$ and $mp \in Q_j$ to say that $p$ runs over an interval $I_j(m) \subseteq [N/m, 2N/m)$. We may then apply Proposition \ref{prop psi} to conclude that
	\begin{align*} \sum_{m \leq 2N^{1/100}} \sum_{\substack{p > \max\{\P(m), N^{99/100} \\ mp \in Q_j}} \lambda(mp) & \ll q_0 (|k| + 1) \sum_{m \leq 2N^{1/100}} m(N/m)^{76/77} \\
		& \ll q_0(|k| + 1) N^{76/77 + (78/77)/100}. 
	\end{align*} 
	Since $76/77 + (78/77)/100 < 1$, this is gives the required power-saving bound. \\ \\
	Next we deal with the terms where every prime factor is at most $N^{99/100}$. To do so we rewrite our sum in terms of bilinear sums. Suppose $n = p_1 \cdots p_{j+1}$ as in the description of the set $Q_j$, and divide the range of each prime $p_i$ into intervals of the shape $(B_i, 2 B_i]$. This will give us at most $(2 \log N)^{1+n_0}$ sets of dyadic ranges, and since $n_0 \ll \delta^{-1} = (\log X)^{1 - \varpi}$ there will be at most $O_\ep(N^\ep)$ such ranges. Moreover we may suppose
	\[\prod_{i=1}^{j+1} B_i \ll N \ll 2^{j+1} \prod_{i=1}^{j+1} B_i.\]
	Since we may now assume that $B_1 \leq N^{99/100}$ there will be an index $u$ such that
	\[N^{1/100} \leq \prod_{i=1}^u B_i \leq N^{99/100}.\]
	Fixing such an index $u$ we split $n = n_1 n_2$ with 
	\[n_1 = \prod_{i=1}^u p_i \text{ and } n_2 = \prod_{i=u+1}^{j+1} p_i,\]
	so that $n_1 \leq N_1$ and $n_2 \leq N_2$ with 
	\[N_1 = 2^{1 + n_0} \prod_{i=1}^u B_i \text{ and } N_2 = 2^{1 + n_0} \prod_{i=u+1}^{j+1} B_i.\]
	It follows that
	\[N_1 N_2 \ll_\ep N^{1 + \ep} \text{ and } N_1, N_2 \ll_\ep N^{99/100 + \ep}\]
	respectively. This implies that
	\[N_1 N^{-\ep} \ll n_1 \leq N_1 \text{ and } N_2 N^{-\ep} \ll n_2 \leq N_2.\]
	We may thus reinterpret our description of $Q_j$ by requiring that $n_1 \in Q_{j,u}$ and $n_2 \in Q_{j,u}^\prime$ for appropriate sets $Q_{j,u}, Q_{j,u}^\prime$, together with the conditions that 
	\begin{equation} \label{good reg} n_1 n_2 \in I = (N^\prime, N^\prime(1 + \varpi)] \cap [Y, X^{1/2 - \delta}), p_{j+1}^{-1} n_1 n_2 < Y, \text{ and } p_{u+1} < p_u.\end{equation} 
	In other words, we put
	\[Q_{j,u} = \{n_1 = p_1 \cdots p_u : p_i \in (B_i, 2B_i], p_u < \cdots < p_1\}\]
	and
	\[Q_{j,u}^\prime = \{n_2 = p_{u+1} \cdots p_{j+1} : p_i \in (B_i, 2B_i], p_{j+1} \in J, p_{j+1} < \cdots < p_{u+1} < Y\}.\]
	In order to separate the variables $n_1, n_2$ completely we subdivide the available ranges for $n_1, n_2, p_{j+1}, p_u$, and $p_{u+1}$ into intervals of the shape $(A, A + A/L), (A^\prime, A^\prime + A^\prime/L]$, $(B_{j+1}^\prime, B_{j+1}^\prime + B_{j+1}^\prime/L]$, $(B_u^\prime, B_u^\prime + B_u^\prime/L]$ and $(B_{u+1}^\prime, B_{u+1}^\prime + B_{u+1}^\prime/L]$. Here the parameter $L$ will be chosen to be a small power of $N$. One should note that these intervals may have length less than one. Indeed such an interval may contain no integers at all. \\ \\
	There will be $O(L^5 (\log X)^2)$ such intervals and there will be some for which the conditions $n_1 n_2 \in I, p_{j+1}^{-1} n_1 n_2 < Y$ and $p_{u+1} < p_u$ hold for every choice of $p_1, \cdots, p_{j+1}$ satisfying 
	\begin{align*} & n_1 \in (A, A + A/L], n_2 \in (A^\prime, A^\prime + A^\prime/L] \\
		& p_{j+1} \in (B_{j+1}^\prime, B_{j+1}^\prime + B_{j+1}^\prime/L], p_u \in (B_u^\prime, B_u^\prime + B_u^\prime/L],\\
		p_{u+1} \in (B_{u+1}^\prime, B_{u+1}^\prime + B_{u+1}^\prime/L],
	\end{align*}
	and
	\[p_i \in I_i \text{ with } i \ne 1, u, u+1.\]
	This case gives the subsum
	\[\sum_{\substack{n_1 \in Q_{j,u} \cap (A, A + N_1/L] \\ p_u \in (B_u^\prime, B_u^\prime + B_u^\prime/K]}} \sum_{\substack{n_2 \in Q_{j,u}^\prime \cap (A^\prime, A^\prime + A^\prime/L] \\ p_{j+1} \in (B_{j+1}^\prime, B_{j+1}^\prime + B_{j+1}^\prime/L] \\ p_{u+1} \in (B_{u+1}^\prime, B_{u+1}^\prime + B_{u+1}^\prime/L]}} \lambda(n_1 n_2), \]
	so that we have separated the variables $n_1, n_2$. For such sums we can apply Proposition \ref{LMN bd} which gives the bound
	\[O_\ep \left((N_1 + N_2)^{\frac{1}{12}} (N_1 N_2)^{\frac{11}{12} + \ep} \right) = O_\ep \left(N^{\frac{99}{100} \cdot \frac{1}{12}} \cdot N^{\frac{11}{12} + \ep} \right) = O_\ep \left(N^{1 - \frac{1}{1200} + \ep} \right). \]
	Since there are $O_\ep(L^5 N^\ep)$ such subsums the overall contribution will be $O(L^5 N^{1- 1/200 + \ep})$. \\ \\
	It remains to consider the contribution from the remaining ``bad" sets of ranges which are not exclusively contained in the region given by (\ref{good reg}). First suppose that the interval $I$ is given by $[e_1, e_2]$ say, and that there are integers $n_1, n_1^\prime \in (A, A + A/L]$ and $n_2, n_2^\prime \in (A^\prime, A^\prime + A^\prime/L]$ such that $n_1 n_2 \in I$ but $n_1^\prime n_2^\prime \not \in I$. Then we must have $n_1 n_2 = (1 + O(L^{-1}) e_1$ or $n_1 n_2 = (1 + O(L^{-1})) e_2$. We now consider the total contribution from integers $n \in Q_j$ for all such ``bad" choices of intervals $(A, A + A/L), (A^\prime, A^\prime + A^\prime/L]$, $(B_{j+1}^\prime, B_{j+1}^\prime + B_{j+1}^\prime/L]$, $(B_u^\prime, B_u^\prime + B_u^\prime/L]$ and $(B_{u+1}^\prime, B_{u+1}^\prime + B_{u+1}^\prime/L]$. Since each integer $n$ occurs at most once, and $\lambda(n) = O(\tau(n))$, the contribution will be 
	\[O_\ep \left(\sum_{n = (1 + O(L^{-1})e_1} \tau(n) \right) = O_\ep\left(N^{1 + \ep} L^{-1} \right).\]
	Similarly, if we have $p_{j+1}^{-1} n_1 n_2 < Y$ but $(p_{j+1}^\prime)^{-1} n_1^\prime n_2^\prime \geq Y$, then $p_{j+1}^{-1} n_1 n_2 = (1 + O(L^{-1})) Y$. This gives
	\[B_{j+1} Y \asymp AA^\prime \leq N_1 N_2 \ll_\ep N^{1+\ep},\]
	so any $n$ which is counted in this case will have a prime factor $p \ll N^{1+\ep}/Y$ and such that $p^{-1} n = (1 + O(L^{-1}))Y$. Thus, on writing $n = pm$, we see that the total contribution in this case is 
	\[O \left(\sum_{p \ll N^{1+\ep}/Y} \sum_{m = (1 + O(L^{-1})Y} \tau(pm) \right) = O_\ep \left( N^{1+\ep} Y^{-1} (1 + L^{-1} Y) \right) = O_\ep \left(N^{1+\ep} L^{-1}\right),  \]
	for $L \leq Y$. \\ \\
	Finally, if $B_u = B_{u+1}$, then it may happen that the condition $p_{u+1} < p_u$ is satisfied by some, but not all, pairs of primes $(p_u, p_{u+1})$ from the intervals $(B_u^\prime, B_u^\prime + B_u^\prime/L]$ and $(B_{u+1}^\prime, B_{u+1}^\prime + B_{u+1}^\prime/L]$. Clearly this problem cannot arise when $L  \geq 2 P_u$ since then the intervals $(B_u^\prime, B_u^\prime + B_u^\prime/L]$ and $(B_{u+1}^\prime, B_{u+1}^\prime + B_{u+1}^\prime/L]$ contain at most one prime each. It follows that any such $n$ to be counted in this case must have two prime factors $p^\prime > p \geq P_u \geq L/2$ with $p^\prime = (1 + O(L^{-1})p$. Hence the corresponding contribution is 
	\[O \left(\sum_{\substack{p^\prime > p \geq L/2 \\ p^\prime = (1 + O(L^{-1})p}} \sum_{\substack{n \ll N \\ p^\prime p | n }} \tau(n) \right) = O_\ep \left(\sum_{\substack{p^\prime > p \geq L/2 \\ p^\prime = (1 + O(L^{-1}))p}} \frac{N^{1+\ep}}{p^\prime p} \right) = O_\ep \left(N^{1+\ep} L^{-1} \right). \]
	We therefore find that our sum is bounded by 
	\[O_\ep \left(L^5 N^{1-1/1200 + \ep} + N^{1+\ep} L^{-1} \right),\]
	whenever $L \leq Y$. We may then choose $L = N^{10^{-5}}$ say, to achieve the claimed power saving in the case of large $d$. 
	
	\subsection{Small $d$} 
	
	To handle small $d$ it suffices to show that for any $\fk \leq \C, d \leq D_1$, and any non-principal $\chi \pmod{d}$ that 
	\[\sideset{}{^\ast} \sum_{c \pmod{d \fk}} \sideset{}{^\flat} \sum_{\substack{(z, y) \in \UUU_1 \times \UUU_2 \\ c z \equiv y \pmod{d\fk}}} \beta_z \beta_y \chi(c) \ll_C \frac{N^2}{(\log X)^C}\]
	for every $C > 0$. Since 
	\[\sideset{}{^\ast} \sum_{c \pmod{d\fk}} \chi(c) = 0,\]
	it suffices to prove that if $\UUU = \UUU_1$ or $\UUU_2$ then there is a number $\fM = \fM(\UUU, d\fk)$ such that
	\[\sum_{\substack{z \in \UUU \\ z \equiv \alpha \pmod{2d\fk}}} \beta_z = \fM + O_C \left(\frac{N}{(\log X)^C} \right) \]
	for any $\gcd(\alpha, 2d\fk) = 1$ and $C > 0$, since $\beta_z$ is supported on those $z$ free of small prime factors, and $2d\fk$ is small. As before we may drop the summation condition $\flat$. For notational convenience, we set $q = 2 d \fk$ and note that $q \leq (\log X)^{C_0}$ for some $C_0 > 0$. \\ \\
	As in the previous subsection we may assume that $\beta_z = \beta_{N(z)}$, where $\beta_n$ is the indicator function of either $Q_j$ or $R$. We describe the procedure for $Q_j$, the method for $R$ being similar. We decompose $z$ as $s_1 s_2$ with $N(s_1)$ being the largest prime factor of $N(s_1 s_2)$. The requirement that $n \in Q_j$ is then equivalent to a condition of the form $N(s_2) \in Q_j^\prime$ together with a restriction of the type $N(s_1) \in I(s_2)$ for some real interval $I(s_2)$. Specifically, we have
	\[Q_{j+1}^\prime = \{p_2 \cdots p_{j+1} : p_{j=1} \in J, p_{j+1} < \cdots < p_2\}\]
	and
	\[I(s_2) = (p_2, \infty) \cap \left(\frac{N^\prime}{N(s_2)}, \frac{N^\prime(1+\varpi)}{N(s_2)} \right] \cap \left[\frac{Y}{N(s_2)}, \frac{X^{1/2 - \delta}}{N(s_2)} \right),\]
	where $p_2$ is the largest prime factor of $N(s_2)$. When $\UUU$ is given by (\ref{U shape}) the condition on the size of $N(s_1 s_2)$ is exactly the condition
	\[N(s_1) \in \left(\frac{N^\prime}{N(s_2)}, \frac{N^\prime(1+ \omega)}{N(s_2)} \right],\]
	and we have $\theta_0 < \arg z \leq \theta_0 + \omega_2$ exactly when $\arg s_1$ is constrained in a small interval of length $O(\varpi_2)$ dependent on $\arg s_2$. It follows that
	\begin{equation} \label{frankM} \sum_{\substack{z \in \UUU \\ z \equiv \alpha \pmod{q}}} \beta_z = \sum_{\substack{N(s_2) \in Q_j^\prime \\ \gcd(s_2, q) = 1}} \N(s_2, \alpha),\end{equation}
	where $\N(s_2, \alpha)$ is the number of ideal numbers $s_1$ satisfying 
	\[s_1 s_2 \equiv \alpha \pmod{q}, N(s_1) \in I(s_2), \text{ and } \arg s_1, \arg s_2 \]
	and for which $N(s_1)$ is prime. We can estimate $\N(s_2, \alpha)$ using a form of the Prime Number Theorem for arithmetic progressions over number fields, due to Mitsui. We note that, as we remarked earlier, we can easily re-divide our sectors in accordance with the condition $N(z) \sim N$ as opposed to $\lVert \Bz \rVert_2 \sim N$, so we may apply Mitsui's theorem without worry in each of our sectors. If we put $\pi(X; q, \alpha, \theta)$ for the number of prime ideal numbers $\fp$ in a fixed ideal class satisfying $\fp \equiv \alpha \pmod{q}$ and having norm at most $X$ with $0 \leq \arg(\fp) \leq \theta$, then Mitsui's theorem gives the estimate
	\begin{equation} \label{mits} \pi(X; q, \alpha, \theta) = \frac{w \theta R_K}{2^{r_1} h_K \varphi_K(\fa) } \operatorname{Li}(X) + O_{K} \left(X \exp \left(-c \sqrt{\log X} \right) \right) 
	\end{equation}
	where $r_1$ is the number of real embeddings of $K$, $w$ the number of roots of unity in $K$, $R_K$ the regulator of $K$, and $h_K$ the class number o $K$. Here $c$ is an absolute constant. Since we do not care about dependence on $K$, we may take the implied constant in (\ref{mits}) as an absolute constant. We emphasize that (\ref{mits}) holds uniformly for $\theta \in [0, 2 \pi]$ and for all $q \leq (\log X)^A$. \\ \\
	Applying (\ref{mits}) with $q = 2 d \fk$ to estimate $\N(s_2, \alpha)$, we have $I(s_2) \subseteq (0, 2N/N(s_2)]$ and so we will need to know that $q = 2 d \fk \leq (\log 2N/N(s_2))^A$ for some constant $A$. This holds whenever $p$ divides an element of $Q_j$ then one has $p \geq X^{\delta_1}$ with $\delta = (A \log \log X)/\log X$. Thus we will have $2N/N(s_2) \geq X^{\delta_1}$ and so
	\[\delta_1 \log X \leq \log \left(\frac{N}{N(s_2)} \right),\]
	which implies that
	\[\log X \leq \left(\log \left(\frac{N}{N(s_2)} \right) \right)^{\frac{1}{\varpi}}.\]
	Therefore whenever $2d \fk \leq (\log X)^{C_0}$ we have
	\[2 d \fk \leq (\log X)^{C_0} \leq \left( \log \left(\frac{2N}{N(s_2)} \right) \right)^{\frac{C_0}{\varpi}} \]
	 The required condition therefore holds when $\fk \leq \fK$ and $d \leq D_1$. \\ \\
	We may then conclude, as in \cite{HBL}, that
	\[\N(s_2, \alpha) = \fM(s_2, d\fk, j, \UUU) + O\left(\frac{N}{N(s_2)} \exp \left(-c (\log X)^{\varpi/2} \right) \right)\]
	where the main term crucially is independent of $\alpha$. Feeding this into (\ref{frankM}) then completes our treatment of small $d$, and hence the proof of Proposition \ref{main bisum}. 
	
	\section{Character sums} 
	\label{char sums} 
	
	In this section our goal is to introduce and prove analogues of Proposition 23.1 and Theorem $\psi$ in \cite{FI1}. To wit, we introduce, for an ideal number $\alpha$ in a fixed class $A$, the vector
	 \[\widehat{\alpha} = (a_1, a_2) \in \bZ^2\]
	 corresponding to the class $A$ with basis produced as in Section 4. We then introduce the symbol
	\[[\alpha] = i^{\frac{a_1 - 1}{2}} \left(\frac{a_2}{|a_1|} \right)\]
	where $\left(\frac{\cdot}{\cdot} \right)$ is the Jacobi symbol. Note that the symbol $[\cdot]$ depends on the class $A$ \emph{and the choice of basis}, which we have suppressed. \\ \\
	Our goal is to obtain an analogue of Lemma 20.1 in \cite{FI1}, which shows that while $[\cdot]$ is not multiplicative, a suitable result exists to separate $[zw]$ into $[z] [w] \kappa(zw)$, where $|\kappa(zw)| = 1$ and $\kappa$ can be described explicitly. To do so we need to introduce an analogue of the so-called Jacobi-Kubota symbol $\xi_w(z)$ in \cite{FI1}. Defining the analogue of $\xi_w(z)$ in the present setting is tricky, due to the fact that in general $\O_K$ need not be a unique factorization domain and could have an infinite unit group. \\ \\
	To prepare for our definition, we first gather several of the key properties satisfied by Friedlander and Iwaneic's $\xi_w(z)$ in \cite{FI1}. In particular, it satisfies the following:
	\begin{enumerate} 
		\item It satisfies an equation of the form
		\[[z][w] = \ep [zw] \xi_w(z)\]
		where $\ep = \pm 1$ depending only on the quadrants containing $z,w$ respectively; 
		\item It is multiplicative for each $w \in \bZ[i]$: one has $\xi_w(z_1) \xi_w(z_2) = \xi_w(z_1 z_2)$; 
		\item It is symmetric: $\xi_w(z) = \xi_z(w)$ for $w, z \in \bZ[i]$; 
		\item (Lemma 21.1 in \cite{FI3}) For $q = |w_1 w_2|^2$ and $d = |\gcd(w_1, \ol{w_2})|^2$ one has
		\[\sum_{\zeta \pmod{q}} \xi_{w_1}(\zeta) \xi_{w_2}(\zeta) = \begin{cases} q \varphi(d) \varphi(q/d) & \text{if } q,d \text{ are squares} \\ 0 & \text{otherwise.} \end{cases}.\]
		\item For $w = u + iv$ and $\omega \equiv - v \ol{u} \pmod{q}$ with $q = |w|^2$, one has
		\[\xi_w(z) = \left(\frac{ur - vs}{q} \right) \text{ and } \xi_w(z) = \left(\frac{r + \omega s}{q} \right),\]
		where $z = r + is$. 
	\end{enumerate}
	We would like to define our function $\xi_\alpha(z)$ to have the same properties. Unfortunately, it seems that at least some of these properties require special structures of the Gaussian integers $\bZ[i]$. Thus, some more preparatory work is needed before we can define our stand-in for the Jacobi-Kubota symbol. We then check that our analogous symbol has the necessary properties to carry out the proofs of analogous statements in \cite{FI1}. \\ 
	
	First we note that our symbol $\xi_\alpha(z)$ depends on $\alpha$, and in particular, depends on the class $A$ of $\alpha$. This of course is a trivial point when $K = \bQ(i)$, since $\bZ[i]$ has unique factorization. Next we will also need to restrict the class of the \emph{inputs} $z$, in order for our symbol to be well-behaved. This is far from ideal and is likely too restrictive, but it suffices for our purposes in this paper. Indeed, later we will see that it is necessary to define a separate symbol $\xi$ for each class of ideal numbers \emph{along with a basis} of said ideal numbers. \\ 
	
	The most important property turns out to be (1), so we define our symbol with this in mind. To simplify matters we will assume that in our composition law the bilinear form $Q_{A,B}(w,z)$ is given by $w_1 z_1 + w_2 z_2$. In particular, we fix bases $\{\alpha_1, \alpha_2\} \subset A, \{\beta_1, \beta_2\} \subset B, \{\gamma_1, \gamma_2\} \subset C = \cdot B$ so that
	\[(\alpha_1 x_1 + \alpha_2 x_2)(\beta_1 y_1 + \beta_2 y_2) = (x_1 y_1 + x_2 y_2) \gamma_1 + (x_1 \ell_1(y_1,  y_2) + x_2 \ell_2(y_1, y_2))\gamma_2.\] 
	Observe that the roles of $R_{A,B}, Q_{A,B}$ are switched from the previous sections, but this is due the freedom to choose our bases. \\ 
	
	We begin with the Jacobi symbol
	\[\left(\frac{w_1 \ell_1(z_1, z_2) + w_2 \ell_2(z_1, z_2)}{|w_1 z_1 + w_2 z_2|} \right)\]
	where $R_{A,B}(w, z) = w_1 z_1 + w_2 z_2$. Note that we can extend the definition of the Jacobi symbol by setting
	\[\left(\frac{a}{b}\right) = \left(\frac{a}{|b|} \right) (a,b)_\infty, \]
	where
	\[(a,b)_\infty = \begin{cases} -1 & \text{if } a, b < 0 \\ 1 &  \text{otherwise} \end{cases}\]
	is the Hilbert symbol. Next we note quadratic reciprocity, which states for $a,b$ odd and coprime that
	\begin{equation} \label{quadrec} \left(\frac{a}{|b|} \right)\left(\frac{b}{|a|} \right) = (-1)^{\frac{a-1}{2} \cdot \frac{b-1}{2}} (a,b)_\infty.
	\end{equation} 
Clearly, not both $Q_{A,B}, R_{A,B}$ can be even otherwise the corresponding ideal number is not primitive. Without loss of generality, let us suppose that $w_1 z_1 + w_2 z_2$ is odd. Let $2^k$ be the highest power of $2$ dividing $w_1 \ell_1(z_1, z_2) + w_2 \ell_2(z_1, z_2)$. Then 

\[\left(\frac{w_1 z_1 + w_2 z_2}{|w_1 \ell_1(z_1, z_2) + w_2 \ell_2(z_1, z_2)|} \right) = \left( \frac{w_1 z_1 + w_2 z_2}{2^{-k} |w_1 \ell_1(z_1, z_2) + w_2 \ell_2(z_1, z_2)|}\right).\]
We put
\[u = w_1 z_1 + w_2 z_2, v = w_1 \ell_1(z_1, z_2) + w_2 \ell_2(z_1, z_2)\]
for simplicity. Applying quadratic reciprocity (\ref{quadrec}) then gives 

\begin{align*}\left( \frac{w_1 z_1 + w_2 z_2}{2^{-k} (w_1 \ell_1(z_1, z_2) + w_2 \ell_2(z_1, z_2))}\right) (u,v)_\infty & = \left(\frac{2^{-k} v}{u} \right) (-1)^{\frac{u-1}{2} \cdot \frac{2^{-k} v - 1}{2}} \\
	& = \left(\frac{2^k}{u} \right) (-1)^{\frac{u-1}{2} \cdot \frac{2^{-k} v - 1}{2}} \left(\frac{v}{u}\right).
\end{align*}
Now we use the fact that
\[w_1 z_1 + w_2 z_2 \equiv 0 \pmod{u}\]
implies
\[w_2 \equiv - z_2^{-1} w_1 z_1 \pmod{u}.\]
Substituting this into $R_{A,B}(w,z)$ gives
\begin{align*} w_1 \ell_1(z_1, z_2) + w_2 \ell_2(z_1, z_2) & \equiv w_1 \ell_1(z_1, z_2) - z_2^{-1} w_1 z_1 \ell_2(z_1, z_2) \pmod{u} \\
	& \equiv z_2^{-1}w_1 \left(z_2 \ell_1(z_1, z_2) - z_1 \ell_2(z_1, z_2) \right) \pmod{u}.
\end{align*}	
Here we require an interpretation of the quadratic form 
\[g(z_1, z_2) = z_2 \ell_1(z_1, z_2) - z_1 \ell_2(z_1, z_2).\]
By definition, our composition law gives the relation
\begin{align} \label{comp law g} (z_2 \alpha_1 - z_1 \alpha_2)(z_1 \beta_1 + z_2 \beta_2) & = R_{A,B}(z_2, -z_1; z_1, z_2) \gamma_1 + Q_{A,B}(z_2, -z_1; z_1, z_2) \gamma_2 \\ 
	& = (z_2 \ell_1(z_1, z_2)) - z_1 \ell_2(z_1, z_2) \gamma_1 + (z_2 z_1 - z_1 z_2) \gamma_2 \notag \\
	& = g(z_1, z_2) \gamma_1. \notag
\end{align} 
Dividing both sides by $\gamma_1$ we then see that $g(z_1, z_2)$ must be equivalent to the norm form of $\O_K$. \\ \\
We must now relate $g(z_1, z_2)$ to $N(z) = N(J(z_1 \beta_1 + z_2 \beta_2))$. Note that 
\[g(z_1, z_2) = \gamma_1^{-1} (\alpha_1 z_2 - \alpha_2 z_1)(\beta_1 z_1 + \beta_2 z_2)\]
is divisible by $z = \beta_1 z_1 + \beta_2 z_2$, which implies that $g(z_1, z_2)$ is a rational integer divisible by $N(z)$. By primitivity we then see that $g(z_1,z_2)$ must be a constant multiple of $N(z)$, the constant depending only on the classes $A,B$. We summarize this as a lemma:

\begin{lemma} \label{norm char} Let $g(x,y)$ be the integral binary quadratic form which arises from the composition law (\ref{comp law g}). Then $g(z_1, z_2)$ is a constant multiple of $N(J(\beta_1 z_1 + \beta_2 z_2))$, with the constant depending only on the classes $A,B$ and choices of bases of $A,B, A \cdot B$. 
\end{lemma}
Similarly, since $v = w_1 \ell_1 + w_2 \ell_2$ is divisible by $2^k$, we may assume without loss of generality that $\ell_1$ is odd to obtain
\[w_1 \equiv -\ell_1^{-1} w_2 \ell_1 \pmod{2^k}\]
and this implies that
\begin{align*} w_1 z_1 + w_2 z_2 & \equiv - \ell_1^{-1} w_2 z_1 + w_2 z_2 \pmod{2^k} \\
	& \equiv -\ell_1^{-1} w_2 (z_1 \ell_1 - z_2 \ell_2) \pmod{2^k} \\
	& \equiv - \ell_1^{-1} w_2 g_C(z_1, z_2) \pmod{2^k}.
\end{align*}
Since $u = w_1 z_1 + w_2 z_2$ is odd by assumption, it follows that $g_C(z_1, z_2)$ must be odd as well. \\ 

Continuing on, we then have
\begin{align*} z_2^{-1}w_1 \left(z_2 \ell_1(z_1, z_2) - z_1 \ell_2(z_1, z_2) \right)  & \equiv z_2^{-1} w_1 g_C(z_1, z_2) \pmod{u},
\end{align*}
which implies that 
\[\left(\frac{v}{u} \right) = \left(\frac{z_2^{-1} w_1 g_C(z_1, z_2)}{u} \right) = \left(\frac{z_2 w_1}{u} \right) \left(\frac{g_C(z_1, z_2)}{u} \right).\]
Applying quadratic reciprocity again
we obtain
\[\left(\frac{z_2 w_1}{u} \right) \left(\frac{g_C(z_1, z_2)}{u} \right) = \left(\frac{2^{k_1 + k_2}}{u} \right) \left(\frac{u}{ w_1 z_2} \right) \left(\frac{u}{g_C(z_1, z_2)} \right). \]
Here $2^{k_1}$ is the highest power of $2$ dividing $w_1$ and $2^{k_2}$ the highest power of $2$ dividing $z_2$. Note that
\begin{align*}\left(\frac{w_1 z_1 + w_2 z_2}{w_1 z_2} \right) & = \left(\frac{w_2 z_2}{w_1}\right) \left(\frac{w_1 z_1}{z_2} \right) \\ 
	& = \left(\frac{w_2}{w_1} \right) \left(\frac{z_2}{w_1} \right) \left(\frac{w_1}{z_2} \right) \left(\frac{z_1}{z_2} \right) \\
	& = \left(\frac{w_2}{|w_1|} \right) \left(\frac{z_1}{|z_2|} \right) \left(\frac{2^{k_2}}{w_1} \right) \left(\frac{2^{k_1}}{z_2} \right) (-1)^{\frac{2^{-k_1} w_1 - 1}{2} \frac{2^{-k_2} z_2 - 1}{2}} (z_1, z_2)_\infty (w_1, w_2)_\infty  \\
	& = \left(\frac{w_2}{|w_1|} \right) \left(\frac{z_2}{|z_1|} \right) \left(\frac{2^{k_2}}{z_1} \right)\left(\frac{2^{k_2}}{w_1} \right) \left(\frac{2^{k_1}}{z_2} \right) (-1)^{\frac{2^{-k_1} w_1 - 1}{2} \frac{2^{-k_2} z_2 - 1}{2}} (z_1, z_2)_\infty (w_1, w_2)_\infty.  
\end{align*}
Collecting these calculations we conclude that 

\begin{equation}\left(\frac{w_1 \ell_1(z_1, z_2) + w_2 \ell_2(z_1, z_2)}{|w_1 z_1 + w_2 z_2|} \right) = \left(\frac{w_2}{|w_1|}\right) \left(\frac{z_2}{|z_1|}\right) \left(\frac{w_1 z_1 + w_2 z_2}{g(z_1, z_2)} \right) \ep(w,z) \theta(w,z),
\end{equation}	
where $\ep(w,z)$ is the product of all of the Hilbert symbols and the terms of the shape $(-1)^x$ for some $x \in \bZ$ which appear. It is clear that $\ep(w,z)$ depends on the congruence class of $w,z$ with a bounded conductor, and thus is of little consequence. Here $\theta(w,z)$ is given by
\[\theta(w,z) = \left(\frac{2^{k_2}}{z_1} \right)\left(\frac{2^{k_2}}{w_1} \right) \left(\frac{2^{k_1}}{z_2} \right)  \left(\frac{2^{k_1 + k_2}}{u} \right) \left(\frac{2^k}{u}\right).\]
Since we have insisted that $w,z$ belong to fixed congruence classes modulo $8etn$ as in (\ref{twistbisum}) it follows that $\theta(w,z)$ can be determined as a function of the congruence class alone, and is therefore a constant for our purposes. \\ 

These calculations compels us to define our analogue of the Jacobi-Kubota symbol as
\begin{equation} \label{jackub} \xi_w(z) = \left(\frac{w_1 z_1 + w_2 z_2}{g(w_1,w_2)} \right).
\end{equation}
Note that $\xi_w(z)$ depends on the ideal classes of $w,z$ and a choice of basis for the ideal classes. \\

Next we observe for $w,z$ satisfying (\ref{twistbisum}), $w,z$ are in the same class and therefore $R_{A,B}(w,z) = R_{A,A}(w,z)$ must be symmetric in $w,z$. From here it follows that 
\begin{align*} z_2^{-1} w_1 g(z_1, z_2) & \equiv R_{A,A}(w,z) \pmod{u} \\ 
	& \equiv R_{A,A}(z,w) \pmod{u}\\
	& \equiv w_1^{-1} z_2 g(w_1, w_2) \pmod{u}. 
\end{align*}
This implies that
\[\left(\frac{g(z_1, z_2)}{u} \right) \left(\frac{g(w_1, w_2)}{u} \right) = 1.\]
Thus, up to a factor $\ep$ depending at most on congruence classes and signs of $w,z$, we have
\begin{equation} \label{symm}\xi_w(z) = \ep \xi_z(w).\end{equation} 

Summarizing, we obtain the following analogue of Lemma 20.1 in \cite{FI3}:   

	\begin{lemma} \label{multi lem}  Let $w,z$ satisfy the hypothesis given in (\ref{twistbisum}). Then there exist numbers $\ep(w,z) \in \{-1,1\}$ depending only on the signs and congruence classes of $w,z$ modulo $8etn$ such that

	\begin{equation} \label{multi ident}  \left(\frac{Q_{A,B}(w,z)}{|R_{A,B}(w,z)|} \right) = \ep(w,z) \left(\frac{w_2}{|w_1|} \right) \left(\frac{z_2}{|z_1|} \right) \xi_z(w). 
	\end{equation}
	
\end{lemma} 
	
Next we show that the analogue of Lemma 21.1 in \cite{FI1} holds:
	
	\begin{lemma} \label{xi cong sum 1} For fixed elements $w, v$ in the class $A$ and $q = g(w_1,w_2) g(v_1,v_2)$ and $d = \gcd(g(w_1,w_2), g(v_1,v_2))$, we have 
		\[\sum_{z \pmod{q}} \xi_{w_1}(z) \xi_{w_2}(z) = \begin{cases} q \varphi(d) \varphi(q/d) & \text{if } q, d \text{ are squares} \\ 0 & \text{otherwise.} \end{cases}  \]
	\end{lemma}
	
	\begin{proof} We have
		\begin{align*} & \sum_{z \pmod{q}} \xi_{w}(z) \xi_{v}(z) \\
			& = \sum_{z \pmod{q}} \left(\frac{w_{1} z_1 + w_{2} z_2}{g(w_1, w_2) } \right) \left(\frac{v_{1} z_1 + v_{2} z_2}{g(v_1, v_2) } \right) \\
			& = \sum_{z \pmod{q}} \left(\frac{(w_{1} z_1 + w_{2} z_2)(v_{1} z_1 + v_{2} z_2)}{d } \right)  \left(\frac{w_{1} z_1 + w_{2} z_2}{g(w_1,w_2)/d } \right) \left(\frac{v_{1} z_1 + v_{2} z_2}{g(v_1, v_2)/d } \right).
		\end{align*} 
		From here we see that the final sum is zero unless each of the summands is equal to $1$ or $0$ identically. This is only the case when $d, g(w_1,w_2)/d, g(w_1,w_2)/d$ are all squares. Since $d | \gcd(g(w_1,w_2), g(v_1,v_2))$ and $d \nmid \Delta(f)$ it follows that $w_{1} x + w_{2} y, v_{1} x + v_{2} y$ are not proportional modulo $d$. From here we see that, modulo $d$, the number of solutions to $\gcd(w_{1} x + w_{2} y, d) = \gcd(v_{1} x + v_{2} y, d) = 1$ is equal to $\varphi(d)^2$. Similarly, modulo $g(w_1,w_2)/d$ and $g(v_1,v_2)/d$ there are $\frac{g(w_1,w_2)\varphi(g(w_1,w_2)/d)}{d}$ solutions to $\gcd(w_{1} x + w_{2} y, g(w_1,w_2)/d) = 1$ and $\gcd(v_{1} x + v_{2} y, g(v_1,v_2)/d) = 1$ respectively. Lifting to the modulus $q$ yields 
		\[\frac{q^2}{d^2} \cdot \varphi(d)^2 \cdot \frac{g(w_1,w_2) g(v_1, v_2)}{d^2} \varphi(g(w_1, w_2)/d) \varphi(g(v_1,v_2)/d) = q \varphi(d) \varphi(q/d), \]
		since $\gcd(q/d^2, d) = 1$. This completes the proof. 
	\end{proof} 
	
	Lemma \ref{xi cong sum 1} is analogous to Lemma 21.1 in \cite{FI1}. \\
	
	We now prove the following analogue of Lemma 21.2 in \cite{FI1}: 
	
	\begin{proposition} \label{QMN bd} Let $A,B$ be classes of ideal numbers with $A \cdot B = \Cl[f]$. Put 
		\begin{equation} \label{QMN def} \Q(M,N) = \sideset{}{^\ast} \sum_w \sum_z \alpha_w \beta_z \xi_w(z),
		\end{equation}
		where $\alpha_w, \beta_z$ are bounded real coefficients supported in appropriate fundamental domains for $A,B$ having norm bounded by $M,N$ respectively. Then for all $\ep > 0$ we have
		\begin{equation} \Q(M,N) \ll_\ep (M+N)^{\frac{1}{12}} (MN)^{\frac{11}{12} + \ep}.
		\end{equation}
	\end{proposition} 
	
	\begin{proof} Applying Cauchy's inequality we obtain
		\begin{align*} \lvert \Q(M,N) \rvert^2 & \leq \lVert \beta \rVert_2^2 \sum_z \left \lvert  \sideset{}{^\ast} \sum_w \alpha_w \xi_w(z) \right \rvert^2 \\
			& = \lVert \beta \rVert_2^2 \sideset{}{^\ast} \sum_{w_1} \sideset{}{^\ast} \sum_{w_2} \alpha_{w_1} \alpha_{w_2}  \sum_z \xi_{w_1}(z) \xi_{w_2}(z).
		\end{align*} 
		We then find that splitting $z$ into congruence classes modulo $q = g(w_1) g(w_2)$ that
		\[\sum_z \xi_{w_1}(z) \xi_{w_2}(z) = \sum_{\zeta \pmod{q}} \xi_{w_1}(\zeta) \xi_{w_2}(\zeta) \cdot \left(\frac{c_f N}{q^2} + O_f \left(\frac{\sqrt{N}}{q} + 1 \right) \right)\]
		where 
		\[c_f = \lim_{s \rightarrow 1} (s - 1) \zeta_K(s).\]
		We obtain, by Lemma \ref{norm char} and using (\ref{symm}) if necessary,
		\begin{equation} \label{QMN2} \Q(M,N)^2 \ll N^2 \mathop{\sum \sum}_{\substack{m_1, m_2 \leq M \\ m_1 m_2 = \square }} \tau(m_1 m_2) + NM^4 \left(\sqrt{N} + M^2\right), \end{equation}
		which gives the bound
		\[\Q(M,N) \ll_\ep \left(M^3 N^{\frac{1}{2}} + M^2 N^{\frac{3}{4}} + M^{\frac{1}{2}} N  \right) (MN)^\ep. \]
		In the next step we shall apply H\"{o}lder's inequality to obtain
		\[\Q(M,N)^k \ll M^{k-1} \sideset{}{^\ast} \sum_w \left \lvert \sum_z \beta_z \xi_w(z) \right \rvert^k = M^{k-1} \widetilde{\Q} \left(M, N^k \right), \]
		say. In \cite{FI1} the next step is to argue that $\widetilde{\Q}(M,N^k)$ can be written as a bilinear form of the shape (\ref{QMN def}), using the fact that in the case $K = \bQ(i)$ that $\xi_{w}(z)$ is multiplicative in $z$. In general this is not the case. However, we are free to choose a basis for the class $B^k$ for each positive integer $k$ one can write 
		\begin{equation} \label{multi law} \xi_{w}(z_1) \cdots \xi_{w}(z_k) = \xi_w^{(k)}(z_1 \cdots z_k) 
		\end{equation} 
		in a consistent way. Recall (\ref{jackub}), we note that
		\[\xi_w(z_1) \xi_w(z_2) = \left(\frac{Q_{B,B}(z_1) Q_{B,B}(z_2)}{g(w_1,w_2)} \right).\]
		The numerator is a bilinear form in $z_1, z_2$. Using composition laws to write 
		\[z_1 z_2 = R_{B^2}(z_1,z_2) \gamma_1^{(2)} + Q_{B^2}(z_1, z_2) \gamma_2^{(2)}\] 
		as ideal numbers, we see that we can apply a change of variables, depending only on $w$, the class $B$, and the choice of bases, so that the numerator $Q_{B,B}(z_1) Q_{B,B}(z_2)$ as a linear form in $R_{B^2}(z_1, z_2), Q_{B^2}(z_1, z_2)$. Inductively, we then find that 
		\[\xi_w(z_1) \cdots \xi_w(z_k) = \left(\frac{L_w(z_1 \cdots z_k)}{g(w_1,w_2)} \right)\]
		where $L_w$ is a linear form in two variables with coefficients depending at most on $w$ and evaluates $z_1 \cdots z_k$ in terms of its representation as an element in the lattice of the corresponding ideal numbers. Defining the right hand side as $\xi_w^{(k)}(z_1 \cdots z_k)$ 
		we obtain (\ref{multi law}). Replacing $\xi_w(\cdot)$ with $\xi_w^{(k)}(\cdot)$ in (\ref{QMN def}) shows that (\ref{QMN2}) holds, and therefore we may proceed as in \cite{FI1} after applying H\"{o}lder's inequality to conclude
		\[\Q(M,N)^k \ll_\ep M^{k-1} \left\{M^3 N^{\frac{k}{2}} + M^2 N^{\frac{3k}{4}} + M^{\frac{1}{2}} N^k \right\} (MN)^\ep, \]
		which upon taking $k$-th roots gives us the bound 
		\[\Q(M,N) \ll_\ep \left\{M^{1 + \frac{2}{k}} N^{\frac{1}{2}} + M^{1 + \frac{1}{k}} N^{\frac{3}{4}} + M^{1 - \frac{1}{2k}} N \right\} (MN)^\ep\] 
		for all positive $k \in \bN$. Switching the roles of $M,N$ and applying Lemma \ref{xi cong sum 1}, we obtain as in \cite{FI1} that 
		\[\Q(M,N) \ll_\ep (M+N)^{\frac{1}{12}} (MN)^{\frac{11}{12} + \ep}\]
		upon setting $k = 6$. 
	\end{proof}

	Next we move on to proving the analogue of Proposition 22.1 in \cite{FI1}. We define, for any ideal number $z$, a rational integer $k$, and a character $\chi$ modulo $4d$ the Hecke character
	\begin{equation} \label{Hecke char} \psi(z) = \chi(z) \left(\frac{z}{|z|} \right)^k. \end{equation}
	Consider the sum
	\[\K(N) = \sideset{}{^\wedge} \sum_{z \in \fB} \psi(z) [wz] \]
	and
	\[\K^\ast(N) = \sideset{}{^\wedge} \sum_{\substack{z \in \fB \\ \gcd(z,w) = 1}} \psi(z) [wz], \]
	where $\fB$ is narrow sector contained in the intersection of a fundamental domain for the ideal class numbers containing $z$ having norm bounded $N$. We treat $w$ as a fixed primitive ideal number. Our analogue of Proposition 22.1 in \cite{FI1} is thus:
	
	\begin{proposition} \label{KN bd} Given $\psi$ and $w$ as above we have
		\begin{equation} \label{KN1} \K(N) \ll d(|k| + 1) |w| N^{\frac{3}{4}} \log (|w| N) \end{equation}
		and 
		\begin{equation} \label{KN2} \K^\ast(N) \ll d(|k| + 1) |w| \tau(N(w)) N^{\frac{3}{4}} \log (|w| N). \end{equation}
	\end{proposition} 
	
	\begin{proof} Just like the proof of Proposition 22.1 in \cite{FI1}, the key result needed to obtain the necessary cancellation is the Polya-Vinogradov theorem, which asserts that
		\[\sum_{n \leq N} \chi(n) \ll \sqrt{q} \log q\]
		for every non-trivial Dirichlet character $\chi \pmod{q}$ with an absolute implied constant. To estimate $\K(N)$ we apply Lemma \ref{multi ident} to obtain
		\[\K(N) = [w] \sideset{}{^\wedge} \sum_{z \in \fB} \ep(w,z) \psi(z) [z] \xi_w(z), \]
		and by breaking the sum up to finitely many congruence classes if necessary, we may factor the $\ep$-factor out (because it will be constant) to obtain 
		\[\K(N) = [w] \ep \sideset{}{^\wedge} \sum_{z \in \fB} \psi(z) [z] \xi_w(z).\]
		Breaking the sum up into a double sum over rational integers forming vectors running over $\fB$ as in \cite{FI1} and applying Polya-Vinogradov we obtain (\ref{KN1}) and (\ref{KN2}) as required. 
	\end{proof}
	
	Put 
	\[\lambda(n) = \sideset{}{^\wedge} \sum_{N(z) = n} \psi(z) [z],\]
	the sum restricted to a fundamental domain of ideal numbers so each ideal is represented at most once. Consider the sum
	\begin{equation} \label{LMN def} \L(M,N) = \sum_m \sum_n \alpha(m) \beta(n) \lambda(cmn)
	\end{equation}
	where $\alpha, \beta$ are complex coefficients having norm at most $1$ and supported on $1 \leq m \leq M$ and $n \leq N$. Like wise, let $\L^\ast(M,N)$ be the subsum of (\ref{LMN def}) restricted to $\gcd(m,n) = 1$. Combining Proposition \ref{QMN bd} and Lemma \ref{multi lem} then gives the following analogue of Proposition 23.1 in \cite{FI1}: 
	
	\begin{proposition} \label{LMN bd} For any complex coefficients $\alpha(m), \beta(n)$ as above and for any positive integer $c$ we have
		\begin{equation} \label{LMN bd eq} \L(M,N) \ll \tau(c) (M+N)^{\frac{1}{12}} (MN)^{\frac{11}{12} + \ep}.
		\end{equation}
	\end{proposition}
	
	We also introduce the analogues of $\K(N), \K^\ast(N)$: 
	\begin{equation} \label{LN def} \L(N) = \sum_{n \leq N} \lambda(mn), \L^\ast(N) = \sum_{\substack{n \leq N \\ \gcd(m,n) = 1}} \lambda(mn)
	\end{equation}
	and obtain the following analogue of Proposition 23.2 in \cite{FI1} by applying Proposition \ref{KN bd}:
	\begin{proposition} \label{LN bd} For $\psi$ as defined by (\ref{Hecke char}) and positive integer $m$ we have the bounds
		\begin{equation} \label{LN bd eq1} \L(N) \ll d(|k| + 1) \tau(m)^4 \sqrt{m} N^{\frac{3}{4}} \log (mN) 
		\end{equation}
		and
		\begin{equation} \label{LN bd eq2} \L^\ast(N) \ll d(|k| + 1) \tau(m)^2 \sqrt{m} N^{\frac{3}{4}} \log (mN). 
		\end{equation}
	\end{proposition}
	
	These estimates then imply the following analogue of Theorem $\psi$ in \cite{FI1}: 
	
	\begin{proposition} \label{prop psi} For any $c \geq 1$ we have
		\begin{equation} \sum_{n \leq X} \Lambda(n) \lambda(cn) \ll c d(|k| + 1) X^{\frac{76}{77}}
		\end{equation}
		with the absolute constant dependent only on $f$. 
	\end{proposition} 
	
	\begin{proof} This is the same as the proof of Theorem $\psi$ in \cite{FI1} with Propositions 23.1 and 23.2 replaced by Propositions \ref{LMN bd} and \ref{LN bd} respectively. 
	\end{proof}

\end{document}